\documentclass[a4paper]{amsart}
\pdfoutput=1 

\usepackage{amssymb}

\usepackage{tikz}

\usepackage[backend = biber,
    style = authoryear,
    giveninits = true,  
    dashed = false,     
    uniquename = false,  
    maxcitenames = 2  
]{biblatex}
\DeclareCiteCommand{\cite}[\mkbibemph]
  {\usebibmacro{prenote}}%
  {\usebibmacro{citeindex}%
   \usebibmacro{cite}}
  {\multicitedelim}
  {\usebibmacro{postnote}}

\AtBeginBibliography{\small}  

\usepackage{tabularx,ragged2e} 
\usepackage{hyperref}  
\usepackage{graphicx}  
\usepackage[font=footnotesize]{caption}   %
\usepackage{mathtools} 
\usepackage{longtable} 
\newtheorem{thm}{Theorem}

\newcommand{\stirlingsecond}[2]{%
  \genfrac{\{}{\}}{0pt}{}{#1}{#2}}
\DeclareMathOperator{\atantwo}{atan2}
\DeclareMathOperator{\sgn}{sgn}
\DeclareMathOperator{\prt}{prt}

\title{Point-to-ellipse series -- Part II\\ Expansion about the center}
\author{\vspace{-2mm}John-Olof Nilsson\vspace{-2mm}}
\address{John-Olof Nilsson is with Avioniq Sweden AB, e-mail: john\_nil@hotmail.com}
\date{}

\begin{document}

\begin{abstract}
Power series expansions in the polar distance about the center are derived for the normal direction and signed distance from an ellipse to points near its center. After extracting constant terms and common factors, the series coefficients are polynomials in the axis ratio and the absolute value of the sine of the polar angle, with rational coefficients. The region of convergence is shown to be the region bounded by the ellipse evolute. The expansions complement previously derived expansions valid for points far from the center.
\end{abstract}

\maketitle

\vspace{-6mm}

\section{Introduction}

\begin{center}
\includegraphics[]{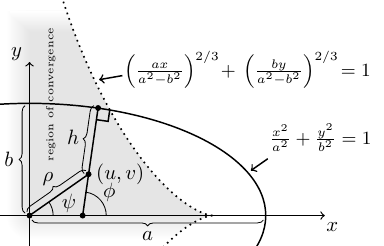}
\end{center}

As illustrated above, let $(u,v)$ be the Cartesian coordinates of a point. Let $a$ and $b$, with $0<b<a$, denote the semi-major and semi-minor axes of an ellipse centered at the origin and aligned with the coordinate axes.
Let $h$ denote the shortest signed distance from the ellipse to $(u,v)$, measured along the outward normal passing through $(u,v)$,
 with $h>0$ outside and $h<0$ inside,
and let $\phi\in(-\pi,\pi]$ denote the angle that the normal makes with the positive $x$-axis. 
Further, let $\rho=(u^2+v^2)^{1/2}$ and $\psi=\atantwo(v,u)\in(-\pi,\pi]$ be the polar coordinates of $(u,v)$.

Series of $\phi$, $h$, and the normal $(\cos(\phi),\sin(\phi))$ in powers of $\rho^{-1}$ and $\sin(\psi)$, valid for points sufficiently far from the ellipse center, were derived in:
\vspace{2mm}

\fullcite{Nilsson2025}
\vspace{2mm}

\noindent
See this Part I for an introduction to the subject and a discussion about the utility of such series. In this Part II, analogous series in $\rho$ are derived for points near the ellipse center. More specifically, the region of convergence is shown to be the region bounded by the evolute, except along the coordinate axes.
To the author's knowledge, both the present series and the study of point-to-ellipse expansions about the ellipse center are new.

The main results, i.e. the series and the region of convergence, are stated in the next Section~\ref{sec:the_series} and discussed in Section~\ref{sec:discussion}. Series and region of convergence derivations and tabulated coefficients are found in Appendices~\ref{app:A}--\ref{sec::tabulated_coefficients} and implementations are found at \url{https://github.com/jnil02/point2ellipse_series}.

\section{The series}\label{sec:the_series}
Consider the geometric setup defined in the introduction. On the major-axis segment where the nearest point is nonunique, select the upper nearest point, and at $\rho=0$ let $\psi=0$. Let $\sgn(\cdot)$ denote the sign function with the convention $\sgn(0)=1$, and $\prt(k)=k\bmod 2$ the parity function.
Further, define the eccentricity $e = (1-b^2/a^2)^{1/2}$, the axis ratio $\epsilon = (1-e^2)^{1/2}=b/a$, the polar radius of the evolute in the direction $\psi$
\begin{equation*}
\rho_c= \frac{ae^2}
{\left(|\cos(\psi)|^{2/3}+\epsilon^{2/3}|\sin(\psi)|^{2/3}\right)^{3/2}}
\end{equation*}
and the polar distance scaled with the evolute major axis extent $\sigma=\rho/(ae^2)$. Finally, at $\psi=\pm\pi/2$, the normalized angular quantities containing $1/\cos(\psi)$ are defined by their limits.

\begin{thm}\label{theorem}
The point-to-ellipse quantities $\phi$, $h$, $\cos(\phi)$, and $\sin(\phi)$ admit the following series expansions in $\sigma$, for each fixed $\psi$, with constant terms and common factors extracted from the series,
\begin{align*}
\frac{\phi-\sgn(\psi)\pi/2}{\sgn(\psi)\cos(\psi)}
\!&=\!
\sum_{k=1}^{\infty}\!
\Bigg(\!
\sum_{l=0}^{\left\lfloor\mkern-2mu\frac{k-1}{2}\mkern-2mu\right\rfloor}\!\!\mkern-1mu
\Bigg(\sum_{n=0}^{\left\lfloor\mkern-2mu\frac{k-1}{2}\mkern-2mu\right\rfloor}\!\!d^\phi_{k,l,n}\epsilon^{2n}\!
\Bigg)
|\sin(\psi)|^{2l}\!
\Bigg)\epsilon^{1\mkern-1mu+\mkern-1mu\prt(k\mkern-1mu+\mkern-1mu1)}
|\sin(\psi)|^{\prt(k\mkern-1mu+\mkern-1mu1)}
\sigma^k\\
\frac{h\!\mkern-1mu+\mkern-1mu\!b\!-\!\mkern-1mu\rho|\sin(\psi)|}{a}
\!&=\!
\sum_{k=2}^{\infty}\!
\Bigg(\!
\sum_{l=0}^{\lfloor k/2\rfloor}
\Bigg(
\sum_{n=0}^{\lfloor k/2\rfloor}
d^{h}_{k,l,n}
\epsilon^{2n}\!
\Bigg)
|\sin(\psi)|^{2l}\!
\Bigg)
\epsilon^{1+\prt(k)}
|\sin(\psi)|^{\prt(k)}
\sigma^k\\
\frac{\cos(\phi)}{\cos(\psi)}
\!&=\!
\sum_{k=1}^{\infty}\!
\Bigg(\!\mkern-2mu
\sum_{l=0}^{\left\lfloor\mkern-2mu\frac{k-1}{2}\mkern-2mu\right\rfloor}
\!\Bigg(\hspace{-8mm}
\sum_{\hspace{8.5mm}n=\prt(k+1)}^{\lfloor k/2\rfloor}
\hspace{-8mm}d^{\cos}_{k,l,n}
\epsilon^{2n}\!
\Bigg)
|\sin(\psi)|^{2l}\!
\Bigg)
\epsilon^{\prt(k)}
|\sin(\psi)|^{\prt(k+1)}
\sigma^k\\
\frac{\sin(\phi)\!-\mkern-2mu\sgn(\psi)}{\sgn(\psi)}
\!&=\!
\sum_{k=2}^{\infty}\!
\Bigg(\!
\sum_{l=0}^{\lfloor k/2\rfloor}
\Bigg(
\sum_{n=1}^{\lfloor k/2\rfloor}
d^{\sin}_{k,l,n}
\epsilon^{2n}\!
\Bigg)
|\sin(\psi)|^{2l}\!
\Bigg)
\epsilon^{\prt(k)}
|\sin(\psi)|^{\prt(k)}
\sigma^k
\end{align*}
The coefficients $d^{*}_{k,l,n}$ are rational and are given, respectively, by \eqref{eq:d_phi}, \eqref{eq:d_h}, \eqref{eq:d_cos} and \eqref{eq:d_sin}. The series converge absolutely and represent the corresponding closest-point quantities, with the selection convention above on the major axis, for
\begin{equation*}
\rho\leq\rho_c
\end{equation*}
except at $\psi=\pm\pi/2$, $\rho=\rho_c$, where the two series normalized by $\cos(\psi)$ diverge. The series for $\phi$ and $\cos(\phi)$ themselves, obtained by restoring the angular factors term by term and the constant term where applicable, converge absolutely and represent the closest-point quantities throughout $\rho\leq\rho_c$, including these two points.
\end{thm}
\begin{proof}
The derivation is straightforward but lengthy and is therefore given in Appendices~\ref{app:A}--\ref{sec:region_of_convergence}.
The main steps of the derivation are as follows:
\begin{enumerate}
\item In the open first quadrant, from suitably transformed basic geometric relations and the \emph{Lagrange-B\"{u}rmann formula}, series expansions are derived for $\phi - \tfrac{\pi}{2}$. The full $\phi$ series then follows by symmetry. 
\item From the series expansions of $\phi - \tfrac{\pi}{2}$, series expansions for powers $(\phi - \tfrac{\pi}{2})^{i}$ and power series of the form $\sum_iA_i(\phi - \tfrac{\pi}{2})^{2i}$ are derived.
\item The expansions of the power series are, in turn, used to derive series expansions for $\sin(\phi)$, $\cos(\phi)$, and $\sin(\psi)/\sin(\phi)$.
\item A lengthy derivation of series expansions for $\epsilon^2(1-e^2\sin^2(\phi))^{-1/2}$ follows.
\item From the derived series, the series expansion for $h$ is directly obtained.
\item The convergence radius of the series for $\phi-\pi/2$ is established by locating the singularities of the selected inverse branch. Boundary convergence and the coordinate-axis cases are then treated separately.
\item Finally, the convergence regions of the remaining series are obtained by examining additional singularities and cancellations under the relevant transformations. Closest-point validity beyond the evolute is checked separately.
\end{enumerate}
\end{proof}
\noindent
Note, along the coordinate axes, some series converge beyond the evolute. On the major axis, these extended sums generally do not represent the closest-point quantities.
On the minor axis, the series for $\phi$, $\cos(\phi)$ and $\sin(\phi)$ reduce to constants and, together with the affine series for $h$, remain valid for all $\rho$. These distinctions are detailed in Appendix~\ref{sec:validity_beyond}.
Further, at $\rho=0$, the conventions $\psi=0$ and $\sgn(0)=1$ select the upper nearest point, and the expansions give $\phi=\pi/2$, $h=-b$, $\cos(\phi)=0$, and $\sin(\phi)=1$. For the remaining points on the major-axis singular segment, i.e. points strictly inside the evolute for which $\psi=0$ or $\psi=\pi$, there are two nearest points. The conventions $\psi\in(-\pi,\pi]$ and $\sgn(0)=1$ select the upper one. The selected nearest-point projection is discontinuous across this segment.
Furthermore, off the coordinate axes, the limiting singularity lies at the reflection of the point through the center, on a non-closest branch. Hence, the closest-point quantities themselves remain analytic across the evolute, and the series terms eventually alternate in sign at physical points near the boundary.

\section{Discussion}\label{sec:discussion}
The presented series provide new analytic representations of the point-to-ellipse relation. The results may be interpreted from several complementary viewpoints:
\begin{itemize}
\item Analytic representations of the nearest-point projection onto an ellipse for points inside the evolute, naturally apart from the singular segment.
\item Series expansions of the relevant root of the quartic point-to-ellipse equation, see the inside-evolute case of~\cite{Vermeille2011}. 
\item Inverse-coordinate transformations linking Cartesian coordinates to geodetic and n-vector coordinates, with $\phi$ corresponding to latitude, $h$ to altitude and $(\cos(\phi),\sin(\phi))$ to the n-vector, respectively.
\end{itemize}
To the author's knowledge, neither these series nor a systematic analysis of expansions about the ellipse center have been previously reported. There are only a handful of publications on series expansions of the point-to-ellipse relation~\cite{
Morrison1961, 
Pick1967, 
Getchell1972, 
Long1975, 
Turner2013, 
Uteshev2018, 
Nilsson2025,
}. (In contrast, the total number of publications on the point-to-ellipse relation numbers in the 100s~\cite{Nilsson2024}.) All series reported in these references demonstrate a singularity at $\rho=0$ and none provides an explicit region of convergence. Except for those previously reported by the author in~\cite{Nilsson2025}, none provides formulas for arbitrary-order coefficients. 

The present series are complementary to those in~\cite{Nilsson2025} in that they have similar structure but different regions of convergence. The series' most immediate strengths are their clear symbolic structure, termwise differentiability, 
and the availability of low-order truncations yielding closed-form approximations.
For highly eccentric ellipses the convergence region is substantial, whereas in the near-circular limit the evolute contracts and the practical relevance for an interior expansion correspondingly diminishes. Thus, the principal value of the present formulation lies in its use for structural analysis and theoretical completeness, together with possible specialized applications requiring smooth local representations inside the evolute.

The series derivations as well as the coefficient formulas are lengthy and a natural open question is whether significantly shorter derivations and formulas exist. Despite cumbersome formulas, the coefficients of the final series for $h$ do not appear more complicated than those of the series leading up to it, see Section~\ref{sec::tabulated_coefficients}. Hence, one could expect that there are simpler formulas and derivations, possibly starting with a series expansion for $h$.

A wider open question is whether analogous expansions exist for the other conics. For parabolas and hyperbolas, the nearest-point relation can also be reduced to cubic or quartic equations, and local implicit expansions are available in principle. However, no counterparts with explicit coefficient structure and geometrically characterized convergence domains appear to be known.



\subsection*{Use of large language models}
Commercial large language models were used during the development of this work, from autumn 2025 to autumn 2026. They assisted in refining and shortening the derivations and contributed two key steps: the moment identification and subsequent factoring in Section~\ref{sec:N_series} used to derive the series for $\epsilon^2N/a$, and the introduction of the polynomial $P(t,x)$ in Section~\ref{sec:singular_points_f} which enabled the use of the argument principle in excluding additional convergence restrictions from the branch points of the inverse tangent.

\section*{References}
\printbibliography[heading=none]

\appendix
\newrefsegment 

\section{Series derivations}\label{app:A}
The derivation starts from the standard normal-coordinate relations for an ellipse
\begin{align}
\label{eq:base1}
\rho\cos(\psi) &= (N + h)\cos(\phi)\\
\label{eq:base2}
\rho\sin(\psi) &= (\epsilon^2 N + h)\sin(\phi)
\end{align}
where $N = a(1-e^2\sin^2(\phi))^{-1/2}$ is the so-called prime vertical radius of curvature.
Further, let $f$ and $g$ be analytic at 0, with $g(0)\neq0$, and let $y(x)$ be the unique analytic solution near $x=0$ of
\begin{equation*}
y(x)=xg(y(x))\quad\text{where}\quad y(0)=0
\end{equation*}
Together with \eqref{eq:base1} and~\eqref{eq:base2}, the following form of the \emph{Lagrange--B\"{u}rmann formula} constitutes the basis of the derivation
\begin{equation}\label{eq:lagrange-burmann}
f(y(x))=f(0) + \sum_{k=1}^\infty\frac{1}{k}\Big([y^{k-1}]f'(y)g(y)^k\Big) x^k
\end{equation}
for sufficiently small $|x|$ where $[y^k]F(y)$ denotes the coefficient of $y^k$ in the Taylor expansion of $F(y)$ about $y=0$.

\subsection{Conventions}
Throughout the derivations
\begin{itemize}
\item The derivations are carried out for $\psi\in(0,\pi/2)$. Symmetry extends the expressions to the open quadrants, continuous angular limits determine the endpoint expressions at $\psi=\pm\pi/2$, and the branch conventions determine the values at $\psi=0$ and $\psi=\pi$. Convergence at these endpoints
is examined separately in Section~\ref{sec:region_of_convergence}.
\item Sparse series expansions, in which every other coefficient vanishes, are used throughout the derivations and have coefficients denoted by $c^{*}_{k,l,n}$. The associated parity constraints are tracked explicitly, and each sparse expansion is reindexed at the end as a dense expansion with coefficients denoted by $d^{*}_{k,l,n}$.
\item For brevity, all congruences between integer indices are taken modulo $2$, unless otherwise stated. In particular, $k\equiv l$ means $k\equiv l \pmod 2$.
\item Generalized binomial coefficients are used throughout and defined by
\begin{equation*}
\binom{\alpha}{m}
=\frac{\alpha(\alpha-1)\cdots(\alpha-m+1)}{m!},
\qquad m\geq1,
\qquad
\binom{\alpha}{0}=1
\end{equation*}
and for integer $m<0$ they are $\binom{\alpha}{m}=0$. In particular, $\binom{n}{m}=0$ for integers $0\leq n<m$, whereas $\binom{-1}{0}=1$, which is used repeatedly below.
\end{itemize}

\subsection{Series expansions for \texorpdfstring{$\phi$}{φ}}\label{app:phi}
Eliminating $h$ from \eqref{eq:base2} and \eqref{eq:base1} and rearranging gives
\begin{equation*}
\frac{1}{\tan(\phi) - \tan(\psi)} = \frac{\rho\cos(\psi)}{ae^2}\frac{(1 - e^2\sin^2(\phi))^{1/2}}{\sin(\phi)}
\end{equation*}
Let
\begin{equation*}
x=\frac{\rho\cos(\psi)}{ae^2}=\sigma\cos(\psi), \quad y=\frac{1}{\tan(\phi) - \tan(\psi)}
\end{equation*}
\begin{equation*}
\text{(for brevity)}\quad z = \tan(\psi)
\end{equation*}
\begin{equation*}
g(y)=\frac{(1 + \epsilon^2(z+1/y)^2)^{1/2}}{z + 1/y}=\frac{(1 - e^2\sin^2(\phi))^{1/2}}{\sin(\phi)}\quad\text{where}\quad g(0)=\epsilon
\end{equation*}
\begin{equation*}
\text{and}\quad f(y)=\tan^{-1}\left(\frac{1}{z +(1+z^2)y}\right)=\phi-\psi + n\pi \quad n\in\mathbb{Z}
\end{equation*}
For $0<\psi<\pi/2$, choose the local analytic branch of $f$ satisfying
\begin{equation*}
f(0)=\frac{\pi}{2}-\psi
\end{equation*}
This means $f(y(x))=\phi(x)-\psi$ near $x=0$. Then, from the \emph{Lagrange-B\"{u}rmann formula}~\eqref{eq:lagrange-burmann}
\begin{equation}\label{eq:base4}
\phi-\psi = f(0) + \sum_{k=1}^\infty \frac{1}{k}\Big([y^{k-1}]f'(y)g(y)^k\Big) x^k
\end{equation}
Now $f(0)=\tan^{-1}(1/z)=\frac{\pi}{2}-\psi$. Therefore, obtaining a series expansion for $\phi$ in $\rho$ and $\psi$ reduces to deriving the coefficients of $f'(y)g(y)^k$.

Differentiating $f(y)$, applying geometric expansion and binomial expansion two times
\begin{align*}
f'(y)&=-\frac{1}{1+2zy+(1+z^2)y^2}\\
&=\sum_{r=0}^{\infty}(-1)^{r+1}\left(2zy+(1+z^2)y^2\right)^r\\
&=\sum_{r=0}^\infty(-1)^{r+1}\sum_{s=0}^r\binom{r}{s}(2zy)^{r-s}\left((1+z^2)y^2\right)^s\\
&=\sum_{r=0}^\infty(-1)^{r+1}\sum_{s=0}^r\binom{r}{s}2^{r-s}\sum_{i=0}^s\binom{s}{i}z^{2i+r-s}y^{r+s}
\end{align*}
Fixing the power $y^q$, where $q\geq 0$, meaning $q=r+s$. Since $0\leq s\leq r$, $0\leq s\leq \lfloor q/2\rfloor$
\begin{equation*}
[y^q]f'(y)=\sum_{j=0}^{\lfloor q/2\rfloor}(-1)^{q-j+1}\binom{q-j}{j}2^{q-2j}\sum_{i=0}^j\binom{j}{i}z^{2i+q-2j}
\end{equation*}
Further, $g(y)$ may be factored as
\begin{align*}
g(y)&=\frac{(1 + \epsilon^2(z+1/y)^2)^{1/2}}{z + 1/y}
=\left(\epsilon^2 + \left(\frac{y}{1+zy}\right)^2\right)^{1/2}
\end{align*}
Taking the $m^{\text{th}}$ power and applying binomial expansion twice
\begin{align*}
g(y)^m
&=\sum_{p=0}^\infty\binom{m/2}{p}\epsilon^{2(m/2-p)}\left(\frac{y}{1+zy}\right)^{2p}\\
&=\sum_{p=0}^\infty\binom{m/2}{p}\epsilon^{2(m/2-p)}\sum_{q=0}^\infty (-1)^q\binom{2p+q-1}{q}z^qy^{2p+q}
\end{align*}
Fixing the power $y^i$, $i\geq 0$, meaning $i=2p+q$, i.e. $0\leq p\leq \lfloor i/2\rfloor$, gives
\begin{equation*}
[y^i]g(y)^m=\sum_{p=0}^{\lfloor i/2\rfloor}\binom{m/2}{p}\epsilon^{2(m/2-p)}(-1)^{i-2p}\binom{i-1}{i-2p}z^{i-2p}
\end{equation*}
Now, series convolution directly gives the desired coefficient for $f'(y)g(y)^m$
\begin{align*}
[y^{m-1}]f'(y)g(y)^m&=\sum_{q=0}^{m-1}\big([y^q]f'(y)\big)\big([y^{m-1-q}]g(y)^m\big)
\end{align*}

To proceed, plug the series expansion into~\eqref{eq:base4}
\begin{align*}
\phi-\tfrac{\pi}{2}=&\sum_{m=1}^\infty\frac{x^m}{m}\sum_{q=0}^{m-1}
\sum_{j=0}^{\lfloor q/2\rfloor}\binom{q-j}{j}2^{q-2j}\sum_{i=0}^j\binom{j}{i}\sum_{p=0}^{\lfloor (m-1-q)/2\rfloor}\\
&\binom{m/2}{p}\epsilon^{m-2p}(-1)^{m-2p-j}\binom{m-1-q-1}{m-1-q-2p}\\
&z^{m-1-2p-2j+2i}\\
\end{align*}
Substituting $z=\tan(\psi)=\sin(\psi)\cos(\psi)^{-1}$ and $x=\sigma\cos(\psi)$
\begin{align*}
\phi-\tfrac{\pi}{2}=&\sum_{m=1}^\infty\frac{\sigma^m}{m}\sum_{q=0}^{m-1}
\sum_{j=0}^{\lfloor q/2\rfloor}\binom{q-j}{j}2^{q-2j}\sum_{i=0}^j\binom{j}{i}\sum_{p=0}^{\lfloor (m-1-q)/2\rfloor}\\
&\binom{m/2}{p}\epsilon^{m-2p}(-1)^{m-2p-j}\binom{m-1-q-1}{m-1-q-2p}\\
&\sin^{m-1-2p-2j+2i}(\psi)\cos^{1+2p+2j-2i}(\psi)
\end{align*}
The resulting formula also remains valid at $\psi=\pi/2$ by continuity, but for notational simplicity we divide by $\cos(\psi)$ at this stage. Further, the remaining even powers of $\cos(\psi)$ are expanded in $\sin(\psi)$ giving
\begin{align*}
\frac{\phi-\tfrac{\pi}{2}}{\cos(\psi)}
=&\sum_{m=1}^\infty\frac{\sigma^m}{m}\sum_{q=0}^{m-1}
\sum_{j=0}^{\lfloor q/2\rfloor}\binom{q-j}{j}2^{q-2j}\sum_{i=0}^j\binom{j}{i}\sum_{p=0}^{\lfloor (m-1-q)/2\rfloor}\\
&\binom{m/2}{p}\epsilon^{m-2p}\binom{m-1-q-1}{m-1-q-2p}\\
&\sum_{r=0}^{p+j-i}(-1)^{r+m-2p-j}\binom{p+j-i}{r}\sin^{m-1-2p-2j+2i+2r}(\psi)
\end{align*}
Note that, none of the factors $\sin(\psi)$, $\sigma$ or $\epsilon$ are dependent on $q$. Further, note that
\begin{itemize}
\item $j\leq \lfloor q/2\rfloor$ implies $q\geq 2j$.
\item $p\leq \lfloor (m-1-q)/2\rfloor$ implies $q\leq m-1-2p$
\item The two preceding items imply $p\leq\lfloor (m-1)/2\rfloor$ and $j\leq\lfloor (m-1-2p)/2\rfloor$.
\end{itemize}
which enables reordering the sums, separating out the $q$ dependency
\begin{align*}
\frac{\phi-\tfrac{\pi}{2}}{\cos(\psi)}&=\sum_{m=1}^\infty\frac{\sigma^m}{m}
\sum_{p=0}^{\lfloor (m-1)/2\rfloor}\binom{m/2}{p}\epsilon^{m-2p}\sum_{j=0}^{\lfloor (m-1-2p)/2\rfloor}\sum_{i=0}^j\binom{j}{i}\\
&\sum_{q=2j}^{m-1-2p}\binom{q-j}{j}2^{q-2j}
\binom{m-1-q-1}{m-1-q-2p}\\
&\sum_{r=0}^{p+j-i}(-1)^{r+m-2p-j}\binom{p+j-i}{r}\sin^{m-1-2p-2j+2i+2r}(\psi)
\end{align*}
Now collect terms in powers of $\sigma^k$ and $\sin^l(\psi)$. Note
\begin{itemize}
\item $2p+2j-2i-2r\geq 2p+2j-2i-2(p+j-i)=0$ meaning $k\geq l+1$
\item $l = m-1-2p-2j+2i+2r$ implying $l\equiv m-1$.
\item The $r$ sum is removed by $r=(l-k+1+2p+2j-2i)/2$, which is an integer since $l\equiv k+1$. The constraint $r\leq p+j-i$ is equivalent to $l\leq k-1$, and $r\geq0$ is equivalent to $i\leq(l-k+1+2p+2j)/2$, which gives the upper limit of the $i$ sum.
\end{itemize}
\begin{align*}
\frac{\phi-\tfrac{\pi}{2}}{\cos(\psi)}&=\sum_{k=1}^\infty\frac{1}{k}\sigma^k\sum_{l=0}^{k-1}\mathbf{1}_{l\equiv k+1}\sin(\psi)^l
\sum_{p=0}^{\lfloor (k-1)/2\rfloor}
\sum_{j=0}^{\lfloor (k-1-2p)/2\rfloor}\binom{k/2}{p}\epsilon^{k-2p}\\
&\sum_{q=2j}^{k-1-2p}\binom{q-j}{j}2^{q-2j}
\binom{k-1-q-1}{k-1-q-2p}
\sum_{i=0}^{\min(j,(l - k +1+2p+2j)/2)}\\&(-1)^{(l+k+1)/2-p+i}\binom{j}{i}\binom{p+j-i}{(l - k +1)/2+p+j-i}
\end{align*}
To get a clean $\epsilon$ power, make the change of index $n=k-2p$ meaning
\begin{itemize}
\item $p=(k-n)/2$
\item $\epsilon^{k-2p}=\epsilon^n$
\item $p\in[0,\lfloor (k-1)/2\rfloor]\Leftrightarrow n\in[1,k],n\equiv k$
\end{itemize}
and exploiting a variant of the alternating \emph{Vandermonde} identity
\begin{equation*}
\sum_{i=0}^{j}(-1)^i
\binom ji
\binom{A+j-i}{B-i}
=
\binom{A}{B}
\end{equation*}
(which follows from $[x^B](1+x)^{A+j}\big(1-\tfrac{x}{1+x}\big)^j=[x^B](1+x)^A$)
with $A=p,\qquad B=\frac{l-k+1}{2}+p+j$, and noting that the sum over $i$ can be extended to $j$.
Combining it all
\begin{equation}
\frac{\phi-\tfrac{\pi}{2}}{\cos(\psi)}
=
\sum^\infty_{k=1}
\sigma^k\sum_{l=0}^{k-1}\sin^l(\psi)\sum_{n=1}^k\epsilon^nc^{\phi}_{k,l,n}\label{eq:phi_sparse_expansion}
\end{equation}
where 
\begin{align*}
c^{\phi}_{k,l,n}&=
\mathbf{1}_{\substack{k\equiv l+1\\ n\equiv k}}\frac{(-1)^{(l+n+1)/2}}{k}\,\binom{\tfrac{k}{2}}{\tfrac{k-n}{2}}
\sum_{j=0}^{\left\lfloor\frac{n-1}{2}\right\rfloor}
\binom{\tfrac{k-n}{2}}{\tfrac{l - n + 1}{2}+j}
\sum_{q=2j}^{n-1}
2^{q-2j}
\binom{\scriptstyle q-j}{\scriptstyle j}
\binom{\scriptstyle k-2-q}{\scriptstyle n-1-q}
\end{align*}

To avoid the parity constraints, i.e. $k\equiv l+1$ and $n\equiv k$, change the summation indices to $l=\prt(k+1)+2l'$ and $n=1+\prt(k+1)+2n'$
giving (with dropped ${}'$),
\begin{equation*}
\frac{\phi-\tfrac{\pi}{2}}{\cos(\psi)}
=
\sum_{k=1}^{\infty}
\sigma^k
\sum_{l=0}^{\left\lfloor\frac{k-1}{2}\right\rfloor}
\sin^{\prt(k+1)+2l}(\psi)
\sum_{n=0}^{\left\lfloor\frac{k-1}{2}\right\rfloor}
\epsilon^{1+\prt(k+1)+2n}
d^\phi_{k,l,n}
\end{equation*}
where
\begin{equation}\label{eq:d_phi}
d^\phi_{k,l,n}
\!=\!
\frac{(-1)^{k+l+n}}{k}
\binom{\scriptstyle k/2}{\scriptstyle \left\lfloor\frac{k-1}{2}\right\rfloor-n}
\!\sum_{j=0}^{n}\!
\binom{\scriptstyle \left\lfloor\frac{k-1}{2}\right\rfloor-n}{\scriptstyle l-n+j}
\hspace{-9mm}
\sum_{q=2j}^{\hspace{8mm}\prt(k+1)+2n}
\hspace{-9mm}2^{q-2j}
\binom{\scriptstyle q-j}{\scriptstyle j}\!
\binom{\scriptstyle k-2-q}{\scriptstyle \prt(k+1)+2n-q}
\end{equation}

Finally, the expansion is derived under the assumption of $\psi\in(0,\pi/2)$. From the reflection symmetry around the axes, this is extended to all $\psi\neq 0$, $\psi\neq \pi$ and $\psi\neq\pm\pi/2$ by
\begin{equation}\label{eq:phi_dense_expansion}
\frac{\phi\!-\!\sgn(\psi)\pi/2}{\sgn(\psi)\cos(\psi)}
\!=\!
\sum_{k=1}^{\infty}\!\Bigg(\!
\sum_{l=0}^{\left\lfloor\mkern-2mu\frac{k-1}{2}\mkern-2mu\right\rfloor}\!\!\Bigg(\!
\sum_{n=0}^{\left\lfloor\mkern-2mu\frac{k-1}{2}\mkern-2mu\right\rfloor}
d^\phi_{k,l,n}
\epsilon^{2n}
\!\Bigg)|\sin(\psi)|^{2l}
\!\Bigg)\epsilon^{1+\prt(k\mkern-1mu+\mkern-1mu1)}|\sin(\psi)|^{\prt(k\mkern-1mu+\mkern-1mu1)}\sigma^k
\end{equation}
%
Each coefficient on the right-hand side has a finite limit at $\psi=\pm\pi/2$. These limits define the endpoint series for the continuously extended normalized angle. Their convergence is examined separately in Section~\ref{sec:region_of_convergence}.
At $\psi=0$ and $\psi=\pi$, the coefficient expressions are also regular. On the open major-axis segment where the nearest point is nonunique, the stated conventions select the upper nearest point, i.e. $\sgn(0)=1$ on the positive axis, and $\psi=\pi$ on the negative axis.

\subsection{Series expansion for \texorpdfstring{$(\phi-\pi/2)^i$}{(φ - π/2) to the power i}}
Rearranging the $l$ and $k$ sums in~\eqref{eq:phi_sparse_expansion} gives
\begin{equation*}
\phi-\tfrac{\pi}{2} = \cos(\psi)\sum_{l=0}^{\infty}\left(
\sum^\infty_{k=l+1}
\left(\sum_{n=1}^kc^{\phi}_{k,l,n}\epsilon^n\right)\sigma^k\right)\sin^l(\psi)
\end{equation*}
The expansion of its $i$th power, for $\psi\in(0,\pi/2)$, is analogous to that of $(\phi-\psi)^i$ in~\cite{Nilsson2025}.
Note, the rearrangements and products below are understood as operations on formal power series in $\sigma$. At each fixed power of $\sigma$, only finitely many terms contribute. Convergence of the resulting series is established separately.
Compared to~\cite{Nilsson2025}, the key structural properties are the same
\begin{itemize}
\item $l$ starts from 0.
\item The lowest values for $n$ and $k$ are 1.
\end{itemize}
These imply that the expansions of $(\phi-\pi/2)^i$ (and $(\phi-\psi)^i$ in \cite{Nilsson2025}) are described by the \emph{ordinary potential polynomials} and the \emph{partial ordinary Bell polynomials}.

In short, write the rearranged series as a power series in $\sin(\psi)$ with coefficients that are themselves power series in $\sigma$,
\begin{equation*}
\phi-\tfrac{\pi}{2}=\cos(\psi)\sum_{l=0}^{\infty}a_l\sin^l(\psi)
\qquad
a_l=\sum_{k=l+1}^{\infty}a_{l,k}\sigma^k
\end{equation*}
where the coefficients $a_{l,k}$ are polynomials in $\epsilon$ and given below. Raising to the power $i$ and collecting powers of $\sin(\psi)$ defines the ordinary potential polynomials
\begin{equation*}
\Bigg(\sum_{l=0}^{\infty}a_l\sin^l(\psi)\Bigg)^{i}=\sum_{l=0}^{\infty}A_{l,i}(a_0,\dots,a_l)\sin^l(\psi)
\end{equation*}
Each $A_{l,i}$ is a sum of monomials $\prod_\ell a_\ell^{m_\ell}$ (see the multinomial form below). Expanding each power $a_\ell^{m_\ell}$ in $\sigma$ defines the partial ordinary Bell polynomials
\begin{equation*}
a_\ell^{m_\ell}=\sum_{\kappa=0}^{\infty}\hat{B}_{\kappa,m_\ell}(a_{\ell,1},\dots,a_{\ell,\kappa-m_\ell+1})\,\sigma^\kappa,
\qquad \hat{B}_{0,0}=1,\quad\hat{B}_{\kappa,0}=0\ (\kappa>0)
\end{equation*}
The coefficient $c^{\phi}_{k,l,n,i}$ below is obtained by substituting these expansions into $A_{l,i}$ and extracting the coefficient of $\epsilon^n\sigma^k$.

However, in contrast to~\cite{Nilsson2025}, the differences are
\begin{itemize}
\item $k$ starts from $l+1$ rather than 1, i.e. $a_{l,k}=0$ for $k<l+1$.
\item The summation limits of $n$ are different.
\end{itemize}
The differences lead to different summation limits from those in~\cite{Nilsson2025} when collecting powers. So expanding
\begin{equation*}
(\phi-\tfrac{\pi}{2})^i = \left(\cos(\psi)\sum_{l=0}^{\infty}\left(
\sum^\infty_{k=l+1}
\left(\sum_{n=1}^kc^{\phi}_{k,l,n}\epsilon^n\right)\sigma^k\right)\sin^l(\psi)\right)^i
\end{equation*}
and temporarily referring to the resulting summation variables with $l'$, $k'$ and $n'$, the summation limits of these variables become
\begin{itemize}
\item The powers $\sin^{l'}(\psi)$ naturally start from 0 and have no upper bound.
\item For a given $i$ and power $\sin^{l'}(\psi)$, since $k$ starts at $l+1$, the lowest power $\sigma^{k'}$ is $l'+i$, and $k'$ has no upper bound.
\item Since $n$ starts from 1, the lowest power $\epsilon^{n'}$ is $i$.
\item Since $n$ is upper bounded by $k$, the highest power $\epsilon^{n'}$ is $k'$. Since $k'\geq i$, this adds no further constraint on $n'$.
\end{itemize}
Altogether, and dropping the primes on the index variables, this means that
\begin{equation}\label{eq:power_series}
(\phi-\pi/2)^i = \cos^i(\psi)\sum_{l=0}^{\infty}\left(
\sum_{k=l+i}^\infty\left(
\sum_{n=i}^kc^{\phi}_{k,l,n,i}
\epsilon^n\right)\sigma^{k}\right)\sin^l(\psi)
\end{equation}
where
\begin{equation*}
c^{\phi}_{k,l,n,i}=[\epsilon^n][\sigma^k]\left(A_ {l,i}(a_0,\dots,a_l)\left[a_\ell^{m_\ell}\rightarrow\!\!\!\sum_{\kappa=m_\ell(\ell+1)}^\infty \!\!\!\hat{B}_{\kappa,m_\ell}(a_{\ell,1},\dots,a_{\ell,\kappa-m_\ell+1})\sigma^\kappa\right]\right)
\end{equation*}
where
\begin{equation*}
a_{l,k} =
\begin{cases}
\sum_{n=1}^k
\epsilon^nc^{\phi}_{k,l,n} & k\equiv l+1 \quad\text{and}\quad k\geq l+1 \\
0 & \text{otherwise}
\end{cases}
\end{equation*}
The $k\equiv l+1$ comes from the common factor $\mathbf{1}_{k\equiv l+1}$ from $c^{\phi}_{k,l,n}$. Note the lower limit $m_\ell(\ell+1)$ (instead of $m_\ell$) of the $\kappa$ sum in $c^\phi_{k,l,n,i}$. From~\eqref{eq:parity_implications} below, $\hat{B}_{\kappa,m_\ell} = [\sigma^\kappa](a_\ell^{m_\ell}) = 0$ for $\kappa < m_\ell(\ell+1)$, so those terms vanish.

The zeros of $a_{l,k}$ provide additional structure (zeros) to $c^{\phi}_{k,l,n,i}$. To see how it propagates through $c^{\phi}_{k,l,n,i}$, it is helpful to write out $\hat{B}_{\kappa,m_\ell}$ in multinomial form
\begin{equation*}
\hat B_{\kappa,m_\ell}(a_{\ell,1},\dots,a_{\ell,\kappa-m_\ell+1})=
\sum_{\substack{\mu_1,\mu_2,\dots\ge 0\\
\sum_j \mu_j=m_\ell\\
\sum_j j\mu_j=\kappa}}
\frac{m_\ell!}{\prod_j \mu_j!}\prod_{j\ge 1} a_{\ell,j}^{\mu_j}
\end{equation*}
For $\hat{B}_{\kappa,m_\ell}\neq 0$, only indices $j$ with $j\equiv \ell+1$ and $j\ge \ell+1$ can contribute. This means $\kappa=\sum_j j\mu_j\equiv m_\ell(\ell+1)$ and $\kappa=\sum_j j\mu_j\geq m_\ell(\ell+1)$
which in turn implies the necessary condition
\begin{equation}\label{eq:parity_implications}
[\sigma^\kappa](a_\ell^{m_\ell})=0 \quad\text{unless}\quad \kappa\ge m_\ell(\ell+1)\quad\text{and}\quad\kappa\equiv m_\ell(\ell+1)
\end{equation}
Similarly, express $A_ {l,i}(a_0,\dots,a_l)$ in multinomial form
\begin{equation*}
A_ {l,i}(a_0,\dots,a_l)
=\sum_{\substack{m_0,m_1,\dots,m_l\ge 0\\
\sum_{\ell=0}^l m_\ell=i\\
\sum_{\ell=0}^l \ell m_\ell=l}}
\frac{i!}{\prod_{\ell=0}^l m_\ell!}
\prod_{\ell=0}^l a_\ell^{m_\ell}
\end{equation*}
Each term is defined by a set $\{m_\ell\}$ and $[\sigma^k](A_ {l,i}(a_0,\dots,a_l))= 0$ unless $k=\sum_{\ell=0}^l k_\ell$ for $[\sigma^{k_\ell}](a_\ell^{m_\ell})$ where each $k_\ell \geq m_\ell(\ell+1)$ and $k_\ell\equiv m_\ell(\ell+1)$.
Then
\begin{equation*}
k=\sum_{\ell=0}^l k_\ell\equiv \sum_{\ell=0}^lm_\ell(\ell+1)=l+i
\end{equation*}
meaning
\begin{equation*}
[\sigma^k](A_ {l,i}(a_0,\dots,a_l))=0\quad\text{unless}\quad k\equiv i+l
\end{equation*}
The lower bounds $k_\ell\geq m_\ell(\ell+1)$ imply a global lower bound on $k$, but since only parity is required here, these bounds are not pursued further.
Further, $c^{\phi}_{k,l,n}$ contains a factor $\mathbf{1}_{n\equiv l+1}$ meaning
\begin{equation*}
[\epsilon^{n_\ell}][\sigma^{k_\ell}](a^{m_\ell}_{\ell}) = 0\quad\text{unless}\quad k_\ell\equiv (\ell+1)m_\ell\quad\text{and}\quad n_\ell\equiv (\ell+1)m_\ell
\end{equation*}
Each term is again defined by the set $\{m_\ell,k_\ell,n_\ell\}$ and will potentially contribute to the coefficients $[\epsilon^n][\sigma^k](A_ {l,i}(a_0,\dots,a_l))$ if $k=\sum_{\ell=0}^l k_\ell$ and $n=\sum_{\ell=0}^l n_\ell$ provided that all factors $[\epsilon^{n_\ell}][\sigma^{k_\ell}](a^{m_\ell}_{\ell})$ are non-zero. This implies
\begin{equation*}
n=\sum_{\ell=0}^l n_\ell\equiv \sum_{\ell=0}^l k_\ell=k\equiv \sum_{\ell=0}^l(\ell+1)m_\ell=i+l
\end{equation*}
Combining the requirements
\begin{equation*}
[\epsilon^n][\sigma^k](A_ {l,i}(a_0,\dots,a_l))=0\quad\text{unless}\quad k\equiv n\equiv i+l
\end{equation*}
and equivalently
\begin{equation}\label{eq:parity_base}
c^{\phi}_{k,l,n,i}=0\quad\text{unless}\quad k\equiv n\equiv i+l
\end{equation}

\subsection{Series expansions for even power series in \texorpdfstring{$\phi-\pi/2$}{φ - π/2}}
Throughout this section, let $0<\psi<\pi/2$ and interpret the following identities as formal power-series identities in $\sigma$. Since $(\phi-\pi/2)^{2i}$ starts at
order $\sigma^{2i}$, each coefficient receives contributions from only finitely many powers. Using~\eqref{eq:power_series}, defining the zeroth power
\begin{equation*}
c^\phi_{k,l,n,0}
=\delta_{k,0}\delta_{l,0}\delta_{n,0}
\end{equation*}
and expanding $\cos^{2i}(\psi)=(1-\sin^2(\psi))^i$ in powers of $\sin(\psi)$ gives
\begin{align}
\nonumber\sum_{i=0}^\infty A_i(\phi-\pi/2)^{2i}&=\sum_{i=0}^\infty A_i\cos^{2i}(\psi)\sum_{l=0}^{\infty}\sin^l(\psi)\sum^\infty_{k=l+2i}
\sigma^{k}\sum^{k}_{n=2i}
\epsilon^nc^{\phi}_{k,l,n,2i}\\
\nonumber&\hspace{-8mm}=\sum_{i=0}^\infty A_i\sum_{j=0}^i\binom{i}{j}(-1)^j\sin^{2j}(\psi)\sum_{l=0}^{\infty}\sin^l(\psi)\sum_{k=l+2i}^\infty
\sigma^{k}\sum_{n=2i}^k
\epsilon^nc^{\phi}_{k,l,n,2i}\\
\label{eq:even_power_series}&\hspace{-8mm}=\sum_{l=0}^\infty\left(\sum_{k=l}^\infty\left(\sum_{n=0}^k\epsilon^nc^A_{k,l,n}\right)\sigma^k\right)\sin^l(\psi)
\end{align}
where coefficients of $\sin(\psi)$, $\sigma$ and $\epsilon$ are collected by
\begin{align*}
c^A_{k',l',n'}&=\sum_{i=0}^\infty A_i\sum_{j=0}^i\binom{i}{j}(-1)^j\sum_{l=0}^{\infty}\delta_{l'-2j-l}\sum_{k=l+2i}^\infty
\delta_{k'-k}\sum_{n=2i}^k
\delta_{n'-n}c^{\phi}_{k,l,n,2i}\\
&=\sum_{i=0}^{\lfloor n'/2\rfloor}A_i\sum_{j=\max(0,\lceil (l'+2i-k')/2\rceil)}^{\min(i,\lfloor l'/2\rfloor)}\binom{i}{j}(-1)^jc^{\phi}_{k',l'-2j,n',2i}
\end{align*}
and where the summation limits follow from the constraints:
\begin{itemize}
\item $\delta_{l'-2j-l}$, $l\geq 0$ and $j\geq 0$ imply
\begin{itemize}
\item $l'\geq 0$
\item $j\leq \lfloor l'/2\rfloor$
\item $l=l'-2j$
\end{itemize} 
\item $\delta_{k'-k}$, $k\geq l + 2i$, $l\geq 0$ and $i\geq j\geq 0$ imply
\begin{itemize}
\item $k'=k$
\item $k'\geq l'-2j+2i\geq l'$
\item $j\geq \lceil(l'+2i-k')/2\rceil$
\end{itemize}
\item $\delta_{n'-n}$, $2i\leq n\leq k$ imply
\begin{itemize}
\item $n=n'$
\item $i\leq \lfloor n'/2\rfloor$
\end{itemize}
\end{itemize}
Note that the parity constraint of \eqref{eq:parity_base} becomes particularly simple. 
\begin{equation*}
c^{\phi}_{k,l-2j,n,2i}=0\quad\text{unless}\quad k\equiv n\equiv (l-2j)+2i
\end{equation*}
but $(l-2j)+2i\equiv l$ so this simplifies to 
\begin{equation}\label{eq:parity_series}
c^{\phi}_{k,l-2j,n,2i}=0\quad\text{unless}\quad l\equiv k\equiv n
\end{equation}
This allows the triple sum to be reindexed with $k''=\prt(l)+2k$ and $n''=\prt(l)+2n$ giving (with $''$ dropped)
\begin{align*}
&\sum_{i=0}^\infty A_i(\phi-\pi/2)^{2i}=\\
&\sum_{l=0}^\infty\left(\sum_{k= \lfloor l/2\rfloor}^\infty\left(\sum_{n=0}^k\epsilon^{\prt(l)+2n}c^A_{\prt(l)+2k,l,\prt(l)+2n}\right)\sigma^{\prt(l)+2k}\right)\sin^l(\psi)
\end{align*}
in which the parity constraint is automatically satisfied.

\subsection{Series expansion for \texorpdfstring{$\sin(\phi)$ and $\cos(\phi)$}{sin(φ) and cos(φ)}}
For $0<\psi<\pi/2$, using \eqref{eq:even_power_series}
\begin{align*}
\sin(\phi)&=\cos(\phi - \pi/2)\\
&=\sum_{i=0}^\infty\frac{(-1)^i}{(2i)!}(\phi-\pi/2)^{2i}\\
&=\sum_{l=0}^\infty\left(\sum_{k=l}^\infty\left(\sum_{n=0}^k\epsilon^nc^{\sin}_{k,l,n}\right)\sigma^k\right)\sin^l(\psi)
\end{align*}
where
\begin{equation*}
c^{\sin}_{k,l,n}=\sum_{i=0}^{\lfloor n/2\rfloor}\frac{(-1)^i}{(2i)!}\sum_{j=\max(0,\lceil (l+2i-k)/2\rceil)}^{\min(i,\lfloor l/2\rfloor)}\binom{i}{j}(-1)^jc^{\phi}_{k,l-2j,n,2i}
\end{equation*}
with the parity constraint
\begin{equation*}
c^{\sin}_{k,l,n}=0\quad\text{unless}\quad l\equiv k\equiv n
\end{equation*}
Since $c^\phi_{k,l,n,0}=\delta_{k,0}\delta_{l,0}\delta_{n,0}$, the zeroth power contributes only to the constant term. Consequently, $c^{\sin}_{k,l,0}=0$ unless $k=l=0$, with $c^{\sin}_{0,0,0}=1$, and $c^{\sin}_{k,l,1}=0$.
Subtracting the leading term, rearranging the sums and reindexing with $l=2l'+\prt(k)$ and $n=2n'+\prt(k)$ meaning
\begin{equation*}
2\leq k,\quad0\leq l'\leq\lfloor k/2\rfloor\quad\text{and}\quad 1\leq n'\leq\lfloor k/2\rfloor
\end{equation*}
gives (dropping the ')
\begin{equation*}
\sin(\phi)-1
=
\sum_{k=2}^{\infty}
\left(
\sum_{l=0}^{\lfloor k/2\rfloor}
\left(
\sum_{n=1}^{\lfloor k/2\rfloor}
d^{\sin}_{k,l,n}
\epsilon^{2n+\prt(k)}
\right)
\sin^{2l+\prt(k)}(\psi)
\right)
\sigma^{k}
\end{equation*}
where
\begin{equation}\label{eq:d_sin}
d^{\sin}_{k,l,n}
=
c^{\sin}_{k,2l+\prt(k),2n+\prt(k)}
\end{equation}
From the symmetry of the $\psi$ and $\phi$ relation, for all $\psi$
\begin{equation*}
\frac{\sin(\phi)-\sgn(\psi)}{\sgn(\psi)}
=\!
\sum_{k=2}^{\infty}
\!\Bigg(
\sum_{l=0}^{\lfloor k/2\rfloor}
\!\Bigg(
\sum_{n=1}^{\lfloor k/2\rfloor}
d^{\sin}_{k,l,n}
\epsilon^{2n}
\Bigg)
|\sin(\psi)|^{2l}
\Bigg)
\epsilon^{\prt(k)}
|\sin(\psi)|^{\prt(k)}
\sigma^k
\end{equation*}

Similarly, using~\eqref{eq:power_series} with exponent $2i+1$ and writing $\cos^{2i+1}(\psi)=\cos(\psi)(1-\sin^2(\psi))^i$, one factor of $\cos(\psi)$ can be extracted, while the remaining even power is expanded in powers of $\sin(\psi)$.
\begin{align*}
\cos(\phi)&=-\sin(\phi - \pi/2)\\
&=\sum_{i=0}^\infty\frac{(-1)^{i+1}}{(2i+1)!}(\phi-\pi/2)^{2i+1}\\
&=\cos(\psi)\sum_{l=0}^\infty\left(\sum_{k=l+1}^\infty\left(\sum_{n=1}^k\epsilon^nc^{\cos}_{k,l,n}\right)\sigma^k\right)\sin^l(\psi)
\end{align*}
where
\begin{equation*}
c^{\cos}_{k,l,n}
=\sum_{i=0}^{\left\lfloor\frac{n-1}{2}\right\rfloor}\frac{(-1)^{i+1}}{(2i+1)!}
\sum_{j=\max(0,\lceil(l+(2i+1)-k)/2\rceil)}^{\min(i,\left\lfloor\frac{l}{2}\right\rfloor)}
\binom{i}{j}(-1)^j\,c^{\phi}_{k,l-2j,n,2i+1}
\end{equation*}
with the parity constraint
\begin{equation*}
c^{\cos}_{k,l,n}=0\quad\text{unless}\quad l+1\equiv  k\equiv n
\end{equation*}
Rearranging the sums and reindexing
\begin{equation*}
\frac{\cos(\phi)}{\cos(\psi)}
=
\sum_{k=1}^{\infty}
\left(
\sum_{l=0}^{\lfloor\frac{k-1}{2}\rfloor}
\left(
\sum_{n=\prt(k+1)}^{\lfloor k/2\rfloor}
d^{\cos}_{k,l,n}\,
\epsilon^{2n+\prt(k)}
\right)
\sin^{2l+\prt(k+1)}(\psi)
\right)
\sigma^k
\end{equation*}
where
\begin{equation}\label{eq:d_cos}
d^{\cos}_{k,l,n}
=
c^{\cos}_{k,2l+\prt(k+1),2n+\prt(k)}
\end{equation}
Factoring out common powers, and from the symmetry of the $\psi$ and $\phi$ relation, for all $\psi$
\begin{equation*}
\frac{\cos(\phi)}{\cos(\psi)}
=\!
\sum_{k=1}^{\infty}
\!\Bigg(
\sum_{l=0}^{\lfloor \frac{k-1}{2}\rfloor}
\!\Bigg(
\sum_{n=\prt(k+1)}^{\lfloor k/2\rfloor}
d^{\cos}_{k,l,n}
\epsilon^{2n}
\!\Bigg)
|\sin(\psi)|^{2l}
\!\Bigg)
\epsilon^{\prt(k)}
|\sin(\psi)|^{\prt(k+1)}
\sigma^k .
\end{equation*}

\subsection{Series expansion for \texorpdfstring{$\sin(\psi)/\sin(\phi)$}{sin(ψ)/sin(φ)}}
Start out with
\begin{equation*}
\frac{\sin(\psi)}{\sin(\phi)}=\frac{\sin(\psi)}{\cos(\phi-\pi/2)}
\end{equation*}
Further, for $|\phi-\pi/2|<\tfrac{\pi}{2}$, and particularly for $\phi\in(0,\pi/2)$
\begin{align*}
\frac{1}{\cos(\phi-\pi/2)}&=\sum_{i=0}^\infty\frac{E_{2i}}{(2i)!}(\phi-\pi/2)^{2i}
\end{align*}
where $E_{2i}$ are the \emph{Euler secant numbers} defined by the recurrence
\begin{equation*}
E_{2i}=\sum_{j=0}^{i-1}(-1)^{i-j+1}\binom{2i}{2j}E_{2j},\quad E_0=1
\end{equation*}
Combining the above with \eqref{eq:even_power_series} gives
\begin{align}
\frac{\sin(\psi)}{\sin(\phi)}&=\frac{\sin(\psi)}{\cos(\phi-\pi/2)}\nonumber\\
&=\sin(\psi)\sum_{i=0}^\infty\frac{E_{2i}}{(2i)!}(\phi-\pi/2)^{2i}\nonumber\\
&=\sin(\psi)\sum_{l=0}^\infty\left(\sum_{k=l}^\infty\left(\sum_{n=0}^k\epsilon^nc^{\sin^{-1}}_{k,l,n}\right)\sigma^k\right)\sin^l(\psi)\nonumber\\
&=\sum_{l=1}^\infty\left(\sum_{k=l-1}^\infty\left(\sum_{n=0}^k\epsilon^nc^{\sin^{-1}}_{k,l-1,n}\right)\sigma^k\right)\sin^l(\psi)\label{eq:sin_psi_sin_phi}
\end{align}
where
\begin{equation*}
c^{\sin^{-1}}_{k,l,n}=\sum_{i=0}^{\lfloor n/2\rfloor}\frac{E_{2i}}{(2i)!}\sum_{j=\max(0,\lceil (l+2i-k)/2\rceil)}^{\lfloor l/2\rfloor}\binom{i}{j}(-1)^jc^{\phi}_{k,l-2j,n,2i}
\end{equation*}
By \eqref{eq:parity_series}
\begin{equation*}
c^{\sin^{-1}}_{k,l,n}=0\quad\text{unless}\quad k\equiv l\equiv n
\end{equation*}
Further, for $n\in\{0,1\}$, the outer sum in $c^{\sin^{-1}}_{k,l,n}$ only contains $i=0$. Since $\binom{0}{j}=0$ for $j>0$, only the term $c^\phi_{k,l,n,0}$ contributes, and by definition $c^\phi_{k,l,n,0}=\delta_{k,0}\delta_{l,0}\delta_{n,0}$. Hence $c^{\sin^{-1}}_{k,l,0}=0$ unless $k=l=0$, with $c^{\sin^{-1}}_{0,0,0}=1$, and $c^{\sin^{-1}}_{k,l,1}=0$. This means that
\begin{equation*}
\frac{\sin(\psi)}{\sin(\phi)}=\sin(\psi)+\sum_{l=1}^\infty\left(\sum_{k=l-1}^\infty\left(\sum_{n=2}^k\epsilon^nc^{\sin^{-1}}_{k,l-1,n}\right)\sigma^k\right)\sin^l(\psi)
\end{equation*}

\subsection{Series expansion for \texorpdfstring{$\epsilon^2N/a$}{ε²N/a}}\label{sec:N_series}
The reason to include the factor $(1-e^2)/a=\epsilon^2/a$, rather than deriving a series expansion just for $N=a/\sqrt{1-e^2\sin^2(\phi)}$, is that it prevents negative powers of $\epsilon$ in the expansion.
Binomial expansion gives
\begin{align}
\nonumber\epsilon^2N/a&=\frac{\epsilon^2}{\sqrt{1-e^{2}\sin^{2}(\phi)}}\\
\label{eq:h_second_term}&=\epsilon^2\sum_{m=0}^{\infty}\binom{2m}{m}2^{-2m}e^{2m}\sin^{2m}(\phi)
\end{align}
Further
\begin{align*}
\sin^{2m}(\phi)&=\cos^{2m}(\phi-\pi/2)\\
&=2^{-2m}\binom{2m}{m}+2^{-2m}2\sum_{j=1}^{m}\binom{2m}{m-j}\cos(2j(\phi-\pi/2))\\
&=2^{-2m}\binom{2m}{m}+2^{-2m}2\sum_{j=1}^{m}\binom{2m}{m-j}\sum_{i=0}^\infty\frac{(-1)^i(2j)^{2i}}{(2i)!}(\phi-\pi/2)^{2i}\\
&=2^{-2m}\binom{2m}{m}+\sum_{i=0}^{\infty}\left(
2^{-2m+1}\frac{(-1)^i}{(2i)!}\sum_{j=1}^{m}\binom{2m}{m-j}(2j)^{2i}
\right) (\phi-\pi/2)^{2i}\\
&=\sum_{i=0}^{\infty}c_{m,i}(\phi-\pi/2)^{2i}
\end{align*}
where
\begin{equation*}
c_{m,i}=
\begin{cases}
1&i=0\\
2^{-2m+1}\frac{(-1)^i}{(2i)!}\sum_{j=1}^{m}\binom{2m}{m-j}(2j)^{2i}&i\geq1
\end{cases}
\end{equation*}
Inserting into \eqref{eq:h_second_term} and using~\eqref{eq:even_power_series}
\begin{align}
\epsilon^2N/a&=\epsilon^2\sum_{m=0}^{\infty}\binom{2m}{m}2^{-2m}e^{2m}\sin^{2m}(\phi)\nonumber\\
&=\epsilon^2\sum_{m=0}^{\infty}\binom{2m}{m}2^{-2m}e^{2m}\sum_{i=0}^{\infty}c_{m,i}(\phi-\pi/2)^{2i}\nonumber\\
&=\epsilon^2\sum_{i=0}^{\infty}\left(\sum_{m=0}^{\infty}\binom{2m}{m}2^{-2m}e^{2m}c_{m,i}\right)(\phi-\pi/2)^{2i}\nonumber\\
&=\sum_{l=0}^\infty\left(\sum_{k=l}^\infty\left(\epsilon^2\sum_{n=0}^k\epsilon^{n}c^{N'}_{k,l,n}\right)\sigma^k\right)\sin^l(\psi)\label{eq:one_over_sqrt_expansion}
\end{align}
where
\begin{equation*}
c^{N'}_{k,l,n}=\sum_{i=0}^{\lfloor n/2\rfloor}\left(\sum_{m=0}^{\infty}\binom{2m}{m}2^{-2m}e^{2m}c_{m,i}\right)\sum_{j=\max(0,\lceil(l+2i-k)/2\rceil)}^{\min(i,\lfloor l/2\rfloor)}\binom{i}{j}(-1)^jc^{\phi}_{k,l-2j,n,2i}
\end{equation*}
The interchange of the $m$- and $i$-sums is justified by absolute convergence near $u=\phi-\pi/2=0$. Indeed, the formula for $c_{m,i}$ gives
\begin{equation*}
\sum_{i=0}^{\infty}|c_{m,i}|r^{2i}
=\cosh^{2m}(r)
\end{equation*}
Since $0<e<1$, choose $r>0$ such that $e^2\cosh^2(r)<1$. Then, for $|u|\leq r$,
\begin{equation*}
\epsilon^2\sum_{m=0}^{\infty}
\binom{2m}{m}4^{-m}e^{2m}
\sum_{i=0}^{\infty}|c_{m,i}u^{2i}|
\leq
\frac{\epsilon^2}{\sqrt{1-e^2\cosh^2(r)}}
<\infty
\end{equation*}
Thus the order of summation may be exchanged.

The challenge with $c^{N'}_{k,l,n}$ is the series in $e^{2m}$. Fortunately, it turns out that this sum can be written as a finite sum in $\epsilon$. There are many ways to proceed, but a convenient approach uses two key observations:
\begin{itemize}
\item $c_{m,i}$ may be expressed in terms of random walk moments, which for a fixed $i$ are polynomials in $m$.
\item Using the polynomials in $m$ and a Cauchy--Euler operator, the infinite sum can be reduced to a finite sum in powers of $\epsilon$.
\end{itemize}
Let $\varepsilon_1,\ldots,\varepsilon_{2m}$ be independent random variables taking the values $+1$ and $-1$ with equal probability, and let $S_{2m}=\varepsilon_1+\cdots+\varepsilon_{2m}$. Then
\begin{equation*}
\mathbb P(S_{2m}=2j)=2^{-2m}\binom{2m}{m+j},\qquad j=-m,\dots,m
\end{equation*}
and, for $i\geq 1$
\begin{align*}
\mathbb{E}[S^{2i}_{2m}]&=2^{-2m}\sum_{j=-m}^{m}\binom{2m}{m+j}(2j)^{2i}\\
&=2^{-2m+1}\sum_{j=1}^{m}\binom{2m}{m-j}(2j)^{2i}
\end{align*}
Consequently, $c_{m,i}$ may be interpreted as
\begin{equation*}
c_{m,i}=\frac{(-1)^i}{(2i)!}\,\mathbb{E}[S_{2m}^{2i}]\quad i\geq 1
\end{equation*}
Expanding the moments
\begin{align*}
\mathbb{E}\left(S_{2m}^{2i}\right)&=\mathbb{E}\left(\left(\varepsilon_1+\dots+\varepsilon_{2m}\right)^{2i}\right)\\
&=\mathbb{E}\left(\sum_{\alpha_1+\cdots+\alpha_{2m}=2i}
\binom{2i}{\alpha_1,\ldots,\alpha_{2m}}
\prod_{j=1}^{2m}\varepsilon_j^{\alpha_j}\right)\\
&=\sum_{\alpha_1+\cdots+\alpha_{2m}=2i}
\binom{2i}{\alpha_1,\ldots,\alpha_{2m}}\prod_{j=1}^{2m}\mathbb E[\varepsilon_j^{\alpha_j}]
\end{align*}
Now $\mathbb{E}[\varepsilon_j^{k}] = 0$ if $k$ is odd and $\mathbb{E}[\varepsilon_j^{k}] = 1$ if $k$ is even,
i.e. all non-zero terms must have all $\alpha_j$ even. Consequently, and with a change of variable to ensure this,
\begin{align*}
\mathbb{E}[S_{2m}^{2i}]&=\sum_{\substack{\alpha_1+\cdots+\alpha_{2m}=2i\\ \text{all }\alpha_j\text{ even}}}
\frac{(2i)!}{\alpha_1!\cdots \alpha_{2m}!}\\
&=(2i)!\!\!\sum_{\beta_1+\cdots+\beta_{2m}=i}
\frac{1}{(2\beta_1)!\cdots(2\beta_{2m})!}
\end{align*}
Now, the $m$ dependency can be separated out by counting the number $k$ of $\beta_i>0$. There are $\binom{2m}{k}$ ways this can happen, i.e.
\begin{equation*}
\mathbb{E}[S_{2m}^{2i}]
=\sum_{k=1}^{i}\binom{2m}{k}(2i)!
\sum_{\substack{\beta_1+\cdots+\beta_k=i\\ \beta_j\ge 1}}
\frac{1}{(2\beta_1)!\cdots(2\beta_k)!}
\end{equation*}
The moments can be converted to a polynomial in $m$ by $\binom{2m}{k}=\frac{(2m)_{\underline{k}}}{k!}$ and $
(2m)_{\underline{k}}=\sum_{r=0}^{k} s(k,r)\,(2m)^r
=\sum_{r=0}^{k} s(k,r)2^r m^r$, i.e.
\begin{equation*}
\mathbb{E}[S_{2m}^{2i}]
=\sum_{k=1}^{i}\frac{(2i)!}{k!}\sum_{r=0}^{k} s(k,r)2^r m^r
\sum_{\substack{\beta_1+\cdots+\beta_k=i\\ \beta_j\ge 1}}
\frac{1}{(2\beta_1)!\cdots(2\beta_k)!}
\end{equation*}
where $s(k,r)$ are signed Stirling numbers of the first kind.

Writing the inner sum as a generating function and identifying the resulting series with exponentials gives
\begin{align*}
\sum_{\substack{\beta_1+\cdots+\beta_k=i\\ \beta_j\ge 1}}
\frac{1}{(2\beta_1)!\cdots(2\beta_k)!}&=[w^i]\left(\sum_{r=1}^\infty\frac{w^r}{(2r)!}\right)^k\\
&=[u^{2i}]\left(\sum_{r=1}^\infty\frac{u^{2r}}{(2r)!}\right)^k\\
&=[u^{2i}]\left(\frac{e^{u}+e^{-u}}{2}-1\right)^k\\
&=[u^{2i}]2^{-k}e^{-ku}(e^u-1)^{2k}\\
&=[u^{2i}]2^{-k}\sum_{j=0}^{2k}(-1)^j\binom{2k}{j}e^{(j-k)u}\\
\end{align*}
Trivially $[u^{2i}]e^{(j-k)u}=\frac{(j-k)^{2i}}{(2i)!}$. Reindexing with $t=j-k$ and exploiting symmetry
\begin{align*}
\sum_{\substack{\beta_1+\cdots+\beta_k=i\\ \beta_j\ge 1}}
\frac{1}{(2\beta_1)!\cdots(2\beta_k)!}&=\frac{2^{-k}}{(2i)!}\sum_{j=0}^{2k}(-1)^j\binom{2k}{j}(j-k)^{2i}\\
&=\frac{2^{1-k}}{(2i)!}\sum_{t=1}^{k}(-1)^{k-t}\binom{2k}{k-t}t^{2i}
\end{align*}
Combining it all gives
\begin{equation*}
c_{m,i}=\sum_{r=0}^{i} a_{i,r}m^r
\end{equation*}
where
\begin{equation*}
a_{i,r}=
\begin{cases}
1 & i=r=0 \\
\frac{2(-1)^i}{(2i)!}\sum_{k=r}^{i} s(k,r) \left(\frac{1}{k!2^{k-r}}\sum_{t=1}^{k}(-1)^{k-t}\binom{2k}{k-t}\,t^{2i}\right) & i\geq 1,\:0\leq r\leq i
\end{cases}
\end{equation*}

Next, designate the Cauchy-Euler operator $D=x\frac{\text{d}}{\text{d} x}$ and note that
\begin{equation*}
D^r=\left(x\frac{d}{dx}\right)^r=\sum_{k=0}^r \stirlingsecond{r}{k}x^k\frac{d^k}{dx^k}
\end{equation*}
where $\stirlingsecond{r}{k}$ are the Stirling numbers of the second kind. Now let
\begin{equation*}
F(x) = \sum_{m=0}^\infty \binom{2m}{m}\left(\frac{x}{4}\right)^m = (1-x)^{-1/2}
\end{equation*}
Then
\begin{align*}
\sum_{m=0}^\infty \binom{2m}{m}\Big(\frac{x}{4}\Big)^m c_{m,i}
&=\sum_{r=0}^i a_{i,r}D^r F(x)\\
&=\sum_{r=0}^i a_{i,r}D^r (1-x)^{-1/2}\\
&=\sum_{r=0}^i a_{i,r}\sum_{k=0}^r \stirlingsecond{r}{k}\,(\tfrac12)_k\,x^k\,(1-x)^{-1/2-k}\\
&=\sum_{r=0}^i a_{i,r}\sum_{k=0}^r \stirlingsecond{r}{k}\,(\tfrac12)_k\sum_{q=0}^k (-1)^q\binom{k}{q}(1-x)^{-1/2-k+q}\\
&=\sum_{r=0}^i a_{i,r}\sum_{k=0}^r \stirlingsecond{r}{k}\,(\tfrac12)_k\sum_{t=0}^k (-1)^{k-t}\binom{k}{t}(1-x)^{-1/2-t}\\
&=\sum_{r=0}^i a_{i,r}\sum_{t=0}^r\sum_{k=t}^r \stirlingsecond{r}{k}\,(\tfrac12)_k(-1)^{k-t}\binom{k}{t}(1-x)^{-1/2-t}\\
&=\sum_{t=0}^i C_{i,t}(1-x)^{-1/2-t}
\end{align*}
where $(\tfrac12)_k$ is the rising factorial and where
\begin{equation*}
C_{i,t}=\sum_{r=t}^i a_{i,r}\sum_{k=t}^r \stirlingsecond{r}{k}\,(\tfrac12)_k(-1)^{k-t}\binom{k}{t}
\end{equation*}
For brevity, let
\begin{equation*}
R_{k,l,n,i}=\sum_{j=\max(0,\lceil (l+2i-k)/2\rceil)}^{\min(i,\lfloor l/2\rfloor)}\binom{i}{j}(-1)^jc^{\phi}_{k,l-2j,n,2i}
\end{equation*}
with the parity constraint
\begin{equation}\label{eq:R_parity}
R_{k,l,n,i} = 0\quad\text{unless}\quad l\equiv k\equiv n
\end{equation}
Now the sum over $\epsilon$ of \eqref{eq:one_over_sqrt_expansion} becomes
\begin{align*}
\epsilon^2\sum_{n=0}^k\epsilon^nc^{N'}_{k,l,n}&=\epsilon^2\sum_{n=0}^k\epsilon^n\sum_{i=0}^{\lfloor n/2\rfloor}\left(\sum_{m=0}^{\infty}\binom{2m}{m}2^{-2m}e^{2m}c_{m,i}\right)R_{k,l,n,i}\\
&=\epsilon^2\sum_{n=0}^k\epsilon^n
\sum_{i=0}^{\lfloor n/2\rfloor}
\sum_{t=0}^{i}
C_{i,t}(1-e^2)^{-1/2-t}R_{k,l,n,i}\\
&=\sum_{n=0}^k\sum_{i=0}^{\lfloor n/2\rfloor}\sum_{t=0}^iC_{i,t}R_{k,l,n,i}\epsilon^{n+1-2t}
\end{align*}
Now, the reason the factor $(1-e^2)=\epsilon^2$ was included right away becomes obvious, i.e. it prevents negative powers of $\epsilon$.
Introduce the new $\epsilon$ exponent $n'=n+1-2t$. Since $0\leq t\leq i\leq\lfloor n/2\rfloor$, it satisfies $1\leq n'\leq k+1$. Relabeling the old exponent as $p$ and the new exponent as $n$ gives
%
\begin{equation*}
\epsilon^2\sum_{n=0}^k\epsilon^nc^{N'}_{k,l,n}=\sum_{n=1}^{k+1}c^N_{k,l,n}\epsilon^n
\end{equation*}
where
\begin{equation}
c^N_{k,l,n}=\sum_{p=0}^{k}\sum_{i=0}^{\lfloor p/2\rfloor}
\mathbf{1}_{\substack{p\equiv n+1\\0\le (p+1-n)/2\le i}}
C_{i,(p+1-n)/2}R_{k,l,p,i}
\end{equation}
where only the terms satisfying the indicator condition are evaluated.
Finally, inserting into \eqref{eq:one_over_sqrt_expansion}
\begin{equation}\label{eq:N_series}
\epsilon^2N/a=\sum_{l=0}^\infty\left(\sum_{k=l}^\infty\left(\sum_{n=1}^{k+1}c^N_{k,l,n}\epsilon^n\right)\sigma^k\right)\sin^l(\psi)
\end{equation}
Note that the parity constraint \eqref{eq:R_parity} applies to $R_{k,l,p,i}$, i.e. $l\equiv k\equiv p$.
A nonzero contribution requires $l\equiv k\equiv p$ and $p\equiv n+1$. Consequently,
%
\begin{equation*}
c^N_{k,l,n}=0\quad\text{unless}\quad l\equiv k\equiv n+1
\end{equation*}

\subsection{Series expansion for \texorpdfstring{$h$}{h}}
Solving~\eqref{eq:base2} for $h$ gives
\begin{equation*}
\frac{h}{a}=\frac{\rho}{a}\frac{\sin(\psi)}{\sin(\phi)}-\epsilon^2N/a
\end{equation*}
%
For the first term, from \eqref{eq:sin_psi_sin_phi}
\begin{align*}
\frac{\rho}{a}\frac{\sin(\psi)}{\sin(\phi)}&=\frac{\rho\sin(\psi)}{a}+\sigma(1 - \epsilon^2)\sum_{l=1}^\infty\left(\sum_{k=l-1}^\infty\left(\sum_{n=2}^k\epsilon^nc^{\sin^{-1}}_{k,l-1,n}\right)\sigma^k\right)\sin^l(\psi)\\
&=\frac{\rho\sin(\psi)}{a}+\sum_{l=1}^\infty\left(\sum_{k=l}^\infty\left(\sum_{n=1}^{k+1}\epsilon^nc^\rho_{k,l,n}\right)\sigma^{k}\right)\sin^l(\psi)
\end{align*}
where
\begin{equation*}
c^\rho_{k,l,n}=
\begin{cases}
c^{\sin^{-1}}_{k-1,l-1,n} & 2\leq n\leq \min(3, k-1)\\
c^{\sin^{-1}}_{k-1,l-1,n} - c^{\sin^{-1}}_{k-1,l-1,n-2} & 4\leq n\leq k-1\\
-c^{\sin^{-1}}_{k-1,l-1,n-2} & \max(4,k) \leq n\leq k+1\\
0 & \text{otherwise}
\end{cases}
\end{equation*}
with the parity constraint
\begin{equation*}
c^\rho_{k,l,n}=0\quad\text{unless}\quad l\equiv k\equiv n+1
\end{equation*}
Combining with \eqref{eq:N_series}, the complete series expansion of $h$ is given by
\begin{align*}
\frac{h}{a}&=\frac{\rho}{a}\frac{\sin(\psi)}{\sin(\phi)}-\epsilon^2N/a\\
&=\frac{\rho\sin(\psi)}{a}+\sum_{l=0}^\infty\left(\sum_{k=l}^\infty\left(\sum_{n=1}^{k+1}c^h_{k,l,n}\epsilon^n\right)\sigma^{k}\right)\sin^l(\psi)
\end{align*}
where
\begin{equation*}
c^h_{k,l,n}=
\begin{cases}
-c^N_{k,l,n}&l=0\\
c^\rho_{k,l,n}-c^N_{k,l,n}&\text{otherwise}
\end{cases}
\end{equation*}
with the parity constraint
\begin{equation*}
c^h_{k,l,n}=0\quad\text{unless}\quad l\equiv k\equiv n+1
\end{equation*}
Since the signed distance is invariant under reflection in either coordinate axis, the first-quadrant formula extends to all quadrants by replacing $\sin(\psi)$ with $|\sin(\psi)|$.
Further, rearranging the sums and reindexing with $l=2l'+\prt(k)$, $n=2n'+\prt(k)+1$ fulfilling the parity constraints and giving the ranges
\begin{equation*}
0\leq k,\quad 0\leq l'\leq\lfloor k/2\rfloor\quad\text{and}\quad 0\leq n'\leq\lfloor k/2\rfloor
\end{equation*}
gives (dropping ')
\begin{equation*}
\frac{h}{a}
=
\frac{\rho|\sin(\psi)|}{a}-\epsilon+
\sum_{k=2}^{\infty}
\Bigg(
\sum_{l=0}^{\lfloor k/2\rfloor}
\Bigg(
\sum_{n=0}^{\lfloor k/2\rfloor}
d^{h}_{k,l,n}
\epsilon^{2n}
\Bigg)
|\sin(\psi)|^{2l}
\Bigg)
\epsilon^{1+\prt(k)}
|\sin(\psi)|^{\prt(k)}
\sigma^k
\end{equation*}
where
\begin{equation}\label{eq:d_h}
d^{h}_{k,l,n}=
c^h_{k,2l+\prt(k),2n+\prt(k)+1}
\end{equation}
Note, the summation of $k$ from 2 is achieved by noting that the coefficient $d^h_{0,0,0}=c^h_{0,0,1}=-C_{0,0}R_{0,0,0,0}=-1$ and the second coefficient $d^{h}_{1,0,0}=c^h_{1,1,2}=c^{\rho}_{1,1,2}-c^N_{1,1,2}=0$, since $c^\rho_{1,1,2}=0$ (otherwise case) and $c^N_{1,1,2}=C_{0,0}c^\phi_{1,1,1,0}=0$. 

\section{Region of convergence}\label{sec:region_of_convergence}
All series are based on the \emph{Lagrange–B\"urmann} series~\eqref{eq:base4}. Hence, first its region of convergence is derived in the subsection after next, followed in subsequent subsections by an analysis of the additional singularities, and of singularity cancellations, introduced by the transformations of~\eqref{eq:base4} used to derive the subsequent series.

Note, the convergence is first determined for non-degenerate parameters and over the open first quadrant, i.e.
\begin{equation*}
0<e<1,\qquad 0<\epsilon=\sqrt{1-e^2}<1,\qquad 0<\psi<\pi/2,\qquad z=\tan(\psi)>0
\end{equation*}
Then the boundary $\rho=\rho_c$ and the end-point cases $\psi=0$ and $\psi=\pi/2$ are treated separately. Finally, the region is extended to all (closed) quadrants by symmetry arguments and stated axis conventions. Next, the transfer argument used to establish the convergence on the boundary is recalled.

\subsection{The transfer method}\label{sec:transfer}
The convergence of the series on its circle of convergence can be determined from the local form of the dominant singularity. The so-called transfer method is outlined here and referred to throughout the convergence proof. For details, see Chapter VI of~\cite{Flajolet2009}.

Suppose that a function $F(x)$, analytic at $x=0$, has the radius of convergence $R=|x_c|$ and that near a unique dominant singularity $x_c$, i.e. the only singularity on $|x|=R$, it has the local expansion
\begin{equation}\label{eq:darboux_structure}
F(x)
=
F_{\mathrm{reg}}(x)
+
A\left(1-\frac{x}{x_c}\right)^{\alpha}
+
O\left(\left(1-\frac{x}{x_c}\right)^{\alpha+1}\right),
\qquad
A\neq0
\end{equation}
where $F_{\mathrm{reg}}$ is analytic at $x_c$ and $\alpha\notin\mathbb{N}_0$. Suppose further that the singularities of $F$ are isolated, so that, with $x_c$ the only one on $|x|=R$, $F$ continues analytically to a disk $|x|<R+\eta$, $\eta>0$, slit along the ray from $x_c$ outwards, and that~\eqref{eq:darboux_structure} holds as $x\to x_c$ in this slit disk. Then $F$ is $\Delta$-analytic at $x_c$ in the sense of Definition VI.1 of~\cite{Flajolet2009}, and the Taylor coefficients
\begin{equation*}
F(x)=\sum_{k=0}^{\infty}a_kx^k
\end{equation*}
satisfy
\begin{equation}\label{eq:darboux_transfer}
a_k
\sim
\frac{A}{\Gamma(-\alpha)}
x_c^{-k}k^{-\alpha-1}
\end{equation}
This follows from the coefficient asymptotics of the standard scale (Theorem VI.1 of~\cite{Flajolet2009}), the transfer of the error term (Theorem VI.3 of~\cite{Flajolet2009}), and their combination in Theorem VI.4 of~\cite{Flajolet2009}, after rescaling $x\mapsto x/x_c$.
%
The regular analytic part does not contribute to the leading algebraic coefficient asymptotics.
Thus the algebraic type of the dominant singularity directly determines the decay of the Taylor coefficients.

In particular, for a square-root singularity, $\alpha=1/2$, so
\begin{equation*}
a_k=O\left(R^{-k}k^{-3/2}\right)
\end{equation*}
Consequently, on the circle $|x|=R$,
\begin{equation*}
|a_kx^k|=O(k^{-3/2})
\end{equation*}
and the series converges absolutely since $\sum k^{-3/2}$ converges. Hence a square-root branch point may determine the radius of convergence while the series itself still converges on the corresponding boundary.

Conversely, a unique dominant inverse square-root singularity, $\alpha=-1/2$, gives the coefficient asymptotic
\begin{equation*}
a_k\sim\frac{A}{\sqrt{\pi}}x_c^{-k}k^{-1/2}
\end{equation*}
Thus the series diverges at $x=x_c$ and is not absolutely convergent anywhere on $|x|=R$. At other points of the circle, conditional convergence may occur, so convergence at the particular boundary point under consideration must be checked separately.

When several dominant singularities are present, each must satisfy the corresponding $\Delta$-analyticity condition and their contributions add (Theorem VI.5 of~\cite{Flajolet2009}). In particular, if each is of square-root type, then still $a_k=O(R^{-k}k^{-3/2})$ and the series converges absolutely on $|x|=R$. 

\subsection{Region of convergence for series of \texorpdfstring{$\phi-\pi/2$}{φ - π/2}}
Recalling the components of Section~\ref{app:phi}, the series~\eqref{eq:base4} is that of
\begin{equation*}
\phi(x)=\psi+f(y_0(x)),
\qquad x=\frac{\rho\cos(\psi)}{ae^2}
\end{equation*}
with $f(y)$ and $y(x)$ depending parametrically on $\psi$, and where $y(x)$ is implicitly defined by
\begin{equation*}
y=xg(y)
\end{equation*}
where
\begin{equation*}
y=\frac{1}{\tan(\phi) - \tan(\psi)},\quad z = \tan(\psi)
\end{equation*}
\begin{equation*}
g(y)=\frac{(1 + \epsilon^2(z+1/y)^2)^{1/2}}{z + 1/y}
=\left(\epsilon^2+\left(\frac{y}{1+zy}\right)^2\right)^{1/2},\quad g(0)=\epsilon
\end{equation*}
\begin{equation*}
\text{and}\quad f(y)=\tan^{-1}\left(\frac{1}{z +(1+z^2)y}\right)
\end{equation*}
where $f(y)$ is defined to be the branch $f(0)=\pi/2-\psi$ and where $y_0(x)$ is the branch of the inverse of $x(y)=y/g(y)$ for which $y(0)=0$. It is essential to track the inverse branch $y_0(x)$ originating at $y_0(0)=0$, since singularities belonging only to other inverse branches do not constrain its Taylor series.


The region of convergence of the \emph{Lagrange–B\"urmann} series~\eqref{eq:base4} is bounded by the nearest singularity of the function $f(y_0(x))$, that is points $x_c$ in the complex $x$-plane for which $f(y_0(x))$ ceases to be analytic. This can happen at points $x_c$:
\begin{itemize}
\item where $y_0(x_c)$ is singular, provided that
the singularity is not removed by composition with $f$.
\item where $y_0(x)$ remains analytic but $y_0(x_c)$ is a singularity of $f$
\end{itemize}
In subsequent subsubsections, it is first shown that the singular points of $y_0(x)$ bound the region of convergence to the evolute. Second, it is shown that there are indeed singular points of $f(y)$ but they do not impose a smaller radius of convergence of the selected branch.
Finally, the outer limit and the end points are treated and the region of convergence is extended to all quadrants.

\subsubsection{Singular points of $y(x)$}
From the implicit function theorem, since the inverse
\begin{equation*}
x(y)=\frac{y}{g(y)}=\frac{y(1+zy)}{(y^2+\epsilon^2(1+zy)^2)^{1/2}}\quad\text{with}\quad 
\frac{\partial x}{\partial y}
=
\frac{zy^3+\epsilon^2(1+zy)^3}
{\left(y^2+\epsilon^2(1+zy)^2\right)^{3/2}}
\end{equation*}
is analytic at 0, i.e. $x(0)=0$ and $\left.\frac{\partial}{\partial y}x(y)\right|_0=1/\epsilon\neq 0$, $y_0(x)$ is analytic in a region around $x=0$ until $x(y)$ exhibits a
\begin{itemize}
\item Critical point, i.e. $\frac{\partial}{\partial y}\frac{y}{g(y)}=0$
\item Pole or singularity, i.e. $g(y)=0$ or $g(y)$ is singular.
\item A sequence $y\to\infty$ for which $x(y)$ tends to a finite value.
\end{itemize}
Clearly
\begin{equation*}
x(y)\sim
\pm\frac{z}{\sqrt{1+\epsilon^2z^2}}\,y
\qquad(y\to\infty)
\end{equation*}
Thus no sequence $y\to\infty$ gives a finite limiting value of $x(y)$.
Being singular means that the denominator of $g(y)$ is zero, i.e.
\begin{equation*}
1+zy = 0 \implies y = -1/z
\end{equation*}
However, inserting back into $x(y)$ gives
\begin{equation}\label{eq:one_over_z_is_zero}
x(-1/z) = 0
\end{equation}
so $x(y)$ is actually not singular at this point and the apparent singularity of $g(y)$ does not impose any convergence bound. Further, for a pole to appear, the numerator of $g(y)$ must be zero, i.e. the complex roots
\begin{equation}\label{eq:values_at_infinity}
y^2+\epsilon^2(1+zy)^2=0 \implies y=\frac{-\epsilon^2z\pm i \epsilon}{1+\epsilon^2z^2}
\end{equation}
However, inserting back into $x(y)$, such a root is mapped to infinity so it does not constitute a finite convergence bound.
Finally, a critical point means $\frac{\partial}{\partial y}\frac{y}{g(y)}=0$ with the derivative given at the beginning of the section. After multiplying the numerator and denominator with $z^2$, this implies
\begin{equation}\label{eq:y_c_req}
(zy)^3+\epsilon^2z^2(1+zy)^3=0
\end{equation}
Let
\begin{equation*}
\frac{zy}{1+zy}=r
\end{equation*}
Then the critical point is determined by
\begin{equation*}
r^3=-\epsilon^2z^2
\end{equation*}
whose roots are
\begin{equation*}
r=-\epsilon^{2/3}z^{2/3}\omega
\end{equation*}
where $\omega\in\{1, \exp(2\pi i/3), \exp(-2\pi i/3)\}$. Note that
\begin{equation*}
\epsilon^2z^2 = -r^3,\quad zy = \frac{r}{1-r}\quad\text{and}\quad 1+zy=\frac{1}{1-r}
\end{equation*}
Using this, $x^2=y^2/g(y)^2$ can be rewritten as
\begin{align*}
x^2 &= \frac{(zy)^2(1+zy)^2}{(zy)^2 + \epsilon^2z^2(1+zy)^2}\\
&=\frac{1}{(1-r)^3}
\end{align*}
The corresponding critical values $x_c$ fulfill
\begin{equation}\label{eq:x_bound}
|x_c| = \frac{1}{|1+\epsilon^{2/3}z^{2/3}\omega|^{3/2}}
\end{equation}
Now the distance from $x=0$ to the root is determined by the magnitude of the denominator and
\begin{equation*}
\left|(1+\epsilon^{2/3}z^{2/3})\right|> \left|(1+\epsilon^{2/3}z^{2/3}\exp(\pm2\pi i/3))\right|
\end{equation*}
and, consequently, the real root $\omega=1$ is the one closest to $x=0$. 
Inserting back $x=\rho\cos(\psi)/(ae^2)$ and $z=\tan(\psi)$ gives
\begin{equation}\label{eq:phi_bound_single_quadrant}
\rho_c =\frac{ae^2}
{\left(\cos(\psi)^{2/3}+\epsilon^{2/3}\sin(\psi)^{2/3}\right)^{3/2}}
\end{equation}
which is the radial equation of the evolute of the ellipse in the first quadrant, i.e. the range $0<\psi<\pi/2$ for which the series~\eqref{eq:base4} is derived. 

Further, for the root $\omega=1$, $zy = \frac{r}{1-r}$ and $r=-\epsilon^{2/3}z^{2/3}$ imply that for $z>0$
\begin{equation*}
-\frac{1}{z}<y_c<0,
\qquad
1+zy_c=\frac{1}{1+\epsilon^{2/3}z^{2/3}}>0
\quad\text{and}\quad
x_c<0
\end{equation*}
$y_c$ is the only real critical point, since $r$ is real only for $\omega=1$ and $y=r/(z(1-r))$ is real only for real $r$. Hence, $x'(y)>0$ on $y_c<y<\infty$, as $x'(0)=1/\epsilon>0$. Further, $x(0)=0$, so the inverse $y_0(x)$ starts at zero and continues along $x_c<x\leq0$ to $y_c$. Thus, the critical point $y_c$ is indeed on the branch starting at $y_0(0)=0$ and, as previously established, it gives the closest critical value. Moreover, $x(y)$ increases from $0$ to $\infty$ on $0\leq y<\infty$, so $y_0$ is real analytic on $0\leq x<\infty$, in particular at $x=-x_c$. Consequently, $x_c$ is the unique singularity of $y_0$ on $|x|=|x_c|$, and~\eqref{eq:phi_bound_single_quadrant} bounds the convergence of the series expansion.

Finally, as shown, for the evolute critical points $x'(y_c)=0$. Direct evaluation gives
\begin{equation*}
x''(y_c)=
\frac{3z}
{\left(y_c^2+\epsilon^2(1+zy_c)^2\right)^{1/2}}
\end{equation*}
which is finite and non-zero for $z\neq0$, since $y_c^2+\epsilon^2(1+zy_c)^2$ vanishes only at the points~\eqref{eq:values_at_infinity}, which are mapped to $x=\infty$. That is, the critical point is simple. Hence, locally
\begin{equation*}
x-x_c = C(y-y_c)^2+O((y-y_c)^3), \qquad C\neq0
\end{equation*}
so the inverse has a square-root branch point,
\begin{equation*}
y-y_c = C_1\sqrt{x-x_c}+O(x-x_c)
\end{equation*}
and
\begin{equation}\label{eq:roc_structure}
f(y(x))
=
f(y_c)+f'(y_c)C_1\sqrt{x-x_c}+O(x-x_c)
\end{equation}
Now
\begin{equation*}
f'(y)=
-\frac{1+z^2}
{\left(z+(1+z^2)y\right)^2+1}
\end{equation*}
which is obviously not zero for any real $z$. The only remaining obstruction for the argument is that the denominator goes to zero, i.e.
\begin{equation*}
y_c=\frac{-z\pm i}{1+z^2}
\end{equation*}
which are exactly the singular points studied in the next subsubsection. This means that we must have
\begin{equation*}
\frac{zy_c}{1+zy_c}=\pm iz
\end{equation*}
However, it must also fulfill~\eqref{eq:y_c_req}, which can be factored to a similar requirements
\begin{equation*}
\left(\frac{zy_c}{1+zy_c}\right)^3
=
-\epsilon^2 z^2
\end{equation*}
Simultaneously fulfilling these two requirements for a $z\neq0$ implies $\mp iz=-\epsilon^2$, which can obviously not be true  for any real $z$ and, hence, the denominator of $f'(y)$ does not vanish at the evolute critical points. Consequently, the composition $f(y(x))$ also has a square-root branch point there. Thus these singularities are not removable.

\subsubsection{Singular points of $f(y)$}\label{sec:singular_points_f}
The branch points of $f(y)$, i.e. of $\tan^{-1}(\cdot)$, occur at
\begin{equation*}
\frac{1}{z +(1+z^2)y}=\pm i\quad\text{corresponding to}\quad y_\pm=\frac{-z\pm i}{1+z^2}
\end{equation*}
In terms of $x$, this gives four solutions
\begin{equation*}
x_{\pm,\pm}=\pm\frac{1\pm iz}{e(1+z^2)}
\end{equation*}
Inserting $z=\tan(\psi)$ yields
\begin{equation*}
x_{\pm,\pm}=\pm\frac{\cos(\psi)\bigl(\cos(\psi) \pm i\sin(\psi)\bigr)}{e}
\end{equation*}
which, in turn, gives the following condition for the branch points
\begin{equation*}
|x_{\pm,\pm}|=\frac{\cos(\psi)}{e}
\end{equation*}
Now since $x=\frac{\rho\cos(\psi)}{ae^2}$, in terms of $\rho$, this is
\begin{equation}\label{eq:ae_bound}
\rho = ae
\end{equation}
This sets a potential bound for the region of convergence.
However, for this bound to apply, necessarily, it has to be tighter than the evolute and the branch points must be reached by the selected inverse branch $y_0(x)$. 

The \emph{Lagrange–B\"urmann} series expands the branch defined by $y_0(0)=0$. For $x_{\pm,\pm}$ to limit the Lagrange branch, it is not enough that $x(y_\pm)=x_{\pm,\pm}$, but one must have $y_0(x_{\pm,\pm})=y_\pm$.
To proceed, introduce
\begin{equation*}
t(y)=\frac{y}{1+zy}\quad\text{and}\quad y(t)=\frac{t}{1-zt}
\end{equation*}
Substituting into $x(y)=y/g(y)$ gives
\begin{equation*}
x(t)=\frac{t}{(1-zt)\sqrt{\epsilon^2+t^2}}
\end{equation*}
$t(y)$ is a M\"{o}bius transformation, which preserves the local branch structure. Under this transformation, the candidate singular points $y_\pm$ become $t=\pm i$, while the Lagrange branch is the inverse branch $t_0(x)$ satisfying $t_0(0)=0$.
To determine the behaviors of the bounds, square $x(t)$ and construct
\begin{equation*}
P(t,x)=x^2(1-zt)^2(t^2+\epsilon^2)-t^2
\end{equation*}
The branch $t_0(x)$ is one of the roots born from the double root $t=0$ at $x=0$. The squaring in $P(t,x)$ does introduce extra roots but since $t_0(x)$ is still one of them, it can still be used for analysis. Now choose a contour $\Gamma$
\begin{equation*}
\Gamma=\{|t|=r\},\qquad r=\left(\frac{\epsilon^2}{z}\right)^{1/3}
\end{equation*}
By the argument principle, if $\Gamma$ excludes $t=\pm i$ and if no roots cross $\Gamma$, the number of roots inside $\Gamma$ remains constant and, consequently, it would mean $t_0(x_{\pm,\pm})\neq\pm i$.
From~\eqref{eq:x_bound} and the root $\omega=1$, the bound~\eqref{eq:ae_bound} is potentially more constraining if
\begin{equation*}
|x|=\frac{1}{e\sqrt{1+z^2}}\leq R\quad\text{where}\quad R=\frac{1}{(1+\epsilon^{2/3}z^{2/3})^{3/2}}=\frac{1}{(1+zr)^{3/2}}
\end{equation*}
Further, since $|x_{\pm,\pm}|\leq R<1$, $e^2(1+z^2)>1$ which in turn implies $z>\epsilon/e$ meaning
\begin{equation*}
r=\left(\frac{\epsilon^2}{z}\right)^{1/3}
<\left(\epsilon e\right)^{1/3}\leq 1
\end{equation*}
So the contour excludes $t=\pm i$.
Furthermore, using $\epsilon^2=zr^3$, this means for $|t|=r$ and $|x|<R$
\begin{align}
\label{eq:P_triangle}\left|x^2(1-zt)^2(t^2+\epsilon^2)\right|
&\leq |x|^2(1+zr)^2(r^2+\epsilon^2)\\
\nonumber&=|x|^2r^2(1+zr)^3\\
\nonumber&=r^2\frac{|x|^2}{R^2}
\end{align}
For $|x|<R$, this is strictly smaller than $|t^2|=r^2$,
so consequently
\begin{equation*}
P(t,x)\neq0\qquad\text{if}\quad|t|=r,\quad |x|<R
\end{equation*}
At $|x|=R$, equality in~\eqref{eq:P_triangle} can only occur at $t=-r$. At these points,
\begin{equation*}
P(-r,x_{\pm,\pm})
=r^2\left(\frac{x_{\pm,\pm}^2}{R^2}-1\right)
\end{equation*}
This is nonzero because $x_{\pm,\pm}$ is nonreal for $z>0$.
Now, for each candidate $x_{\pm,\pm}$ satisfying $|x_{\pm,\pm}|\leq R$, consider the segment $x=sx_{\pm,\pm}$, $0\leq s\leq1$. The preceding estimates show that $P(t,sx_{\pm,\pm})$ does not vanish on $\Gamma$ anywhere along this segment. Hence, by the argument principle, the number of roots inside $\Gamma$, counted with multiplicity, remains constant.
Since $sx_{\pm,\pm}$ lies in $|x|<R$ for $s<1$, where $y_0$ and hence $t_0$ is analytic by the preceding subsubsection, $t_0(sx_{\pm,\pm})$ varies continuously along the segment.
$P(t,0)=-t^2$, so there are two roots inside $\Gamma$. The selected branch $t_0(x)$ starts at zero and cannot leave the contour without crossing it. Since $r<1$, it therefore satisfies
\begin{equation*}
t_0(x_{\pm,\pm})\notin\{i,-i\}
\end{equation*}
Consequently, the singularities of $f$ are not reached by the selected inverse branch at any value within or on $|x|=R$ and they impose no smaller convergence radius and introduce no additional singularities on the convergence circle.

\subsubsection{Convergence on the boundary and at the end-points}
At this point, it has been shown that~\eqref{eq:phi_bound_single_quadrant} provides an open bound $\rho<\rho_c$ for the region of convergence on the open interval $0<\psi<\pi/2$. To determine convergence on the boundary, note that by the two preceding subsubsections, $x_c$ is the unique singularity of $f(y_0(x))$ on $|x|=|x_c|$, and from~\eqref{eq:roc_structure} it is of square-root type. Hence, by the transfer method of Section~\ref{sec:transfer}, the series converges absolutely on the limit $\rho=\rho_c$.
Further, the convergence for $\psi=0$, i.e. $z=0$, is immediate by noting that, in this case,
\begin{equation*}
x(y)=y/\sqrt{y^2+\epsilon^2}\quad\text{implying}\quad y(x)=\epsilon x/\sqrt{1-x^2}
\end{equation*}
Here, $f(y)=\pi/2-\tan^{-1}(y)$, whose branch points $y=\pm i$ correspond to $x=\pm1/e$, with $|x|>1$. Hence, $f(y(x))$ is analytic for $|x|<1$ and has two singularities on $|x|=1$, at $x=\pm1$. As $x\to\pm1$, $|y(x)|\to\infty$, and $\tan^{-1}(y)=\pm\pi/2-1/y+O(y^{-3})$ gives
\begin{equation*}
f(y(x))=f(y(\pm1))+\frac{\sqrt{1-x^2}}{\epsilon x}+O\big((1-x^2)^{3/2}\big)
\end{equation*}
i.e. square-root singularities at both points. By the multiple-singularity case of the transfer method, the series converges absolutely on $|x|=1$, i.e.\ $\rho=ae^2$, which coincides with the evolute.

Finally, at $\psi=\pi/2$, every nonconstant coefficient of the original series~\eqref{eq:base4}, after rewriting it in terms of $\rho$ and $\psi$, vanishes because it contains the factor $\cos(\psi)$. Thus the series for $\phi-\pi/2$ is identically zero and converges for all $\rho$.

\subsubsection{Complete region of convergence}
The evolute bound and the convergence on the bound and at the end-points mean that, for $0\leq\psi<\pi/2$ the series for $\phi-\pi/2$ converges absolutely up to and including the bound $\rho=\rho_c$. For the exceptional endpoint $\psi=\pi/2$, the series converges for all $\rho$.

The series~\eqref{eq:base4} derivation and the region of convergence were carried out for $z>0$, i.e. the open first quadrant. The series is later on extended to all quadrants with~\eqref{eq:phi_dense_expansion}. By reflection symmetry, the evolute bound holds in the other quadrants after replacing $\sin(\psi)$ and $\cos(\psi)$ with their absolute values. Combining these results, the unnormalized series for $\phi$, obtained by restoring the angular factor and constant term in~\eqref{eq:phi_dense_expansion}, converges absolutely for
\begin{equation*}
\rho \leq \frac{ae^2}{\left(|\cos(\psi)|^{2/3}+\epsilon^{2/3}|\sin(\psi)|^{2/3}\right)^{3/2}}\quad\text{or}\quad \psi=\pm\pi/2
\end{equation*}
where the absolute values and $\pm$ extend the evolute to all quadrants. The normalized series is treated below.

Strictly inside the evolute and off the coordinate axes, the series for $\phi-\pi/2$ may be differentiated termwise with respect to $\rho$, $\psi$ and $\epsilon$. The singularities and branch-point conditions identified above depend continuously on $(\psi,\epsilon)$, also for small complex perturbations, so the sum is jointly analytic and locally uniformly bounded there. Cauchy's estimate then bounds the coefficients locally uniformly, and the claim follows from the Weierstrass theorem. The same argument applies to the series derived in Sections~\ref{sec:conv_powers}--\ref{sec:conv_h}, since the additional conditions established there are likewise stable under small perturbations. These arguments justify termwise differentiation strictly inside the convergence region. Convergence of differentiated series on the boundary requires separate analysis and depends on the quantity and derivative order.

\subsubsection{Region of convergence for series of $(\phi-\pi/2)/\cos(\psi)$}
For fixed $0\leq\psi<\pi/2$, the factor $1/\cos(\psi)$ is a finite nonzero constant with respect to $\rho$. Hence, dividing the series for $\phi-\pi/2$ by $\cos(\psi)$ changes neither its radius of convergence nor its absolute convergence on the boundary. Consequently, the series for $(\phi-\pi/2)/\cos(\psi)$ has radius of convergence $\rho_c$
and converges absolutely for
\begin{equation*}
0\leq\rho\leq\rho_c
\end{equation*}
This includes the endpoint $\psi=0$, where $\cos(\psi)=1$ and
$\rho_c=ae^2$.

At $\psi=\pi/2$, both the numerator and denominator vanish, so the quotient is defined by its limit. Eliminating $h$ from \eqref{eq:base1} and \eqref{eq:base2} gives
\begin{equation*}
\frac{\cos(\phi)}{\cos(\psi)}
=
\frac{\rho}
{e^2N+\rho\sin(\psi)/\sin(\phi)}
\end{equation*}
As $\psi\to\pi/2$ on the closest-point branch, $\phi\to\pi/2$ and $N\to a/\epsilon$. Therefore
\begin{equation}\label{eq:coscos_limit}
\lim_{\psi\to\pi/2}\frac{\cos(\phi)}{\cos(\psi)}
=\frac{\epsilon\rho}{ae^2+\epsilon\rho}
\end{equation}
Since
\begin{equation*}
\lim_{\phi\to\pi/2}\frac{\phi-\pi/2}{\cos(\phi)}=-1
\end{equation*}
the continuously extended quotient is
\begin{equation*}
\left.\frac{\phi-\pi/2}{\cos(\psi)}\right|_{\psi=\pi/2}
=-\frac{\epsilon\rho}{ae^2+\epsilon\rho}
\end{equation*}
Introducing our variable $\sigma=\rho/(ae^2)$ and expanding gives
\begin{equation*}
-\frac{\epsilon\sigma}{1+\epsilon\sigma}
=\sum_{k=1}^{\infty}(-1)^k \left(\epsilon\sigma\right)^k,\qquad \left|\epsilon\sigma\right|<1
\end{equation*}
Its radius of convergence in $\rho$ is $ae^2/\epsilon$, determined by the pole at $\rho=-ae^2/\epsilon$. This radius equals $\rho_c$ at $\psi=\pi/2$. At the positive boundary $\rho=ae^2/\epsilon$, however, the terms alternate between $-1$ and $1$ and do not tend to zero. The series therefore diverges there.

It remains to show that this is the endpoint series obtained by evaluating the coefficients in \eqref{eq:phi_dense_expansion} at $\psi=\pi/2$. At $\psi=\pi/2$, the coefficient of $\sigma^k$ is
\begin{equation*}
\epsilon^{1+\prt(k+1)}\sum_{n=0}^{\lfloor(k-1)/2\rfloor}\epsilon^{2n}
\sum_{l=0}^{\lfloor(k-1)/2\rfloor}d^\phi_{k,l,n}
\end{equation*}
In \eqref{eq:d_phi}, the dependence on $l$ is through
$(-1)^l\binom{\lfloor(k-1)/2\rfloor-n}{l-n+j}$. For $0\leq j\leq n$,
\begin{equation*}
\sum_{l=0}^{\lfloor(k-1)/2\rfloor}(-1)^l\binom{\lfloor(k-1)/2\rfloor-n}{l-n+j}
=
\begin{cases}
0, & n<\lfloor(k-1)/2\rfloor\\
(-1)^{\lfloor(k-1)/2\rfloor-j}, & n=\lfloor(k-1)/2\rfloor
\end{cases}
\end{equation*}
by the alternating binomial identity. Thus only $n=\lfloor(k-1)/2\rfloor$ contributes. In this case, the factor $\binom{k-2-q}{k-1-q}$ eliminates every term of the $q$-sum except $q=k-1$, where it equals $\binom{-1}{0}=1$. Consequently,
\begin{equation*}
\sum_{l=0}^{\lfloor(k-1)/2\rfloor}d^\phi_{k,l,\lfloor(k-1)/2\rfloor}
=\frac{(-1)^k}{k}\sum_{j=0}^{\lfloor(k-1)/2\rfloor}(-1)^j2^{k-1-2j}\binom{k-1-j}{j}
\end{equation*}
To evaluate the remaining sum, expand the following generating function geometrically and then apply the binomial theorem
\begin{align*}
\frac{1}{1-2t+t^2}
&=\sum_{m=0}^{\infty}(2t-t^2)^m\\
&=\sum_{m=0}^{\infty}\sum_{j=0}^{m}(-1)^j2^{m-j}\binom{m}{j}t^{m+j}
\end{align*}
Extracting the coefficient of $t^{k-1}$ requires $m=k-1-j$, and hence
\begin{align*}
\sum_{j=0}^{\lfloor(k-1)/2\rfloor}
(-1)^j2^{k-1-2j}\binom{k-1-j}{j}
&=[t^{k-1}](1-2t+t^2)^{-1}\\
&=[t^{k-1}](1-t)^{-2}\\
&=k
\end{align*}
This in turn means
\begin{equation*}
\sum_{l=0}^{\lfloor(k-1)/2\rfloor}d^\phi_{k,l,\lfloor(k-1)/2\rfloor}=\frac{(-1)^k}{k}k=(-1)^k
\end{equation*}
For the $\epsilon$ power, since $1+\prt(k+1)+2\lfloor(k-1)/2\rfloor=k$, the continuously extended coefficients therefore satisfy
\begin{equation*}
\left.[\sigma^k]\frac{\phi-\pi/2}{\cos(\psi)}\right|_{\psi=\pi/2}
=(-1)^k\epsilon^k
\end{equation*}
These coefficients coincide with those obtained by first taking the limit and then expanding. Thus the series for $(\phi-\pi/2)/\cos(\psi)$ converges absolutely for $0\leq\rho\leq\rho_c$ when $0\leq\psi<\pi/2$. At $\psi=\pi/2$, it converges absolutely for $0\leq\rho<\rho_c$ and diverges at $\rho=\rho_c$.

\subsection{Region of convergence for series of \texorpdfstring{$(\phi-\pi/2)^i$}{(φ - π/2) to the power i}}\label{sec:conv_powers}
The transformation
\begin{equation*}
F(u) = u^{i}\quad\text{where}\quad i\in\mathbb{N},\quad i\geq 1
\end{equation*}
is a polynomial and therefore introduces no additional singularities.
Let
\begin{equation*}
u(x)=\phi(x)-\frac{\pi}{2}
\end{equation*}
For $0<\psi<\pi/2$, at the nearest evolute critical value $x_c$, the branch $u(x)$ has a genuine square-root singularity:
\begin{equation*}
u(x)=u_c+B\sqrt{x-x_c}+O(x-x_c),\qquad B\neq0
\end{equation*}
This follows from the fact that $y_0(x)$ has a simple square-root branch point and $f'(y_c)\neq0$, so the singularities of $f(y_0(x))$ are not removable.
Now let $G(x)=F(u(x))=u(x)^i$. If $u_c\neq0$
\begin{align}
G(x)&=\left(u_c+B\sqrt{x-x_c}+O(x-x_c)\right)^i\nonumber\\
&=u_c^i+i u_c^{i-1}B\sqrt{x-x_c}+O(x-x_c)\label{eq:power_roc_structure}
\end{align}
Since
\begin{equation*}
i u_c^{i-1}B\neq0
\end{equation*}
the square-root singularity is not removable. It remains to verify that $u_c\neq0$.

At the nearest evolute critical point, corresponding to $\omega=1$, we have
\begin{equation*}
-\frac1z<y_c<0.
\end{equation*}
Since
\begin{equation*}
y=\frac{1}{\tan(\phi)-\tan\psi},\qquad z=\tan(\psi)
\end{equation*}
this gives
\begin{equation*}
\tan(\phi_c)=z+\frac1{y_c}<0
\end{equation*}
The relevant continuation is the real branch starting from
$\phi=\pi/2$ at $y=0$. As $y$ decreases from $0$ to $y_c<0$,
$\tan(\phi)=z+1/y$ passes from $-\infty$ to a finite negative value,
so $\phi_c$ lies in the second quadrant. Hence
\begin{equation*}
\frac{\pi}{2}<\phi_c<\pi,\qquad 0<\phi_c-\frac{\pi}{2}<\frac{\pi}{2}
\end{equation*}
Thus $u_c=\phi_c-\pi/2\neq0$, so $x_c$ remains a genuine square-root singularity of $G(x)$, and the radius of convergence is not enlarged. Since $F$ is a polynomial, $G$ is analytic wherever $u$ is, so $x_c$ is also the unique singularity of $G$ on $|x|=|x_c|$. Hence, the series of $(\phi-\pi/2)^i$ converges inside the evolute and, by the transfer method, absolutely on the evolute boundary.

Furthermore, at the endpoint $\psi=0$, the singularities on the evolute bound $\rho=ae^2$ are at $x=\pm1$, where $u=\mp\pi/2$. The structure~\eqref{eq:power_roc_structure} applies at both points, with square-root coefficients proportional to $i(\mp\pi/2)^{i-1}\neq0$. Consequently, both square-root singularities are preserved under $F(u)=u^i$, and the corresponding series converges absolutely on $|x|=1$, i.e. on the evolute boundary $\rho=ae^2$. Consequently, $(\phi-\pi/2)^i$ has the same region of convergence as $\phi-\pi/2$.

For $\psi=\pi/2$, the original series of $\phi-\pi/2$ vanishes identically, and hence, for $i\geq1$, $u^i=0$ and the series therefore converges trivially for all $\rho$.

Altogether, for $0\leq\psi<\pi/2$, the series of $\left(\phi-\frac{\pi}{2}\right)^i$ converges absolutely up to and including the evolute boundary $\rho\leq\rho_c$ and the evolute remains the convergence boundary. At the exceptional end-point $\psi=\pi/2$, the series converges for all $\rho$.

\subsection{Region of convergence for even power series in \texorpdfstring{$\phi-\pi/2$}{φ - π/2}}\label{sec:even_power_series_convergence}
Let $u(x)=\phi(x)-\pi/2$ and suppose that $F$ has a Taylor expansion about $u=0$,
\begin{equation*}
F(u)=\sum_{i=0}^{\infty} A_i u^{2i},
\qquad |u|<R_F
\end{equation*}

For fixed $0\leq\psi<\pi/2$, let
\begin{equation*}
R=\frac{\rho_c\cos(\psi)}{ae^2}
\end{equation*}
denote the radius of convergence of $u(x)$, and define
\begin{equation*}
K=\{u(x):|x|\leq R\}
\end{equation*}
The previously established absolute convergence on the boundary makes $u$ continuous on the closed disk, so $K$ is compact.

Suppose that $F$, initially defined by its Taylor series at zero, admits a single-valued analytic continuation to an open neighborhood of $K$. Then $F(u(x))$ is analytic for $|x|<R$, and composition with $F$ introduces no additional singularities on $|x|=R$. The Taylor series of $F(u(x))$ about $x=0$ therefore converges throughout $|x|<R$. The same argument applies to any analytic transformation of $u(x)$, whether or not its Taylor series is even, but since the derivation only uses even power series, the discussion is restricted to this case.

For $0<\psi<\pi/2$, at the nearest evolute critical value $x_c$, one has locally
\begin{equation*}
u(x)=u_c+B\sqrt{x-x_c}+O(x-x_c),
\qquad B\neq0
\end{equation*}
Since the assumed continuation of $F$ is analytic at
$u_c$, it follows that
\begin{equation*}
F(u(x))
=
F(u_c)+F'(u_c)B\sqrt{x-x_c}+O(x-x_c)
\end{equation*}
Hence, if
\begin{equation*}
F'(u_c)\neq0
\end{equation*}
the evolute singularity is not removable, and the evolute remains the actual convergence boundary.

More generally, near each dominant singularity $x_c$, $u$ is a convergent series in powers of $(x-x_c)^{1/2}$, and hence so is $F(u(x))$, i.e. $F(u(x))=H_0(x)+(x-x_c)^{1/2}H_1(x)$ with $H_0$, $H_1$ analytic at $x_c$. Whether or not the leading singular term cancels, the singular part is thus $O((x-x_c)^{1/2})$, and the transfer of $O$-terms (Theorem VI.3 of~\cite{Flajolet2009}) gives $a_k=O(R^{-k}k^{-3/2})$. Consequently, the composed Taylor series converges absolutely on $|x|=R$, including the point $x=R$, i.e. $\rho=\rho_c$.

At the endpoint $\psi=0$, one has $R=1$, with
\begin{equation*}
u(1)=-\pi/2,\qquad u(-1)=\pi/2
\end{equation*}
Under the same analyticity assumption on a neighborhood of $K$, the square-root expansions at both endpoints can be composed with $F$. The resulting series therefore converges absolutely on $|x|=1$, including the evolute endpoint $\rho=ae^2$. If $F'(-\pi/2)\neq0$, equivalently $F'(\pi/2)\neq0$ since $F$ is even, the singularities at $x=\pm1$ are preserved and the radius remains $1$.

At the other endpoint, $\psi=\pi/2$, the original series for $\phi-\pi/2$ in $\rho$ vanishes identically. Consequently, $F(u)=F(0)=A_0$, and the series converges for all $\rho\geq0$.

Altogether, a sufficient condition for absolute convergence throughout the closed evolute region is that, for each fixed $0\leq\psi<\pi/2$, $F$ admit the analytic continuation to a neighborhood of $K$ described above. In particular, the stronger condition
\begin{equation*}
R_F>\max_{|x|\leq R}|u(x)|
\end{equation*}
is sufficient. If, in addition, $F'(u_c)\neq0$, the evolute remains the actual convergence boundary for $0<\psi<\pi/2$. Failure of these sufficient conditions requires separate analysis. Additional singularities may restrict convergence, while cancellation may enlarge the convergence region.

\subsection{Region of convergence for series of \texorpdfstring{$\cos(\phi)$, $\sin(\phi)$, $\cos(\phi)/\cos(\psi)$, and $\sin(\phi)-1$}{cos(φ), sin(φ), cos(φ)/cos(ψ), and sin(φ) - 1}}
Let $u=\phi-\frac{\pi}{2}$. Then
\begin{equation*}
\sin(\phi)=\cos(u),
\qquad
\cos(\phi)=-\sin(u)
\end{equation*}
The functions $\cos(u)$ and $-\sin(u)$ are entire, so the condition of Section~\ref{sec:even_power_series_convergence} holds trivially.
For $0<\psi<\pi/2$, at the nearest evolute singularity $0<u_c<\frac{\pi}{2}$
\begin{equation*}
\cos'(u_c)=-\sin(u_c)\neq0,\qquad-\sin'(u_c)=-\cos(u_c)\neq0
\end{equation*}
Hence the square-root singularity of $u(x)$ is not removable under either transformation. Further, for fixed $0<\psi<\pi/2$, $1/\cos(\psi)$ is a finite constant and therefore does not affect the convergence of the series for $\cos(\phi)/\cos(\psi)$. Subtracting the constant 1 likewise does not affect the convergence of the series for $\sin(\phi)-1$. Hence all four series have the same convergence bound as the series for $u=\phi-\pi/2$. By Section~\ref{sec:even_power_series_convergence}, they converge absolutely on the evolute boundary.

It remains to consider the endpoints. For $\psi=0$, $z=0$ and
\begin{equation*}
y(x)=\frac{\epsilon x}{\sqrt{1-x^2}},
\qquad
\tan(\phi)=\frac{\sqrt{1-x^2}}{\epsilon x}
\end{equation*}
Consequently,
\begin{equation*}
\sin(\phi)
=
\frac{\sqrt{1-x^2}}{\sqrt{1-e^2x^2}},
\qquad
\cos(\phi)
=
\frac{\epsilon x}{\sqrt{1-e^2x^2}}
\end{equation*}

For $\sin(\phi)$, the nearest singularities are the square-root branch points at $x=\pm1$. The transfer argument therefore gives absolute convergence on $|x|=1$, including the evolute endpoint $x=1$, or $\rho=ae^2$. The same conclusions hold for $\sin(\phi)-1$.


For \(\cos(\phi)\), the derivative \(-\cos(u)\) vanishes at \(u=\mp\pi/2\), so the leading square-root terms disappear. In fact, the singularities are completely cancelled, as the explicit expression $\cos(\phi)=\frac{\epsilon x}{\sqrt{1-e^2x^2}}$ shows by being analytic at $x=1$. Hence the evolute is not the convergence boundary for the series for $\cos(\phi)$ in the direction $\psi=0$. Its nearest singularities occur at $x=\pm\frac{1}{e}$, so the radius of convergence is $1/e$, corresponding to $\rho=ae$ which is always larger than the evolute bound.
At $x=1/e$, the binomial expansion gives the positive terms
\begin{equation*}
\frac{\epsilon}{e}\frac{\binom{2m}{m}}{4^m}
\sim\frac{\epsilon}{e\sqrt{\pi m}}
\end{equation*}
Their sum diverges. Thus the radius of convergence is $1/e$ in $x$, or equivalently $ae$ in $\rho$, i.e. the series converges for $\rho<ae$ and diverges at $\rho=ae$.
Since $\cos(\psi)=1$ at $\psi=0$, the same result holds for $\cos(\phi)/\cos(\psi)$.

At the other endpoint, $\psi=\pi/2$, the original series satisfies
\begin{equation*}
u=\phi-\frac{\pi}{2}=0
\end{equation*}
identically. Hence
\begin{equation*}
\sin(\phi)=1,
\qquad
\cos(\phi)=0
\end{equation*}
and the series of $\cos(\phi)$, $\sin(\phi)$ and $\sin(\phi)-1$ converge trivially for all $\rho$. However, at $\psi=\pi/2$, both $\cos(\phi)$ and $\cos(\psi)$ vanish, so the ratio $\cos(\phi)/\cos(\psi)$ must be treated by taking the limit $\psi\to\pi/2$. From~\eqref{eq:coscos_limit}
\begin{equation*}
\left.\frac{\cos(\phi)}{\cos(\psi)}\right|_{\psi=\pi/2}
=
\frac{\epsilon\rho}{ae^2+\epsilon\rho}
\end{equation*}
Introducing
\begin{equation*}
s=\frac{\epsilon\rho}{ae^2}
\end{equation*}
the corresponding series becomes
\begin{equation*}
\frac{s}{1+s}
=
\sum_{k=1}^{\infty}(-1)^{k-1}s^k,
\qquad |s|<1
\end{equation*}
Hence the radius of convergence in $\rho$ is $\frac{ae^2}{\epsilon}$ so the series converges for $\rho<\frac{ae^2}{\epsilon}$, which coincides with the evolute bound at $\psi=\pi/2$. At the boundary $\rho=ae^2/\epsilon$, however, $s=1$ and the series reduces to $1-1+1-\cdots$, whose terms do not tend to zero. Thus, at $\psi=\pi/2$, the series for $\cos(\phi)/\cos(\psi)$ converges inside the evolute but diverges at the evolute boundary, even though the function itself remains finite there.

Thus, for $0<\psi<\pi/2$, the series for $\cos(\phi)$, $\sin(\phi)$, $\cos(\phi)/\cos(\psi)$, and $\sin(\phi)-1$ converge absolutely up to and including the evolute boundary.
At $\psi=0$, the series for $\sin(\phi)$ and $\sin(\phi)-1$ likewise converge absolutely on the evolute boundary. In contrast, for $\cos(\phi)$ and $\cos(\phi)/\cos(\psi)$ the evolute singularity is cancelled and the convergence region extends to $\rho<ae$ with both series diverging at $\rho=ae$. However, beyond $\rho=ae^2$, this convergent series represents the analytic continuation of the expanded branch and not the closest-point. See Section~\ref{sec:validity_beyond} for details.
At the exceptional endpoint $\psi=\pi/2$, the series of $\sin(\phi)$, $\sin(\phi)-1$, and $\cos(\phi)$ all converge for all $\rho$ whereas the series of $\cos(\phi)/\cos(\psi)$ converges for $\rho<ae^2/\epsilon$ but diverges at the evolute boundary $\rho=ae^2/\epsilon$.

\subsection{Region of convergence for series of \texorpdfstring{$\sin(\psi)/\sin(\phi)$}{sin(ψ)/sin(φ)}}
Since
\begin{equation*}
\sin(\phi)=\cos(\phi-\pi/2)=\cos(u),
\qquad
u=\phi-\frac{\pi}{2}
\end{equation*}
for $0<\psi<\pi/2$, the function
\begin{equation*}
F(u)=\frac{\sin(\psi)}{\cos(u)}
\end{equation*}
has poles at
\begin{equation*}
u=\frac{\pi}{2} + k\pi,\qquad k\in\mathbb{Z}
\end{equation*}
The poles nearest to zero are $u=\pm\pi/2$, so radius of convergence of its Taylor series is
\begin{equation*}
R_F=\frac{\pi}{2}
\end{equation*}
Since $\max_{|x|\leq R}|u(x)|<R_F$ is not readily verified, the neighborhood condition of Section~\ref{sec:even_power_series_convergence} is used instead.

For $0<\psi<\pi/2$, a pole would require $\cos(u(x_\ast))=0$ and, hence, $\phi(x_\ast)=k\pi$ for some integer $k$. By its definition
\begin{equation*}
y_0(x_\ast)=\frac{1}{\tan(\phi(x_\ast))-z}=-\frac{1}{z}
\end{equation*}
However, from~\eqref{eq:one_over_z_is_zero}, $x(-1/z)=0$, but the selected branch satisfies $y_0(0)=0$. This is a contradiction and consequently $\cos(u(x))\neq0$ throughout the closed convergence disk.
Thus the compact set $K$ defined in Section~\ref{sec:even_power_series_convergence} contains no pole of $F$. Consequently, $F$ is analytic on an open neighborhood of $K$, and the composition result establishes absolute convergence up to and including the evolute boundary.

Moreover, at the nearest evolute singularity of $u(x)$,
\begin{equation*}
0<u_c<\frac{\pi}{2}
\end{equation*}
so $F$ is analytic at $u_c$. Further,
\begin{equation*}
F'(u_c)
=
\sin(\psi)\frac{\sin(u_c)}{\cos^2(u_c)}\neq0
\end{equation*}
Hence the square-root singularity of $u(x)$ is not removable under the transformation $F$, and the evolute remains the convergence boundary. 

It remains to consider the endpoints. At $\psi=0$,
\begin{equation*}
\sin(\psi)=0,
\end{equation*}
so all coefficients of the Taylor series
\begin{equation*}
\frac{\sin(\psi)}{\cos(u)}
=
\sin(\psi)\sum_{i=0}^{\infty}\frac{E_{2i}}{(2i)!}u^{2i}
\end{equation*}
vanish. Consequently, the series is identically zero and converges trivially for all $\rho$. This agrees with the continuous value of $\sin(\psi)/\sin(\phi)$ inside and on the evolute, including the indeterminate endpoint $\rho=ae^2$, where the limiting value is zero. The zero limit at $\rho=ae^2$ follows from
$\sin(\psi)/\sin(\phi)=(\epsilon^2N+h)/\rho$, since $N\to a$ and $h\to-a\epsilon^2$ at this cusp. Beyond the evolute the physical solution changes branch, and the identically zero series no longer represents the continuous extension of the quotient.

At the other endpoint, $\psi=\pi/2$, the original series satisfies
\begin{equation*}
u=\phi-\frac{\pi}{2}=0
\end{equation*}
identically, while $\sin(\psi)=1$. Hence
\begin{equation*}
\frac{\sin(\psi)}{\sin(\phi)}=1
\end{equation*}
and the corresponding series again converges trivially for all $\rho$.

Thus, for $0<\psi<\pi/2$, the series for $\sin(\psi)/\sin(\phi)$ has the evolute as its convergence boundary and converges absolutely up to and including that boundary. At both endpoints $\psi=0$ and $\psi=\pi/2$, the series degenerates to a constant series and converges for all $\rho$.

\subsection{Region of convergence for series of \texorpdfstring{$\epsilon^2N/a$}{ε²N/a}}\label{sec:roc_N}
Writing $\epsilon^2N/a$ as a function of $u=\phi-\pi/2$ gives
\begin{equation*}
\frac{\epsilon^2N}{a}
=
\frac{\epsilon^2}{\sqrt{1-e^2\sin^2(\phi)}}
=
\frac{\epsilon^2}{\sqrt{1-e^2\cos^2(u)}}
\end{equation*}
Let
\begin{equation*}
F(u)
=
\frac{\epsilon^2}{\sqrt{1-e^2\cos^2(u)}}, \qquad F(0)=\epsilon
\end{equation*}
The nearest singularities of $F(u)$ occur when
\begin{equation*}
1-e^2\cos^2(u)=0
\end{equation*}
i.e. at
\begin{equation*}
u=\pm i\operatorname{arcosh}(1/e)
\end{equation*}
However, these are singularities of $F(u)$, and constitute finite convergence bounds for $F(u(x))$ only if they are reached at finite values of $x$.

Since
\begin{equation*}
\tan(\phi)=z+\frac1y
\end{equation*}
and
\begin{equation*}
\sin^2(\phi)
=
\frac{\tan^2(\phi)}{1+\tan^2(\phi)}
\end{equation*}
the singularity condition can equivalently be written as
\begin{equation*}
1-e^2\sin^2(\phi)=0
\qquad\Longleftrightarrow\qquad
1+\epsilon^2\tan^2(\phi)=0
\end{equation*}
or
\begin{equation*}
1+\epsilon^2\left(z+\frac1y\right)^2=0
\end{equation*}
Multiplication by $y^2$ gives
\begin{equation*}
y^2+\epsilon^2(1+zy)^2=0
\end{equation*}
which is precisely the condition in~\eqref{eq:values_at_infinity}. These points are mapped to $x=\infty$, and hence the singularities of $F(u)$ introduce no additional finite convergence bound.
Hence $1-e^2\cos^2(u(x))$ is nonzero on the closed convergence disk. Its square root therefore has an analytic branch along $u(x)$, specified by $F(0)=\epsilon$, wherever $u$ is analytic. Near $x_c$, $F$ is analytic at $u_c$, so the local argument of Section~\ref{sec:even_power_series_convergence} applies. This replaces the single-valuedness assumption of that section.

For $0<\psi<\pi/2$, the nearest finite singularity therefore remains the evolute singularity of $u(x)$. At this point,
\begin{equation*}
0<u_c<\frac{\pi}{2}
\end{equation*}
and
\begin{equation*}
F'(u_c)=-\frac{\epsilon^2e^2\sin(u_c)\cos(u_c)}
{\left(1-e^2\cos^2(u_c)\right)^{3/2}}
\neq 0
\end{equation*}
Since $1-e^2\cos^2(u_c)>0$, the continued branch of $F$ is analytic near $u_c$. Together with $F'(u_c)\neq 0$, this shows that the square-root singularity of $u(x)$ is not removable under the transformation $F$. The evolute therefore remains the convergence boundary and, by the same transfer argument as before, the series converges absolutely on the evolute boundary.

It remains to consider the endpoints. For $\psi=0$, using
\begin{equation*}
\sin(\phi)
=
\frac{\sqrt{1-x^2}}{\sqrt{1-e^2x^2}}
\end{equation*}
gives
\begin{align*}
\frac{\epsilon^2N}{a}
&=
\frac{\epsilon^2}
{\sqrt{1-e^2\sin^2(\phi)}}
=
\epsilon\sqrt{1-e^2x^2}
\end{align*}
Thus the square-root singularity of $u(x)$ at the evolute point $x=1$ is cancelled. The nearest singularities of the transformed function instead occur at
\begin{equation*}
x=\pm\frac1e,
\end{equation*}
corresponding to
\begin{equation*}
\rho=ae.
\end{equation*}
Both singularities are of square-root type
\begin{equation*}
\frac{\epsilon^2N}{a}
\sim C_+(1-ex)^{1/2}
\quad\text{as }x\to1/e,
\qquad C_+\neq0
\end{equation*}
and
\begin{equation*}
\frac{\epsilon^2N}{a}
\sim C_-(1+ex)^{1/2}
\quad\text{as }x\to-1/e,
\qquad C_-\neq0
\end{equation*}
The transfer argument therefore gives Taylor coefficients
of order\newline $O(e^k k^{-3/2})$, and hence absolute convergence on $|x|=1/e$, including the physical boundary $\rho=ae$. Hence, for $\psi=0$, the convergence region is
\begin{equation*}
\rho\leq ae.
\end{equation*}

At the other endpoint, $\psi=\pi/2$, one has $u=0$ identically and therefore
\begin{equation*}
\frac{\epsilon^2N}{a}
=
F(0)
=
\epsilon
\end{equation*}
The corresponding series is constant and converges trivially for all $\rho$.

Thus, for $0<\psi<\pi/2$, the series for $\epsilon^2N/a$ converges absolutely up to and including the evolute boundary. At $\psi=0$, cancellation of the evolute singularities at $x=\pm1$ extends the convergence region to $\rho\leq ae$. However, beyond $\rho=ae^2$, this convergent series represents the analytic continuation of the expanded branch and not the closest point, see Section~\ref{sec:validity_beyond}. At the exceptional endpoint $\psi=\pi/2$, the series converges for all $\rho$.

\subsection{Region of convergence for series of \texorpdfstring{$h$}{h}}\label{sec:conv_h}
The series for $h$ is obtained from
\begin{equation*}
\frac{h}{a}
=\frac{\rho}{a}\frac{\sin(\psi)}{\sin(\phi)}-\frac{\epsilon^2N}{a}
\end{equation*}
Hence, for $0<\psi<\pi/2$, the series converges absolutely up to and including the evolute boundary, since both component series do so. However, since singularities of the two terms may cancel, this alone does not establish that the evolute is the actual convergence boundary for $h$.

To examine whether cancellation extends the convergence radius, write $H=h/a$. Multiplying the base relations \eqref{eq:base1} and \eqref{eq:base2} by $\cos(\phi)$ and $\sin(\phi)$, respectively, and adding gives
\begin{equation*}
H=\frac{\rho}{a}\cos(\phi-\psi)
-\sqrt{1-e^2\sin^2(\phi)}
\end{equation*}
Using $x=\rho\cos(\psi)/(ae^2)$ and $z=\tan(\psi)$,
this becomes
\begin{equation*}
H=e^2x(\cos(\phi)+z\sin(\phi))-\sqrt{\cos^2(\phi)+\epsilon^2\sin^2(\phi)}
\end{equation*}
The question is whether $H$ is analytic at the nearest critical value
\begin{equation*}
x_c=-\frac{1}{(1+\epsilon^{2/3}z^{2/3})^{3/2}}
\end{equation*}
We consider $H$ on the branch induced by $y_0(x)$, selected by $y_0(0)=0$ and $\phi(0)=\pi/2$. As previously established, along the continuation $x_c\leq x<0$, one has $-1/z<y_c\leq y_0(x)<0$. Consequently, $\tan(\phi)=z+1/y_0(x)<0$, and continuity from $\phi(0)=\pi/2$ selects $\pi/2<\phi<\pi$. Thus, with $t=\tan(\phi)$
\begin{equation*}
\cos(\phi)=-\frac{1}{\sqrt{1+t^2}},\qquad\sin(\phi)=-\frac{t}{\sqrt{1+t^2}}
\end{equation*}
and
\begin{equation*}
H=-e^2x\frac{1+zt}{\sqrt{1+t^2}}-\frac{\sqrt{1+\epsilon^2t^2}}{\sqrt{1+t^2}}
\end{equation*}
The square roots are positive along this real continuation. Further, $y=1/(t-z)$. Substituting into $x(y)$ gives
\begin{equation*}
x(t)=-\frac{t}{(t-z)\sqrt{1+\epsilon^2t^2}}
\end{equation*}
In terms of $t$, the critical point is
\begin{equation*}
t_c=z+\frac1{y_c}
=-\left(\frac{z}{\epsilon^2}\right)^{1/3}<0
\end{equation*}
%
Along the studied branch,
\begin{equation*}
\frac{\partial H}{\partial t}
=
\frac{e^2}{(1+t^2)^{3/2}}
\left(
-x(z-t)+\frac{t}{\sqrt{1+\epsilon^2t^2}}
\right)=0.
\end{equation*}
Consequently, away from $x_c$, the chain rule gives
\begin{equation*}
\frac{dH}{dx}=\frac{\partial H}{\partial x}+\frac{\partial H}{\partial t}\frac{dt}{dx}=-e^2\frac{1+zt}{\sqrt{1+t^2}}
\end{equation*}
Differentiating once more along the selected branch gives
\begin{equation*}
\frac{d^2H}{dx^2}
=
\frac{d}{dt}
\left(-e^2\frac{1+zt}{\sqrt{1+t^2}}\right)
\frac{dt}{dx}
=
\frac{-e^2(z-t)/(1+t^2)^{3/2}}{dx/dt}
\end{equation*}
where
\begin{equation*}
\frac{dx}{dt}
=
\frac{z+\epsilon^2t^3}{(t-z)^2(1+\epsilon^2t^2)^{3/2}}
\end{equation*}
As $t\to t_c$, the denominator tends to zero, whereas the numerator tends to
\begin{equation*}
-e^2\frac{z-t_c}{(1+t_c^2)^{3/2}}\neq0
\end{equation*}
since $z>0$ and $t_c<0$. Consequently, $\left|d^2H/dx^2\right|\to\infty$
as $x\to x_c$ along the selected branch. Therefore, $H$ is not analytic at $x_c$, and the evolute remains the convergence boundary.
Note that $dH/dx$ remains finite as $x\to x_c$, so the square-root terms of the two components cancel and the leading singular term of $H$ is of order $(x-x_c)^{3/2}$.

Absolute convergence on the evolute boundary follows from the already established absolute convergence of the two component series in
\begin{equation*}
H=\frac{\rho}{a}\frac{\sin(\psi)}{\sin(\phi)}-\frac{\epsilon^2N}{a}
\end{equation*}

At the endpoint $\psi=0$, the series for $\sin(\psi)/\sin(\phi)$ vanishes identically, while the series for $\epsilon^2N/a=\epsilon\sqrt{1-e^2x^2}$ converges absolutely up to and including its extended boundary $\rho=ae$. Hence the series for $h$ likewise converges absolutely for
\begin{equation*}
\rho\leq ae
\end{equation*}

At the other endpoint, $\psi=\pi/2$, one has
\begin{equation*}
\frac{\sin(\psi)}{\sin(\phi)}=1,
\qquad
\frac{\epsilon^2N}{a}=\epsilon
\end{equation*}
so
\begin{equation*}
\frac{h}{a}=\frac{\rho}{a}-\epsilon
\end{equation*}
The series is therefore affine in $\rho$ and converges for all $\rho$.

Thus, for $0<\psi<\pi/2$, the evolute remains the convergence boundary for the series for $h$, and the series converges absolutely up to and including that boundary. At $\psi=0$, its convergence region extends to $\rho\leq ae$, although beyond $\rho=ae^2$ the sum represents the analytic continuation of the expanded branch and not the closest point, see Section~\ref{sec:validity_beyond}. At $\psi=\pi/2$ it converges for all $\rho$.

\subsection{Validity of convergent series beyond the evolute}\label{sec:validity_beyond}
The preceding subsections determined the regions of convergence for the various series. However, convergence of a transformed series beyond the evolute does not necessarily imply that its sum represents the geometrically defined closest-point quantity.

This distinction is particularly clear at the endpoint $\psi=0$. In this direction
\begin{equation*}
x=\frac{\rho}{ae^2}
\end{equation*}
and the evolute is reached at
\begin{equation*}
\rho=ae^2,\qquad x=1
\end{equation*}
For $\rho>ae^2$, the geometrically selected closest point lies at the end of the major axis, so
\begin{equation*}
\phi_{\mathrm{cp}}=0,
\qquad
\cos(\phi_{\mathrm{cp}})=1,
\qquad
\frac{\epsilon^2N_{\mathrm{cp}}}{a}=\epsilon^2,
\qquad
h_{\mathrm{cp}}=\rho-a
\end{equation*}

The series for $\cos(\phi)$, however, converges beyond the evolute for $\rho<ae$ and sums to
\begin{equation*}
\cos(\phi)_{\mathrm{ser}}
= \frac{\epsilon x}{\sqrt{1-e^2x^2}}
=
\frac{\epsilon\rho/(ae^2)}
{\sqrt{1-\rho^2/(a^2e^2)}}
\end{equation*}
This agrees with the closest-point value at the evolute, where it equals $1$, but for
\begin{equation*}
ae^2<\rho<ae
\end{equation*}
it no longer equals $\cos(\phi_{\mathrm{cp}})=1$. Thus the convergent series represents the analytic continuation of the original branch rather than the closest-point value.

Similarly, the series for $\epsilon^2N/a$ converges up to and including $\rho=ae$ and sums to
\begin{equation*}
\left(\frac{\epsilon^2N}{a}\right)_{\mathrm{ser}}
=\epsilon\sqrt{1-e^2x^2}
=\epsilon\sqrt{1-\frac{\rho^2}{a^2e^2}}
\end{equation*}
At the evolute this equals $\epsilon^2$, as required, but beyond the evolute it differs from the closest-point value
\begin{equation*}
\left(\frac{\epsilon^2N}{a}\right)_{\mathrm{cp}}
=
\epsilon^2
\end{equation*}

At $\psi=0$, the factor $\sin(\psi)$ makes the series for
\begin{equation*}
\frac{\sin(\psi)}{\sin(\phi)}
\end{equation*}
identically zero. It therefore converges to zero for all $\rho$. This agrees with the limiting closest-point quantity up to and including the evolute. Beyond the evolute, however, taking the closest-point solution first and then approaching the axis gives
\begin{equation*}
\lim_{\psi\to0}\frac{\sin(\psi)}{\sin(\phi_{\mathrm{cp}})}
=
1-\frac{ae^2}{\rho},\qquad\rho>ae^2
\end{equation*}
which is not zero. Thus the identically zero endpoint series also ceases to represent the closest-point continuation beyond the evolute.

The same phenomenon consequently occurs for $h$. At $\psi=0$, its series reduces to
\begin{equation*}
h_{\mathrm{ser}}
=
-a\epsilon\sqrt{1-\frac{\rho^2}{a^2e^2}},\qquad 0\leq\rho\leq ae
\end{equation*}
At the evolute,
\begin{equation*}
h_{\mathrm{ser}}(ae^2)
= -a\epsilon^2
= ae^2-a
= h_{\mathrm{cp}}(ae^2)
\end{equation*}
so the two branches join continuously. For
\begin{equation*}
ae^2<\rho\leq ae
\end{equation*}
however,
\begin{equation*}
h_{\mathrm{ser}}\neq h_{\mathrm{cp}}=\rho-a
\end{equation*}
Thus the extended convergence of the series for $h$ at $\psi=0$ is an analytic continuation and not an extension of its closest-point validity.

The endpoint $\psi=\pi/2$ is different. By symmetry, the closest-point normal remains aligned with the $y$-axis, and therefore
\begin{equation*}
\phi-\frac{\pi}{2}=0,
\qquad
\sin(\phi)-1=0,
\qquad
\cos(\phi)=0
\end{equation*}
for all $\rho$. The corresponding zero series consequently remain both convergent and equal to the closest-point quantities beyond the evolute. Likewise,
\begin{equation*}
h=\rho-a\epsilon
\end{equation*}
on this axis, so the affine series for $h$ remains the closest-point value for all $\rho$.

Quantities containing a factor $1/\cos(\psi)$ require a different interpretation at $\psi=\pi/2$, since their direct expressions are of the indeterminate form $0/0$. They may instead be assigned values by continuous transverse limits. For example,
\begin{equation*}
\lim_{\psi\to\pi/2}\frac{\cos(\phi)}{\cos(\psi)}
=
\frac{\epsilon\rho}{ae^2+\epsilon\rho}
\end{equation*}
and, since $\cos(\phi)\sim-(\phi-\pi/2)$,
\begin{equation*}
\lim_{\psi\to\pi/2}\frac{\phi-\pi/2}{\cos(\psi)}
=
-\frac{\epsilon\rho}{ae^2+\epsilon\rho}
\end{equation*}
These are not direct values of the original quotients at $\psi=\pi/2$, but continuously extended limiting quantities. Their Taylor series about $\rho=0$ have radius of convergence
\begin{equation*}
\frac{ae^2}{\epsilon}
\end{equation*}
and diverge at the evolute boundary even though the limiting functions themselves remain finite there.

Consequently, convergence and closest-point validity must be distinguished. On the major axis, several transformed series converge beyond the evolute but cease to represent the closest-point quantities there. On the minor axis, the constant series and the affine series for $h$ remain equal to their closest-point values for all $\rho$. The normalized angular quantities are interpreted through continuous transverse limits, whose Taylor series converge only strictly inside the evolute bound. The regions of convergence and the regions in which the series converge to the closest-point quantities are summarized in Tables~\ref{tab:convergence_regions} and~\ref{tab:closest_point_validity}, respectively.

\begin{table}[ht]
\centering
\caption{Regions of convergence of the series.
Here $\rho\geq0$, and $\rho_c$ denotes the radial evolute
bound. All included boundaries have absolute convergence.}
\small
\setlength{\tabcolsep}{4pt}
\label{tab:convergence_regions}
\begin{tabular}{c|c|c|c|c|c|c}
\hline
& $\phi-\pi/2$
& $(\phi\!-\!\pi/2)/\!\cos(\psi)$
& $\sin(\phi)\!-\!1$
& $\cos(\phi)$
& $\cos(\phi)/\!\cos(\psi)$
& $h$ \\
\hline
$0\!<\!\psi\!<\!\pi/2$
& $\rho\leq\rho_c$
& $\rho\leq\rho_c$
& $\rho\leq\rho_c$
& $\rho\leq\rho_c$
& $\rho\leq\rho_c$
& $\rho\leq\rho_c$
\\[1mm]
$\psi=0$
& $\rho\leq ae^2$
& $\rho\leq ae^2$
& $\rho\leq ae^2$
& $\rho<ae$
& $\rho<ae$
& $\rho\leq ae$
\\[1mm]
$\psi=\pi/2$
& $\text{all }\rho$
& $\rho<ae^2/\epsilon$
& $\text{all }\rho$
& $\text{all }\rho$
& $\rho<ae^2/\epsilon$
& $\text{all }\rho$
\\
\hline
\end{tabular}
\end{table}

\begin{table}[ht]
\centering
\caption{Regions in which the series converge to the closest-point quantities, with the upper nearest point selected on the major-axis segment. Here $\rho\geq0$. At $\psi=\pi/2$, quotients containing $1/\cos(\psi)$ are interpreted through continuous transverse limits. 
}
\small
\setlength{\tabcolsep}{4pt}
\label{tab:closest_point_validity}
\begin{tabular}{c|c|c|c|c|c|c}
\hline
& $\phi-\pi/2$
& $(\phi\!-\!\pi/2)/\!\cos(\psi)$
& $\sin(\phi)\!-\!1$
& $\cos(\phi)$
& $\cos(\phi)/\!\cos(\psi)$
& $h$ \\
\hline
$0\!<\!\psi\!<\!\pi/2$
& $\rho\leq\rho_c$
& $\rho\leq\rho_c$
& $\rho\leq\rho_c$
& $\rho\leq\rho_c$
& $\rho\leq\rho_c$
& $\rho\leq\rho_c$
\\[1mm]
$\psi=0$
& $\rho\leq ae^2$
& $\rho\leq ae^2$
& $\rho\leq ae^2$
& $\rho\leq ae^2$
& $\rho\leq ae^2$
& $\rho\leq ae^2$
\\[1mm]
$\psi=\pi/2$
& $\text{all }\rho$
& $\rho<ae^2/\epsilon$
& $\text{all }\rho$
& $\text{all }\rho$
& $\rho<ae^2/\epsilon$
& $\text{all }\rho$
\\
\hline
\end{tabular}
\end{table}

\newpage
\section{Tabulated coefficients}\label{sec::tabulated_coefficients}
The rational series expansion coefficients, i.e. \eqref{eq:d_phi}, \eqref{eq:d_h}, \eqref{eq:d_sin}, \eqref{eq:d_cos}, are tabulated in Tables~\ref{tab:d_phi_evo}--\ref{tab:d_cos_phi_evo}. 

\begin{table}
\scriptsize
\begin{tabular}{@{\hspace{0mm}} l@{\hspace{0mm}} c@{\hspace{0.5mm}} | @{\hspace{0.5mm}} l@{}  l@{}  l@{}  l@{}  l@{}  l@{} }
$k$ & $l\setminus n$ & $\phantom{{-}}\epsilon^{1}$ & $\phantom{{-}}\epsilon^{3}$ & $\phantom{{-}}\epsilon^{5}$ & $\phantom{{-}}\epsilon^{7}$ & $\phantom{{-}}\epsilon^{9}$ & $\phantom{{-}}\epsilon^{11}$ \\ \hline\rule{0pt}{\dimexpr\ht\strutbox+0.8mm\relax}
 $\sigma^{1}$ & 1  & $-1$  & $\phantom{-}\times$  & $\phantom{-}\times$  & $\phantom{-}\times$  & $\phantom{-}\times$  & $\phantom{-}\times$  \\[0.4mm]
\noalign{\vskip 1.5mm} & & $\phantom{{-}}\epsilon^{2}$ & $\phantom{{-}}\epsilon^{4}$ & $\phantom{{-}}\epsilon^{6}$ & $\phantom{{-}}\epsilon^{8}$ & $\phantom{{-}}\epsilon^{10}$ & $\phantom{{-}}\epsilon^{12}$ \\ \hline\rule{0pt}{\dimexpr\ht\strutbox+0.8mm\relax}
 $\sigma^{2}$ & $|\sin(\psi)|^{1}$  & $\phantom{-}1$  & $\phantom{-}\times$  & $\phantom{-}\times$  & $\phantom{-}\times$  & $\phantom{-}\times$  & $\phantom{-}\times$  \\[0.4mm]
\noalign{\vskip 1.5mm} & & $\phantom{{-}}\epsilon^{1}$ & $\phantom{{-}}\epsilon^{3}$ & $\phantom{{-}}\epsilon^{5}$ & $\phantom{{-}}\epsilon^{7}$ & $\phantom{{-}}\epsilon^{9}$ & $\phantom{{-}}\epsilon^{11}$ \\ \hline\rule{0pt}{\dimexpr\ht\strutbox+0.8mm\relax}
 $\sigma^{3}$ & 1  & $-\frac{1}{2}$  & $\phantom{-}\frac{1}{3}$  & $\phantom{-}\times$  & $\phantom{-}\times$  & $\phantom{-}\times$  & $\phantom{-}\times$  \\[0.3mm]
 & $|\sin(\psi)|^{2}$  & $\phantom{-}\frac{1}{2}$  & $-\frac{4}{3}$  & $\phantom{-}\times$  & $\phantom{-}\times$  & $\phantom{-}\times$  & $\phantom{-}\times$  \\[0.4mm]
\noalign{\vskip 1.5mm} & & $\phantom{{-}}\epsilon^{2}$ & $\phantom{{-}}\epsilon^{4}$ & $\phantom{{-}}\epsilon^{6}$ & $\phantom{{-}}\epsilon^{8}$ & $\phantom{{-}}\epsilon^{10}$ & $\phantom{{-}}\epsilon^{12}$ \\ \hline\rule{0pt}{\dimexpr\ht\strutbox+0.8mm\relax}
 $\sigma^{4}$ & $|\sin(\psi)|^{1}$  & $\phantom{-}2$  & $-1$  & $\phantom{-}\times$  & $\phantom{-}\times$  & $\phantom{-}\times$  & $\phantom{-}\times$  \\[0.3mm]
 & $|\sin(\psi)|^{3}$  & $-2$  & $\phantom{-}2$  & $\phantom{-}\times$  & $\phantom{-}\times$  & $\phantom{-}\times$  & $\phantom{-}\times$  \\[0.4mm]
\noalign{\vskip 1.5mm} & & $\phantom{{-}}\epsilon^{1}$ & $\phantom{{-}}\epsilon^{3}$ & $\phantom{{-}}\epsilon^{5}$ & $\phantom{{-}}\epsilon^{7}$ & $\phantom{{-}}\epsilon^{9}$ & $\phantom{{-}}\epsilon^{11}$ \\ \hline\rule{0pt}{\dimexpr\ht\strutbox+0.8mm\relax}
 $\sigma^{5}$ & 1  & $-\frac{3}{8}$  & $\phantom{-}\frac{1}{2}$  & $-\frac{1}{5}$  & $\phantom{-}\times$  & $\phantom{-}\times$  & $\phantom{-}\times$  \\[0.3mm]
 & $|\sin(\psi)|^{2}$  & $\phantom{-}\frac{3}{4}$  & $-6$  & $\phantom{-}\frac{12}{5}$  & $\phantom{-}\times$  & $\phantom{-}\times$  & $\phantom{-}\times$  \\[0.3mm]
 & $|\sin(\psi)|^{4}$  & $-\frac{3}{8}$  & $\phantom{-}\frac{11}{2}$  & $-\frac{16}{5}$  & $\phantom{-}\times$  & $\phantom{-}\times$  & $\phantom{-}\times$  \\[0.4mm]
\noalign{\vskip 1.5mm} & & $\phantom{{-}}\epsilon^{2}$ & $\phantom{{-}}\epsilon^{4}$ & $\phantom{{-}}\epsilon^{6}$ & $\phantom{{-}}\epsilon^{8}$ & $\phantom{{-}}\epsilon^{10}$ & $\phantom{{-}}\epsilon^{12}$ \\ \hline\rule{0pt}{\dimexpr\ht\strutbox+0.8mm\relax}
 $\sigma^{6}$ & $|\sin(\psi)|^{1}$  & $\phantom{-}3$  & $-3$  & $\phantom{-}1$  & $\phantom{-}\times$  & $\phantom{-}\times$  & $\phantom{-}\times$  \\[0.3mm]
 & $|\sin(\psi)|^{3}$  & $-6$  & $\phantom{-}16$  & $-\frac{16}{3}$  & $\phantom{-}\times$  & $\phantom{-}\times$  & $\phantom{-}\times$  \\[0.3mm]
 & $|\sin(\psi)|^{5}$  & $\phantom{-}3$  & $-13$  & $\phantom{-}\frac{16}{3}$  & $\phantom{-}\times$  & $\phantom{-}\times$  & $\phantom{-}\times$  \\[0.4mm]
\noalign{\vskip 1.5mm} & & $\phantom{{-}}\epsilon^{1}$ & $\phantom{{-}}\epsilon^{3}$ & $\phantom{{-}}\epsilon^{5}$ & $\phantom{{-}}\epsilon^{7}$ & $\phantom{{-}}\epsilon^{9}$ & $\phantom{{-}}\epsilon^{11}$ \\ \hline\rule{0pt}{\dimexpr\ht\strutbox+0.8mm\relax}
 $\sigma^{7}$ & 1  & $-\frac{5}{16}$  & $\phantom{-}\frac{5}{8}$  & $-\frac{1}{2}$  & $\phantom{-}\frac{1}{7}$  & $\phantom{-}\times$  & $\phantom{-}\times$  \\[0.3mm]
 & $|\sin(\psi)|^{2}$  & $\phantom{-}\frac{15}{16}$  & $-15$  & $\phantom{-}12$  & $-\frac{24}{7}$  & $\phantom{-}\times$  & $\phantom{-}\times$  \\[0.3mm]
 & $|\sin(\psi)|^{4}$  & $-\frac{15}{16}$  & $\phantom{-}\frac{225}{8}$  & $-40$  & $\phantom{-}\frac{80}{7}$  & $\phantom{-}\times$  & $\phantom{-}\times$  \\[0.3mm]
 & $|\sin(\psi)|^{6}$  & $\phantom{-}\frac{5}{16}$  & $-\frac{55}{4}$  & $\phantom{-}\frac{57}{2}$  & $-\frac{64}{7}$  & $\phantom{-}\times$  & $\phantom{-}\times$  \\[0.4mm]
\noalign{\vskip 1.5mm} & & $\phantom{{-}}\epsilon^{2}$ & $\phantom{{-}}\epsilon^{4}$ & $\phantom{{-}}\epsilon^{6}$ & $\phantom{{-}}\epsilon^{8}$ & $\phantom{{-}}\epsilon^{10}$ & $\phantom{{-}}\epsilon^{12}$ \\ \hline\rule{0pt}{\dimexpr\ht\strutbox+0.8mm\relax}
 $\sigma^{8}$ & $|\sin(\psi)|^{1}$  & $\phantom{-}4$  & $-6$  & $\phantom{-}4$  & $-1$  & $\phantom{-}\times$  & $\phantom{-}\times$  \\[0.3mm]
 & $|\sin(\psi)|^{3}$  & $-12$  & $\phantom{-}60$  & $-40$  & $\phantom{-}10$  & $\phantom{-}\times$  & $\phantom{-}\times$  \\[0.3mm]
 & $|\sin(\psi)|^{5}$  & $\phantom{-}12$  & $-102$  & $\phantom{-}96$  & $-24$  & $\phantom{-}\times$  & $\phantom{-}\times$  \\[0.3mm]
 & $|\sin(\psi)|^{7}$  & $-4$  & $\phantom{-}48$  & $-60$  & $\phantom{-}16$  & $\phantom{-}\times$  & $\phantom{-}\times$  \\[0.4mm]
\noalign{\vskip 1.5mm} & & $\phantom{{-}}\epsilon^{1}$ & $\phantom{{-}}\epsilon^{3}$ & $\phantom{{-}}\epsilon^{5}$ & $\phantom{{-}}\epsilon^{7}$ & $\phantom{{-}}\epsilon^{9}$ & $\phantom{{-}}\epsilon^{11}$ \\ \hline\rule{0pt}{\dimexpr\ht\strutbox+0.8mm\relax}
 $\sigma^{9}$ & 1  & $-\frac{35}{128}$  & $\phantom{-}\frac{35}{48}$  & $-\frac{7}{8}$  & $\phantom{-}\frac{1}{2}$  & $-\frac{1}{9}$  & $\phantom{-}\times$  \\[0.3mm]
 & $|\sin(\psi)|^{2}$  & $\phantom{-}\frac{35}{32}$  & $-\frac{175}{6}$  & $\phantom{-}35$  & $-20$  & $\phantom{-}\frac{40}{9}$  & $\phantom{-}\times$  \\[0.3mm]
 & $|\sin(\psi)|^{4}$  & $-\frac{105}{64}$  & $\phantom{-}\frac{665}{8}$  & $-210$  & $\phantom{-}120$  & $-\frac{80}{3}$  & $\phantom{-}\times$  \\[0.3mm]
 & $|\sin(\psi)|^{6}$  & $\phantom{-}\frac{35}{32}$  & $-\frac{245}{3}$  & $\phantom{-}\frac{637}{2}$  & $-224$  & $\phantom{-}\frac{448}{9}$  & $\phantom{-}\times$  \\[0.3mm]
 & $|\sin(\psi)|^{8}$  & $-\frac{35}{128}$  & $\phantom{-}\frac{1295}{48}$  & $-\frac{1141}{8}$  & $\phantom{-}\frac{247}{2}$  & $-\frac{256}{9}$  & $\phantom{-}\times$  \\[0.4mm]
\noalign{\vskip 1.5mm} & & $\phantom{{-}}\epsilon^{2}$ & $\phantom{{-}}\epsilon^{4}$ & $\phantom{{-}}\epsilon^{6}$ & $\phantom{{-}}\epsilon^{8}$ & $\phantom{{-}}\epsilon^{10}$ & $\phantom{{-}}\epsilon^{12}$ \\ \hline\rule{0pt}{\dimexpr\ht\strutbox+0.8mm\relax}
 $\sigma^{10}$ & $|\sin(\psi)|^{1}$  & $\phantom{-}5$  & $-10$  & $\phantom{-}10$  & $-5$  & $\phantom{-}1$  & $\phantom{-}\times$  \\[0.3mm]
 & $|\sin(\psi)|^{3}$  & $-20$  & $\phantom{-}160$  & $-160$  & $\phantom{-}80$  & $-16$  & $\phantom{-}\times$  \\[0.3mm]
 & $|\sin(\psi)|^{5}$  & $\phantom{-}30$  & $-420$  & $\phantom{-}672$  & $-336$  & $\phantom{-}\frac{336}{5}$  & $\phantom{-}\times$  \\[0.3mm]
 & $|\sin(\psi)|^{7}$  & $-20$  & $\phantom{-}400$  & $-904$  & $\phantom{-}512$  & $-\frac{512}{5}$  & $\phantom{-}\times$  \\[0.3mm]
 & $|\sin(\psi)|^{9}$  & $\phantom{-}5$  & $-130$  & $\phantom{-}382$  & $-251$  & $\phantom{-}\frac{256}{5}$  & $\phantom{-}\times$  \\[0.4mm]
\noalign{\vskip 1.5mm} & & $\phantom{{-}}\epsilon^{1}$ & $\phantom{{-}}\epsilon^{3}$ & $\phantom{{-}}\epsilon^{5}$ & $\phantom{{-}}\epsilon^{7}$ & $\phantom{{-}}\epsilon^{9}$ & $\phantom{{-}}\epsilon^{11}$ \\ \hline\rule{0pt}{\dimexpr\ht\strutbox+0.8mm\relax}
 $\sigma^{11}$ & 1  & $-\frac{63}{256}$  & $\phantom{-}\frac{105}{128}$  & $-\frac{21}{16}$  & $\phantom{-}\frac{9}{8}$  & $-\frac{1}{2}$  & $\phantom{-}\frac{1}{11}$  \\[0.3mm]
 & $|\sin(\psi)|^{2}$  & $\phantom{-}\frac{315}{256}$  & $-\frac{1575}{32}$  & $\phantom{-}\frac{315}{4}$  & $-\frac{135}{2}$  & $\phantom{-}30$  & $-\frac{60}{11}$  \\[0.3mm]
 & $|\sin(\psi)|^{4}$  & $-\frac{315}{128}$  & $\phantom{-}\frac{12075}{64}$  & $-735$  & $\phantom{-}630$  & $-280$  & $\phantom{-}\frac{560}{11}$  \\[0.3mm]
 & $|\sin(\psi)|^{6}$  & $\phantom{-}\frac{315}{128}$  & $-\frac{8925}{32}$  & $\phantom{-}\frac{13965}{8}$  & $-2016$  & $\phantom{-}896$  & $-\frac{1792}{11}$  \\[0.3mm]
 & $|\sin(\psi)|^{8}$  & $-\frac{315}{256}$  & $\phantom{-}\frac{23625}{128}$  & $-\frac{25515}{16}$  & $\phantom{-}\frac{19251}{8}$  & $-1152$  & $\phantom{-}\frac{2304}{11}$  \\[0.3mm]
 & $|\sin(\psi)|^{10}$  & $\phantom{-}\frac{63}{256}$  & $-\frac{735}{16}$  & $\phantom{-}\frac{4053}{8}$  & $-954$  & $\phantom{-}\frac{1013}{2}$  & $-\frac{1024}{11}$  \\[0.4mm]
\noalign{\vskip 1.5mm} & & $\phantom{{-}}\epsilon^{2}$ & $\phantom{{-}}\epsilon^{4}$ & $\phantom{{-}}\epsilon^{6}$ & $\phantom{{-}}\epsilon^{8}$ & $\phantom{{-}}\epsilon^{10}$ & $\phantom{{-}}\epsilon^{12}$ \\ \hline\rule{0pt}{\dimexpr\ht\strutbox+0.8mm\relax}
 $\sigma^{12}$ & $|\sin(\psi)|^{1}$  & $\phantom{-}6$  & $-15$  & $\phantom{-}20$  & $-15$  & $\phantom{-}6$  & $-1$  \\[0.3mm]
 & $|\sin(\psi)|^{3}$  & $-30$  & $\phantom{-}350$  & $-\frac{1400}{3}$  & $\phantom{-}350$  & $-140$  & $\phantom{-}\frac{70}{3}$  \\[0.3mm]
 & $|\sin(\psi)|^{5}$  & $\phantom{-}60$  & $-1250$  & $\phantom{-}\frac{8960}{3}$  & $-2240$  & $\phantom{-}896$  & $-\frac{448}{3}$  \\[0.3mm]
 & $|\sin(\psi)|^{7}$  & $-60$  & $\phantom{-}1800$  & $-6360$  & $\phantom{-}5760$  & $-2304$  & $\phantom{-}384$  \\[0.3mm]
 & $|\sin(\psi)|^{9}$  & $\phantom{-}30$  & $-1175$  & $\phantom{-}\frac{16580}{3}$  & $-6125$  & $\phantom{-}2560$  & $-\frac{1280}{3}$  \\[0.3mm]
 & $|\sin(\psi)|^{11}$  & $-6$  & $\phantom{-}290$  & $-\frac{5120}{3}$  & $\phantom{-}2270$  & $-1018$  & $\phantom{-}\frac{512}{3}$  \\[0.4mm]
\end{tabular}
\captionsetup{width=0.87\linewidth}
\caption{Tabulated coefficient values for series expansion of $(\phi-\sgn(\psi)\pi/2)/(\sgn(\psi)\cos(\psi))$ as in Theorem 1. Top rows and left-most columns show series and polynomial variables and the table body contains the corresponding values of $d^{\phi}_{k,l,n}$ for indicated indices. Values marked with $\times$ are not part of the expansion.}
\label{tab:d_phi_evo}
\end{table}

\begin{table}
\scriptsize
\begin{tabular}{@{\hspace{0mm}} l@{\hspace{0mm}} c@{\hspace{0.5mm}} | @{\hspace{0.5mm}} l@{}  l@{}  l@{}  l@{}  l@{}  l@{} }
$k$ & $l\setminus n$ & $\phantom{{-}}\epsilon^{1}$ & $\phantom{{-}}\epsilon^{3}$ & $\phantom{{-}}\epsilon^{5}$ & $\phantom{{-}}\epsilon^{7}$ & $\phantom{{-}}\epsilon^{9}$ & $\phantom{{-}}\epsilon^{11}$ \\ \hline\rule{0pt}{\dimexpr\ht\strutbox+0.8mm\relax}
 $\sigma^{2}$ & 1  & $\phantom{-}\frac{1}{2}$  & $-\frac{1}{2}$  & $\phantom{-}\times$  & $\phantom{-}\times$  & $\phantom{-}\times$  & $\phantom{-}\times$  \\[0.3mm]
 & $|\sin(\psi)|^{2}$  & $-\frac{1}{2}$  & $\phantom{-}\frac{1}{2}$  & $\phantom{-}\times$  & $\phantom{-}\times$  & $\phantom{-}\times$  & $\phantom{-}\times$  \\[0.4mm]
\noalign{\vskip 1.5mm} & & $\phantom{{-}}\epsilon^{2}$ & $\phantom{{-}}\epsilon^{4}$ & $\phantom{{-}}\epsilon^{6}$ & $\phantom{{-}}\epsilon^{8}$ & $\phantom{{-}}\epsilon^{10}$ & $\phantom{{-}}\epsilon^{12}$ \\ \hline\rule{0pt}{\dimexpr\ht\strutbox+0.8mm\relax}
 $\sigma^{3}$ & $|\sin(\psi)|^{1}$  & $-\frac{1}{2}$  & $\phantom{-}\frac{1}{2}$  & $\phantom{-}\times$  & $\phantom{-}\times$  & $\phantom{-}\times$  & $\phantom{-}\times$  \\[0.3mm]
 & $|\sin(\psi)|^{3}$  & $\phantom{-}\frac{1}{2}$  & $-\frac{1}{2}$  & $\phantom{-}\times$  & $\phantom{-}\times$  & $\phantom{-}\times$  & $\phantom{-}\times$  \\[0.4mm]
\noalign{\vskip 1.5mm} & & $\phantom{{-}}\epsilon^{1}$ & $\phantom{{-}}\epsilon^{3}$ & $\phantom{{-}}\epsilon^{5}$ & $\phantom{{-}}\epsilon^{7}$ & $\phantom{{-}}\epsilon^{9}$ & $\phantom{{-}}\epsilon^{11}$ \\ \hline\rule{0pt}{\dimexpr\ht\strutbox+0.8mm\relax}
 $\sigma^{4}$ & 1  & $\phantom{-}\frac{1}{8}$  & $-\frac{1}{4}$  & $\phantom{-}\frac{1}{8}$  & $\phantom{-}\times$  & $\phantom{-}\times$  & $\phantom{-}\times$  \\[0.3mm]
 & $|\sin(\psi)|^{2}$  & $-\frac{1}{4}$  & $\phantom{-}1$  & $-\frac{3}{4}$  & $\phantom{-}\times$  & $\phantom{-}\times$  & $\phantom{-}\times$  \\[0.3mm]
 & $|\sin(\psi)|^{4}$  & $\phantom{-}\frac{1}{8}$  & $-\frac{3}{4}$  & $\phantom{-}\frac{5}{8}$  & $\phantom{-}\times$  & $\phantom{-}\times$  & $\phantom{-}\times$  \\[0.4mm]
\noalign{\vskip 1.5mm} & & $\phantom{{-}}\epsilon^{2}$ & $\phantom{{-}}\epsilon^{4}$ & $\phantom{{-}}\epsilon^{6}$ & $\phantom{{-}}\epsilon^{8}$ & $\phantom{{-}}\epsilon^{10}$ & $\phantom{{-}}\epsilon^{12}$ \\ \hline\rule{0pt}{\dimexpr\ht\strutbox+0.8mm\relax}
 $\sigma^{5}$ & $|\sin(\psi)|^{1}$  & $-\frac{1}{2}$  & $\phantom{-}\frac{7}{8}$  & $-\frac{3}{8}$  & $\phantom{-}\times$  & $\phantom{-}\times$  & $\phantom{-}\times$  \\[0.3mm]
 & $|\sin(\psi)|^{3}$  & $\phantom{-}1$  & $-\frac{9}{4}$  & $\phantom{-}\frac{5}{4}$  & $\phantom{-}\times$  & $\phantom{-}\times$  & $\phantom{-}\times$  \\[0.3mm]
 & $|\sin(\psi)|^{5}$  & $-\frac{1}{2}$  & $\phantom{-}\frac{11}{8}$  & $-\frac{7}{8}$  & $\phantom{-}\times$  & $\phantom{-}\times$  & $\phantom{-}\times$  \\[0.4mm]
\noalign{\vskip 1.5mm} & & $\phantom{{-}}\epsilon^{1}$ & $\phantom{{-}}\epsilon^{3}$ & $\phantom{{-}}\epsilon^{5}$ & $\phantom{{-}}\epsilon^{7}$ & $\phantom{{-}}\epsilon^{9}$ & $\phantom{{-}}\epsilon^{11}$ \\ \hline\rule{0pt}{\dimexpr\ht\strutbox+0.8mm\relax}
 $\sigma^{6}$ & 1  & $\phantom{-}\frac{1}{16}$  & $-\frac{3}{16}$  & $\phantom{-}\frac{3}{16}$  & $-\frac{1}{16}$  & $\phantom{-}\times$  & $\phantom{-}\times$  \\[0.3mm]
 & $|\sin(\psi)|^{2}$  & $-\frac{3}{16}$  & $\phantom{-}\frac{29}{16}$  & $-\frac{41}{16}$  & $\phantom{-}\frac{15}{16}$  & $\phantom{-}\times$  & $\phantom{-}\times$  \\[0.3mm]
 & $|\sin(\psi)|^{4}$  & $\phantom{-}\frac{3}{16}$  & $-\frac{49}{16}$  & $\phantom{-}\frac{81}{16}$  & $-\frac{35}{16}$  & $\phantom{-}\times$  & $\phantom{-}\times$  \\[0.3mm]
 & $|\sin(\psi)|^{6}$  & $-\frac{1}{16}$  & $\phantom{-}\frac{23}{16}$  & $-\frac{43}{16}$  & $\phantom{-}\frac{21}{16}$  & $\phantom{-}\times$  & $\phantom{-}\times$  \\[0.4mm]
\noalign{\vskip 1.5mm} & & $\phantom{{-}}\epsilon^{2}$ & $\phantom{{-}}\epsilon^{4}$ & $\phantom{{-}}\epsilon^{6}$ & $\phantom{{-}}\epsilon^{8}$ & $\phantom{{-}}\epsilon^{10}$ & $\phantom{{-}}\epsilon^{12}$ \\ \hline\rule{0pt}{\dimexpr\ht\strutbox+0.8mm\relax}
 $\sigma^{7}$ & $|\sin(\psi)|^{1}$  & $-\frac{1}{2}$  & $\phantom{-}\frac{5}{4}$  & $-\frac{17}{16}$  & $\phantom{-}\frac{5}{16}$  & $\phantom{-}\times$  & $\phantom{-}\times$  \\[0.3mm]
 & $|\sin(\psi)|^{3}$  & $\phantom{-}\frac{3}{2}$  & $-\frac{25}{4}$  & $\phantom{-}\frac{111}{16}$  & $-\frac{35}{16}$  & $\phantom{-}\times$  & $\phantom{-}\times$  \\[0.3mm]
 & $|\sin(\psi)|^{5}$  & $-\frac{3}{2}$  & $\phantom{-}\frac{35}{4}$  & $-\frac{179}{16}$  & $\phantom{-}\frac{63}{16}$  & $\phantom{-}\times$  & $\phantom{-}\times$  \\[0.3mm]
 & $|\sin(\psi)|^{7}$  & $\phantom{-}\frac{1}{2}$  & $-\frac{15}{4}$  & $\phantom{-}\frac{85}{16}$  & $-\frac{33}{16}$  & $\phantom{-}\times$  & $\phantom{-}\times$  \\[0.4mm]
\noalign{\vskip 1.5mm} & & $\phantom{{-}}\epsilon^{1}$ & $\phantom{{-}}\epsilon^{3}$ & $\phantom{{-}}\epsilon^{5}$ & $\phantom{{-}}\epsilon^{7}$ & $\phantom{{-}}\epsilon^{9}$ & $\phantom{{-}}\epsilon^{11}$ \\ \hline\rule{0pt}{\dimexpr\ht\strutbox+0.8mm\relax}
 $\sigma^{8}$ & 1  & $\phantom{-}\frac{5}{128}$  & $-\frac{5}{32}$  & $\phantom{-}\frac{15}{64}$  & $-\frac{5}{32}$  & $\phantom{-}\frac{5}{128}$  & $\phantom{-}\times$  \\[0.3mm]
 & $|\sin(\psi)|^{2}$  & $-\frac{5}{32}$  & $\phantom{-}\frac{45}{16}$  & $-\frac{23}{4}$  & $\phantom{-}\frac{67}{16}$  & $-\frac{35}{32}$  & $\phantom{-}\times$  \\[0.3mm]
 & $|\sin(\psi)|^{4}$  & $\phantom{-}\frac{15}{64}$  & $-\frac{15}{2}$  & $\phantom{-}\frac{647}{32}$  & $-\frac{143}{8}$  & $\phantom{-}\frac{315}{64}$  & $\phantom{-}\times$  \\[0.3mm]
 & $|\sin(\psi)|^{6}$  & $-\frac{5}{32}$  & $\phantom{-}\frac{115}{16}$  & $-\frac{193}{8}$  & $\phantom{-}\frac{389}{16}$  & $-\frac{231}{32}$  & $\phantom{-}\times$  \\[0.3mm]
 & $|\sin(\psi)|^{8}$  & $\phantom{-}\frac{5}{128}$  & $-\frac{75}{32}$  & $\phantom{-}\frac{603}{64}$  & $-\frac{335}{32}$  & $\phantom{-}\frac{429}{128}$  & $\phantom{-}\times$  \\[0.4mm]
\noalign{\vskip 1.5mm} & & $\phantom{{-}}\epsilon^{2}$ & $\phantom{{-}}\epsilon^{4}$ & $\phantom{{-}}\epsilon^{6}$ & $\phantom{{-}}\epsilon^{8}$ & $\phantom{{-}}\epsilon^{10}$ & $\phantom{{-}}\epsilon^{12}$ \\ \hline\rule{0pt}{\dimexpr\ht\strutbox+0.8mm\relax}
 $\sigma^{9}$ & $|\sin(\psi)|^{1}$  & $-\frac{1}{2}$  & $\phantom{-}\frac{13}{8}$  & $-\frac{33}{16}$  & $\phantom{-}\frac{155}{128}$  & $-\frac{35}{128}$  & $\phantom{-}\times$  \\[0.3mm]
 & $|\sin(\psi)|^{3}$  & $\phantom{-}2$  & $-\frac{27}{2}$  & $\phantom{-}\frac{89}{4}$  & $-\frac{449}{32}$  & $\phantom{-}\frac{105}{32}$  & $\phantom{-}\times$  \\[0.3mm]
 & $|\sin(\psi)|^{5}$  & $-3$  & $\phantom{-}\frac{123}{4}$  & $-\frac{491}{8}$  & $\phantom{-}\frac{2845}{64}$  & $-\frac{693}{64}$  & $\phantom{-}\times$  \\[0.3mm]
 & $|\sin(\psi)|^{7}$  & $\phantom{-}2$  & $-\frac{55}{2}$  & $\phantom{-}\frac{257}{4}$  & $-\frac{1669}{32}$  & $\phantom{-}\frac{429}{32}$  & $\phantom{-}\times$  \\[0.3mm]
 & $|\sin(\psi)|^{9}$  & $-\frac{1}{2}$  & $\phantom{-}\frac{69}{8}$  & $-\frac{369}{16}$  & $\phantom{-}\frac{2627}{128}$  & $-\frac{715}{128}$  & $\phantom{-}\times$  \\[0.4mm]
\noalign{\vskip 1.5mm} & & $\phantom{{-}}\epsilon^{1}$ & $\phantom{{-}}\epsilon^{3}$ & $\phantom{{-}}\epsilon^{5}$ & $\phantom{{-}}\epsilon^{7}$ & $\phantom{{-}}\epsilon^{9}$ & $\phantom{{-}}\epsilon^{11}$ \\ \hline\rule{0pt}{\dimexpr\ht\strutbox+0.8mm\relax}
 $\sigma^{10}$ & 1  & $\phantom{-}\frac{7}{256}$  & $-\frac{35}{256}$  & $\phantom{-}\frac{35}{128}$  & $-\frac{35}{128}$  & $\phantom{-}\frac{35}{256}$  & $-\frac{7}{256}$  \\[0.3mm]
 & $|\sin(\psi)|^{2}$  & $-\frac{35}{256}$  & $\phantom{-}\frac{1015}{256}$  & $-\frac{1351}{128}$  & $\phantom{-}\frac{1471}{128}$  & $-\frac{1535}{256}$  & $\phantom{-}\frac{315}{256}$  \\[0.3mm]
 & $|\sin(\psi)|^{4}$  & $\phantom{-}\frac{35}{128}$  & $-\frac{1855}{128}$  & $\phantom{-}\frac{3703}{64}$  & $-\frac{4951}{64}$  & $\phantom{-}\frac{5471}{128}$  & $-\frac{1155}{128}$  \\[0.3mm]
 & $|\sin(\psi)|^{6}$  & $-\frac{35}{128}$  & $\phantom{-}\frac{2695}{128}$  & $-\frac{7231}{64}$  & $\phantom{-}\frac{11287}{64}$  & $-\frac{13775}{128}$  & $\phantom{-}\frac{3003}{128}$  \\[0.3mm]
 & $|\sin(\psi)|^{8}$  & $\phantom{-}\frac{35}{256}$  & $-\frac{3535}{256}$  & $\phantom{-}\frac{11935}{128}$  & $-\frac{21151}{128}$  & $\phantom{-}\frac{28367}{256}$  & $-\frac{6435}{256}$  \\[0.3mm]
 & $|\sin(\psi)|^{10}$  & $-\frac{7}{256}$  & $\phantom{-}\frac{875}{256}$  & $-\frac{3563}{128}$  & $\phantom{-}\frac{7043}{128}$  & $-\frac{10259}{256}$  & $\phantom{-}\frac{2431}{256}$  \\[0.4mm]
\noalign{\vskip 1.5mm} & & $\phantom{{-}}\epsilon^{2}$ & $\phantom{{-}}\epsilon^{4}$ & $\phantom{{-}}\epsilon^{6}$ & $\phantom{{-}}\epsilon^{8}$ & $\phantom{{-}}\epsilon^{10}$ & $\phantom{{-}}\epsilon^{12}$ \\ \hline\rule{0pt}{\dimexpr\ht\strutbox+0.8mm\relax}
 $\sigma^{11}$ & $|\sin(\psi)|^{1}$  & $-\frac{1}{2}$  & $\phantom{-}2$  & $-\frac{27}{8}$  & $\phantom{-}\frac{95}{32}$  & $-\frac{343}{256}$  & $\phantom{-}\frac{63}{256}$  \\[0.3mm]
 & $|\sin(\psi)|^{3}$  & $\phantom{-}\frac{5}{2}$  & $-25$  & $\phantom{-}\frac{435}{8}$  & $-\frac{1645}{32}$  & $\phantom{-}\frac{6155}{256}$  & $-\frac{1155}{256}$  \\[0.3mm]
 & $|\sin(\psi)|^{5}$  & $-5$  & $\phantom{-}80$  & $-\frac{903}{4}$  & $\phantom{-}\frac{3991}{16}$  & $-\frac{15635}{128}$  & $\phantom{-}\frac{3003}{128}$  \\[0.3mm]
 & $|\sin(\psi)|^{7}$  & $\phantom{-}5$  & $-110$  & $\phantom{-}\frac{1539}{4}$  & $-\frac{7753}{16}$  & $\phantom{-}\frac{32651}{128}$  & $-\frac{6435}{128}$  \\[0.3mm]
 & $|\sin(\psi)|^{9}$  & $-\frac{5}{2}$  & $\phantom{-}70$  & $-\frac{2343}{8}$  & $\phantom{-}\frac{13171}{32}$  & $-\frac{59827}{256}$  & $\phantom{-}\frac{12155}{256}$  \\[0.3mm]
 & $|\sin(\psi)|^{11}$  & $\phantom{-}\frac{1}{2}$  & $-17$  & $\phantom{-}\frac{663}{8}$  & $-\frac{4097}{32}$  & $\phantom{-}\frac{19983}{256}$  & $-\frac{4199}{256}$  \\[0.4mm]
\end{tabular}
\captionsetup{width=0.87\linewidth}
\caption{Tabulated coefficient values for series expansion of $(h+b-\rho|\sin(\psi)|)/a$ as in Theorem 1. Top rows and left-most columns show series and polynomial variables and the table body contains the corresponding values of $d^{h}_{k,l,n}$ for indicated indices. Values marked with $\times$ are not part of the expansion.}
\label{tab:d_h}
\end{table}

\begin{table}
\scriptsize
\begin{tabular}{@{\hspace{0mm}} l@{\hspace{0mm}} c@{\hspace{0.5mm}} | @{\hspace{0.5mm}} l@{}  l@{}  l@{}  l@{}  l@{} }
$k$ & $l\setminus n$ & $\phantom{{-}}\epsilon^{2}$ & $\phantom{{-}}\epsilon^{4}$ & $\phantom{{-}}\epsilon^{6}$ & $\phantom{{-}}\epsilon^{8}$ & $\phantom{{-}}\epsilon^{10}$ \\ \hline\rule{0pt}{\dimexpr\ht\strutbox+0.8mm\relax}
 $\sigma^{2}$ & 1  & $-\frac{1}{2}$  & $\phantom{-}\times$  & $\phantom{-}\times$  & $\phantom{-}\times$  & $\phantom{-}\times$  \\[0.3mm]
 & $|\sin(\psi)|^{2}$  & $\phantom{-}\frac{1}{2}$  & $\phantom{-}\times$  & $\phantom{-}\times$  & $\phantom{-}\times$  & $\phantom{-}\times$  \\[0.4mm]
\noalign{\vskip 1.5mm} & & $\phantom{{-}}\epsilon^{3}$ & $\phantom{{-}}\epsilon^{5}$ & $\phantom{{-}}\epsilon^{7}$ & $\phantom{{-}}\epsilon^{9}$ & $\phantom{{-}}\epsilon^{11}$ \\ \hline\rule{0pt}{\dimexpr\ht\strutbox+0.8mm\relax}
 $\sigma^{3}$ & $|\sin(\psi)|^{1}$  & $\phantom{-}1$  & $\phantom{-}\times$  & $\phantom{-}\times$  & $\phantom{-}\times$  & $\phantom{-}\times$  \\[0.3mm]
 & $|\sin(\psi)|^{3}$  & $-1$  & $\phantom{-}\times$  & $\phantom{-}\times$  & $\phantom{-}\times$  & $\phantom{-}\times$  \\[0.4mm]
\noalign{\vskip 1.5mm} & & $\phantom{{-}}\epsilon^{2}$ & $\phantom{{-}}\epsilon^{4}$ & $\phantom{{-}}\epsilon^{6}$ & $\phantom{{-}}\epsilon^{8}$ & $\phantom{{-}}\epsilon^{10}$ \\ \hline\rule{0pt}{\dimexpr\ht\strutbox+0.8mm\relax}
 $\sigma^{4}$ & 1  & $-\frac{1}{2}$  & $\phantom{-}\frac{3}{8}$  & $\phantom{-}\times$  & $\phantom{-}\times$  & $\phantom{-}\times$  \\[0.3mm]
 & $|\sin(\psi)|^{2}$  & $\phantom{-}1$  & $-\frac{9}{4}$  & $\phantom{-}\times$  & $\phantom{-}\times$  & $\phantom{-}\times$  \\[0.3mm]
 & $|\sin(\psi)|^{4}$  & $-\frac{1}{2}$  & $\phantom{-}\frac{15}{8}$  & $\phantom{-}\times$  & $\phantom{-}\times$  & $\phantom{-}\times$  \\[0.4mm]
\noalign{\vskip 1.5mm} & & $\phantom{{-}}\epsilon^{3}$ & $\phantom{{-}}\epsilon^{5}$ & $\phantom{{-}}\epsilon^{7}$ & $\phantom{{-}}\epsilon^{9}$ & $\phantom{{-}}\epsilon^{11}$ \\ \hline\rule{0pt}{\dimexpr\ht\strutbox+0.8mm\relax}
 $\sigma^{5}$ & $|\sin(\psi)|^{1}$  & $\phantom{-}\frac{5}{2}$  & $-\frac{3}{2}$  & $\phantom{-}\times$  & $\phantom{-}\times$  & $\phantom{-}\times$  \\[0.3mm]
 & $|\sin(\psi)|^{3}$  & $-5$  & $\phantom{-}5$  & $\phantom{-}\times$  & $\phantom{-}\times$  & $\phantom{-}\times$  \\[0.3mm]
 & $|\sin(\psi)|^{5}$  & $\phantom{-}\frac{5}{2}$  & $-\frac{7}{2}$  & $\phantom{-}\times$  & $\phantom{-}\times$  & $\phantom{-}\times$  \\[0.4mm]
\noalign{\vskip 1.5mm} & & $\phantom{{-}}\epsilon^{2}$ & $\phantom{{-}}\epsilon^{4}$ & $\phantom{{-}}\epsilon^{6}$ & $\phantom{{-}}\epsilon^{8}$ & $\phantom{{-}}\epsilon^{10}$ \\ \hline\rule{0pt}{\dimexpr\ht\strutbox+0.8mm\relax}
 $\sigma^{6}$ & 1  & $-\frac{1}{2}$  & $\phantom{-}\frac{3}{4}$  & $-\frac{5}{16}$  & $\phantom{-}\times$  & $\phantom{-}\times$  \\[0.3mm]
 & $|\sin(\psi)|^{2}$  & $\phantom{-}\frac{3}{2}$  & $-\frac{39}{4}$  & $\phantom{-}\frac{75}{16}$  & $\phantom{-}\times$  & $\phantom{-}\times$  \\[0.3mm]
 & $|\sin(\psi)|^{4}$  & $-\frac{3}{2}$  & $\phantom{-}\frac{69}{4}$  & $-\frac{175}{16}$  & $\phantom{-}\times$  & $\phantom{-}\times$  \\[0.3mm]
 & $|\sin(\psi)|^{6}$  & $\phantom{-}\frac{1}{2}$  & $-\frac{33}{4}$  & $\phantom{-}\frac{105}{16}$  & $\phantom{-}\times$  & $\phantom{-}\times$  \\[0.4mm]
\noalign{\vskip 1.5mm} & & $\phantom{{-}}\epsilon^{3}$ & $\phantom{{-}}\epsilon^{5}$ & $\phantom{{-}}\epsilon^{7}$ & $\phantom{{-}}\epsilon^{9}$ & $\phantom{{-}}\epsilon^{11}$ \\ \hline\rule{0pt}{\dimexpr\ht\strutbox+0.8mm\relax}
 $\sigma^{7}$ & $|\sin(\psi)|^{1}$  & $\phantom{-}\frac{35}{8}$  & $-\frac{21}{4}$  & $\phantom{-}\frac{15}{8}$  & $\phantom{-}\times$  & $\phantom{-}\times$  \\[0.3mm]
 & $|\sin(\psi)|^{3}$  & $-\frac{105}{8}$  & $\phantom{-}\frac{133}{4}$  & $-\frac{105}{8}$  & $\phantom{-}\times$  & $\phantom{-}\times$  \\[0.3mm]
 & $|\sin(\psi)|^{5}$  & $\phantom{-}\frac{105}{8}$  & $-\frac{203}{4}$  & $\phantom{-}\frac{189}{8}$  & $\phantom{-}\times$  & $\phantom{-}\times$  \\[0.3mm]
 & $|\sin(\psi)|^{7}$  & $-\frac{35}{8}$  & $\phantom{-}\frac{91}{4}$  & $-\frac{99}{8}$  & $\phantom{-}\times$  & $\phantom{-}\times$  \\[0.4mm]
\noalign{\vskip 1.5mm} & & $\phantom{{-}}\epsilon^{2}$ & $\phantom{{-}}\epsilon^{4}$ & $\phantom{{-}}\epsilon^{6}$ & $\phantom{{-}}\epsilon^{8}$ & $\phantom{{-}}\epsilon^{10}$ \\ \hline\rule{0pt}{\dimexpr\ht\strutbox+0.8mm\relax}
 $\sigma^{8}$ & 1  & $-\frac{1}{2}$  & $\phantom{-}\frac{9}{8}$  & $-\frac{15}{16}$  & $\phantom{-}\frac{35}{128}$  & $\phantom{-}\times$  \\[0.3mm]
 & $|\sin(\psi)|^{2}$  & $\phantom{-}2$  & $-\frac{51}{2}$  & $\phantom{-}\frac{99}{4}$  & $-\frac{245}{32}$  & $\phantom{-}\times$  \\[0.3mm]
 & $|\sin(\psi)|^{4}$  & $-3$  & $\phantom{-}\frac{279}{4}$  & $-\frac{829}{8}$  & $\phantom{-}\frac{2205}{64}$  & $\phantom{-}\times$  \\[0.3mm]
 & $|\sin(\psi)|^{6}$  & $\phantom{-}2$  & $-\frac{135}{2}$  & $\phantom{-}\frac{547}{4}$  & $-\frac{1617}{32}$  & $\phantom{-}\times$  \\[0.3mm]
 & $|\sin(\psi)|^{8}$  & $-\frac{1}{2}$  & $\phantom{-}\frac{177}{8}$  & $-\frac{911}{16}$  & $\phantom{-}\frac{3003}{128}$  & $\phantom{-}\times$  \\[0.4mm]
\noalign{\vskip 1.5mm} & & $\phantom{{-}}\epsilon^{3}$ & $\phantom{{-}}\epsilon^{5}$ & $\phantom{{-}}\epsilon^{7}$ & $\phantom{{-}}\epsilon^{9}$ & $\phantom{{-}}\epsilon^{11}$ \\ \hline\rule{0pt}{\dimexpr\ht\strutbox+0.8mm\relax}
 $\sigma^{9}$ & $|\sin(\psi)|^{1}$  & $\phantom{-}\frac{105}{16}$  & $-\frac{189}{16}$  & $\phantom{-}\frac{135}{16}$  & $-\frac{35}{16}$  & $\phantom{-}\times$  \\[0.3mm]
 & $|\sin(\psi)|^{3}$  & $-\frac{105}{4}$  & $\phantom{-}\frac{483}{4}$  & $-\frac{387}{4}$  & $\phantom{-}\frac{105}{4}$  & $\phantom{-}\times$  \\[0.3mm]
 & $|\sin(\psi)|^{5}$  & $\phantom{-}\frac{315}{8}$  & $-\frac{2331}{8}$  & $\phantom{-}\frac{2421}{8}$  & $-\frac{693}{8}$  & $\phantom{-}\times$  \\[0.3mm]
 & $|\sin(\psi)|^{7}$  & $-\frac{105}{4}$  & $\phantom{-}\frac{1071}{4}$  & $-\frac{1395}{4}$  & $\phantom{-}\frac{429}{4}$  & $\phantom{-}\times$  \\[0.3mm]
 & $|\sin(\psi)|^{9}$  & $\phantom{-}\frac{105}{16}$  & $-\frac{1365}{16}$  & $\phantom{-}\frac{2151}{16}$  & $-\frac{715}{16}$  & $\phantom{-}\times$  \\[0.4mm]
\noalign{\vskip 1.5mm} & & $\phantom{{-}}\epsilon^{2}$ & $\phantom{{-}}\epsilon^{4}$ & $\phantom{{-}}\epsilon^{6}$ & $\phantom{{-}}\epsilon^{8}$ & $\phantom{{-}}\epsilon^{10}$ \\ \hline\rule{0pt}{\dimexpr\ht\strutbox+0.8mm\relax}
 $\sigma^{10}$ & 1  & $-\frac{1}{2}$  & $\phantom{-}\frac{3}{2}$  & $-\frac{15}{8}$  & $\phantom{-}\frac{35}{32}$  & $-\frac{63}{256}$  \\[0.3mm]
 & $|\sin(\psi)|^{2}$  & $\phantom{-}\frac{5}{2}$  & $-\frac{105}{2}$  & $\phantom{-}\frac{615}{8}$  & $-\frac{1525}{32}$  & $\phantom{-}\frac{2835}{256}$  \\[0.3mm]
 & $|\sin(\psi)|^{4}$  & $-5$  & $\phantom{-}195$  & $-\frac{1995}{4}$  & $\phantom{-}\frac{5395}{16}$  & $-\frac{10395}{128}$  \\[0.3mm]
 & $|\sin(\psi)|^{6}$  & $\phantom{-}5$  & $-285$  & $\phantom{-}\frac{4215}{4}$  & $-\frac{13465}{16}$  & $\phantom{-}\frac{27027}{128}$  \\[0.3mm]
 & $|\sin(\psi)|^{8}$  & $-\frac{5}{2}$  & $\phantom{-}\frac{375}{2}$  & $-\frac{7275}{8}$  & $\phantom{-}\frac{27415}{32}$  & $-\frac{57915}{256}$  \\[0.3mm]
 & $|\sin(\psi)|^{10}$  & $\phantom{-}\frac{1}{2}$  & $-\frac{93}{2}$  & $\phantom{-}\frac{2235}{8}$  & $-\frac{9785}{32}$  & $\phantom{-}\frac{21879}{256}$  \\[0.4mm]
\noalign{\vskip 1.5mm} & & $\phantom{{-}}\epsilon^{3}$ & $\phantom{{-}}\epsilon^{5}$ & $\phantom{{-}}\epsilon^{7}$ & $\phantom{{-}}\epsilon^{9}$ & $\phantom{{-}}\epsilon^{11}$ \\ \hline\rule{0pt}{\dimexpr\ht\strutbox+0.8mm\relax}
 $\sigma^{11}$ & $|\sin(\psi)|^{1}$  & $\phantom{-}\frac{1155}{128}$  & $-\frac{693}{32}$  & $\phantom{-}\frac{1485}{64}$  & $-\frac{385}{32}$  & $\phantom{-}\frac{315}{128}$  \\[0.3mm]
 & $|\sin(\psi)|^{3}$  & $-\frac{5775}{128}$  & $\phantom{-}\frac{10395}{32}$  & $-\frac{25245}{64}$  & $\phantom{-}\frac{6875}{32}$  & $-\frac{5775}{128}$  \\[0.3mm]
 & $|\sin(\psi)|^{5}$  & $\phantom{-}\frac{5775}{64}$  & $-\frac{17325}{16}$  & $\phantom{-}\frac{59697}{32}$  & $-\frac{17369}{16}$  & $\phantom{-}\frac{15015}{64}$  \\[0.3mm]
 & $|\sin(\psi)|^{7}$  & $-\frac{5775}{64}$  & $\phantom{-}\frac{24255}{16}$  & $-\frac{110781}{32}$  & $\phantom{-}\frac{36047}{16}$  & $-\frac{32175}{64}$  \\[0.3mm]
 & $|\sin(\psi)|^{9}$  & $\phantom{-}\frac{5775}{128}$  & $-\frac{31185}{32}$  & $\phantom{-}\frac{178497}{64}$  & $-\frac{65549}{32}$  & $\phantom{-}\frac{60775}{128}$  \\[0.3mm]
 & $|\sin(\psi)|^{11}$  & $-\frac{1155}{128}$  & $\phantom{-}\frac{7623}{32}$  & $-\frac{52569}{64}$  & $\phantom{-}\frac{21703}{32}$  & $-\frac{20995}{128}$  \\[0.4mm]
\end{tabular}
\captionsetup{width=0.87\linewidth}
\caption{Tabulated coefficient values for series expansion of $(\sin(\phi)-\sgn(\psi))/\sgn(\psi)$ as in Theorem 1. Top rows and left-most columns show series and polynomial variables and the table body contains the corresponding values of $d^{\sin}_{k,l,n}$ for indicated indices. Values marked with $\times$ are not part of the expansion.}
\label{tab:d_sin_phi_evo}
\end{table}

\begin{table}
\scriptsize
\begin{tabular}{@{\hspace{0mm}} l@{\hspace{0mm}} c@{\hspace{0.5mm}} | @{\hspace{0.5mm}} l@{}  l@{}  l@{}  l@{}  l@{}  l@{} }
$k$ & $l\setminus n$ & $\phantom{{-}}\epsilon^{1}$ & $\phantom{{-}}\epsilon^{3}$ & $\phantom{{-}}\epsilon^{5}$ & $\phantom{{-}}\epsilon^{7}$ & $\phantom{{-}}\epsilon^{9}$ & $\phantom{{-}}\epsilon^{11}$ \\ \hline\rule{0pt}{\dimexpr\ht\strutbox+0.8mm\relax}
 $\sigma^{1}$ & 1  & $\phantom{-}1$  & $\phantom{-}\times$  & $\phantom{-}\times$  & $\phantom{-}\times$  & $\phantom{-}\times$  & $\phantom{-}\times$  \\[0.4mm]
\noalign{\vskip 1.5mm} & & $\phantom{{-}}{1}$ & $\phantom{{-}}\epsilon^{2}$ & $\phantom{{-}}\epsilon^{4}$ & $\phantom{{-}}\epsilon^{6}$ & $\phantom{{-}}\epsilon^{8}$ & $\phantom{{-}}\epsilon^{10}$ \\ \hline\rule{0pt}{\dimexpr\ht\strutbox+0.8mm\relax}
 $\sigma^{2}$ & $|\sin(\psi)|^{1}$  & $\phantom{-}\times$  & $-1$  & $\phantom{-}\times$  & $\phantom{-}\times$  & $\phantom{-}\times$  & $\phantom{-}\times$  \\[0.4mm]
\noalign{\vskip 1.5mm} & & $\phantom{{-}}\epsilon^{1}$ & $\phantom{{-}}\epsilon^{3}$ & $\phantom{{-}}\epsilon^{5}$ & $\phantom{{-}}\epsilon^{7}$ & $\phantom{{-}}\epsilon^{9}$ & $\phantom{{-}}\epsilon^{11}$ \\ \hline\rule{0pt}{\dimexpr\ht\strutbox+0.8mm\relax}
 $\sigma^{3}$ & 1  & $\phantom{-}\frac{1}{2}$  & $-\frac{1}{2}$  & $\phantom{-}\times$  & $\phantom{-}\times$  & $\phantom{-}\times$  & $\phantom{-}\times$  \\[0.3mm]
 & $|\sin(\psi)|^{2}$  & $-\frac{1}{2}$  & $\phantom{-}\frac{3}{2}$  & $\phantom{-}\times$  & $\phantom{-}\times$  & $\phantom{-}\times$  & $\phantom{-}\times$  \\[0.4mm]
\noalign{\vskip 1.5mm} & & $\phantom{{-}}{1}$ & $\phantom{{-}}\epsilon^{2}$ & $\phantom{{-}}\epsilon^{4}$ & $\phantom{{-}}\epsilon^{6}$ & $\phantom{{-}}\epsilon^{8}$ & $\phantom{{-}}\epsilon^{10}$ \\ \hline\rule{0pt}{\dimexpr\ht\strutbox+0.8mm\relax}
 $\sigma^{4}$ & $|\sin(\psi)|^{1}$  & $\phantom{-}\times$  & $-2$  & $\phantom{-}\frac{3}{2}$  & $\phantom{-}\times$  & $\phantom{-}\times$  & $\phantom{-}\times$  \\[0.3mm]
 & $|\sin(\psi)|^{3}$  & $\phantom{-}\times$  & $\phantom{-}2$  & $-\frac{5}{2}$  & $\phantom{-}\times$  & $\phantom{-}\times$  & $\phantom{-}\times$  \\[0.4mm]
\noalign{\vskip 1.5mm} & & $\phantom{{-}}\epsilon^{1}$ & $\phantom{{-}}\epsilon^{3}$ & $\phantom{{-}}\epsilon^{5}$ & $\phantom{{-}}\epsilon^{7}$ & $\phantom{{-}}\epsilon^{9}$ & $\phantom{{-}}\epsilon^{11}$ \\ \hline\rule{0pt}{\dimexpr\ht\strutbox+0.8mm\relax}
 $\sigma^{5}$ & 1  & $\phantom{-}\frac{3}{8}$  & $-\frac{3}{4}$  & $\phantom{-}\frac{3}{8}$  & $\phantom{-}\times$  & $\phantom{-}\times$  & $\phantom{-}\times$  \\[0.3mm]
 & $|\sin(\psi)|^{2}$  & $-\frac{3}{4}$  & $\phantom{-}\frac{13}{2}$  & $-\frac{15}{4}$  & $\phantom{-}\times$  & $\phantom{-}\times$  & $\phantom{-}\times$  \\[0.3mm]
 & $|\sin(\psi)|^{4}$  & $\phantom{-}\frac{3}{8}$  & $-\frac{23}{4}$  & $\phantom{-}\frac{35}{8}$  & $\phantom{-}\times$  & $\phantom{-}\times$  & $\phantom{-}\times$  \\[0.4mm]
\noalign{\vskip 1.5mm} & & $\phantom{{-}}{1}$ & $\phantom{{-}}\epsilon^{2}$ & $\phantom{{-}}\epsilon^{4}$ & $\phantom{{-}}\epsilon^{6}$ & $\phantom{{-}}\epsilon^{8}$ & $\phantom{{-}}\epsilon^{10}$ \\ \hline\rule{0pt}{\dimexpr\ht\strutbox+0.8mm\relax}
 $\sigma^{6}$ & $|\sin(\psi)|^{1}$  & $\phantom{-}\times$  & $-3$  & $\phantom{-}\frac{9}{2}$  & $-\frac{15}{8}$  & $\phantom{-}\times$  & $\phantom{-}\times$  \\[0.3mm]
 & $|\sin(\psi)|^{3}$  & $\phantom{-}\times$  & $\phantom{-}6$  & $-19$  & $\phantom{-}\frac{35}{4}$  & $\phantom{-}\times$  & $\phantom{-}\times$  \\[0.3mm]
 & $|\sin(\psi)|^{5}$  & $\phantom{-}\times$  & $-3$  & $\phantom{-}\frac{29}{2}$  & $-\frac{63}{8}$  & $\phantom{-}\times$  & $\phantom{-}\times$  \\[0.4mm]
\noalign{\vskip 1.5mm} & & $\phantom{{-}}\epsilon^{1}$ & $\phantom{{-}}\epsilon^{3}$ & $\phantom{{-}}\epsilon^{5}$ & $\phantom{{-}}\epsilon^{7}$ & $\phantom{{-}}\epsilon^{9}$ & $\phantom{{-}}\epsilon^{11}$ \\ \hline\rule{0pt}{\dimexpr\ht\strutbox+0.8mm\relax}
 $\sigma^{7}$ & 1  & $\phantom{-}\frac{5}{16}$  & $-\frac{15}{16}$  & $\phantom{-}\frac{15}{16}$  & $-\frac{5}{16}$  & $\phantom{-}\times$  & $\phantom{-}\times$  \\[0.3mm]
 & $|\sin(\psi)|^{2}$  & $-\frac{15}{16}$  & $\phantom{-}\frac{255}{16}$  & $-\frac{297}{16}$  & $\phantom{-}\frac{105}{16}$  & $\phantom{-}\times$  & $\phantom{-}\times$  \\[0.3mm]
 & $|\sin(\psi)|^{4}$  & $\phantom{-}\frac{15}{16}$  & $-\frac{465}{16}$  & $\phantom{-}\frac{829}{16}$  & $-\frac{315}{16}$  & $\phantom{-}\times$  & $\phantom{-}\times$  \\[0.3mm]
 & $|\sin(\psi)|^{6}$  & $-\frac{5}{16}$  & $\phantom{-}\frac{225}{16}$  & $-\frac{547}{16}$  & $\phantom{-}\frac{231}{16}$  & $\phantom{-}\times$  & $\phantom{-}\times$  \\[0.4mm]
\noalign{\vskip 1.5mm} & & $\phantom{{-}}{1}$ & $\phantom{{-}}\epsilon^{2}$ & $\phantom{{-}}\epsilon^{4}$ & $\phantom{{-}}\epsilon^{6}$ & $\phantom{{-}}\epsilon^{8}$ & $\phantom{{-}}\epsilon^{10}$ \\ \hline\rule{0pt}{\dimexpr\ht\strutbox+0.8mm\relax}
 $\sigma^{8}$ & $|\sin(\psi)|^{1}$  & $\phantom{-}\times$  & $-4$  & $\phantom{-}9$  & $-\frac{15}{2}$  & $\phantom{-}\frac{35}{16}$  & $\phantom{-}\times$  \\[0.3mm]
 & $|\sin(\psi)|^{3}$  & $\phantom{-}\times$  & $\phantom{-}12$  & $-69$  & $\phantom{-}\frac{129}{2}$  & $-\frac{315}{16}$  & $\phantom{-}\times$  \\[0.3mm]
 & $|\sin(\psi)|^{5}$  & $\phantom{-}\times$  & $-12$  & $\phantom{-}111$  & $-\frac{269}{2}$  & $\phantom{-}\frac{693}{16}$  & $\phantom{-}\times$  \\[0.3mm]
 & $|\sin(\psi)|^{7}$  & $\phantom{-}\times$  & $\phantom{-}4$  & $-51$  & $\phantom{-}\frac{155}{2}$  & $-\frac{429}{16}$  & $\phantom{-}\times$  \\[0.4mm]
\noalign{\vskip 1.5mm} & & $\phantom{{-}}\epsilon^{1}$ & $\phantom{{-}}\epsilon^{3}$ & $\phantom{{-}}\epsilon^{5}$ & $\phantom{{-}}\epsilon^{7}$ & $\phantom{{-}}\epsilon^{9}$ & $\phantom{{-}}\epsilon^{11}$ \\ \hline\rule{0pt}{\dimexpr\ht\strutbox+0.8mm\relax}
 $\sigma^{9}$ & 1  & $\phantom{-}\frac{35}{128}$  & $-\frac{35}{32}$  & $\phantom{-}\frac{105}{64}$  & $-\frac{35}{32}$  & $\phantom{-}\frac{35}{128}$  & $\phantom{-}\times$  \\[0.3mm]
 & $|\sin(\psi)|^{2}$  & $-\frac{35}{32}$  & $\phantom{-}\frac{245}{8}$  & $-\frac{861}{16}$  & $\phantom{-}\frac{305}{8}$  & $-\frac{315}{32}$  & $\phantom{-}\times$  \\[0.3mm]
 & $|\sin(\psi)|^{4}$  & $\phantom{-}\frac{105}{64}$  & $-\frac{1365}{16}$  & $\phantom{-}\frac{8379}{32}$  & $-\frac{3237}{16}$  & $\phantom{-}\frac{3465}{64}$  & $\phantom{-}\times$  \\[0.3mm]
 & $|\sin(\psi)|^{6}$  & $-\frac{35}{32}$  & $\phantom{-}\frac{665}{8}$  & $-\frac{5901}{16}$  & $\phantom{-}\frac{2693}{8}$  & $-\frac{3003}{32}$  & $\phantom{-}\times$  \\[0.3mm]
 & $|\sin(\psi)|^{8}$  & $\phantom{-}\frac{35}{128}$  & $-\frac{875}{32}$  & $\phantom{-}\frac{10185}{64}$  & $-\frac{5483}{32}$  & $\phantom{-}\frac{6435}{128}$  & $\phantom{-}\times$  \\[0.4mm]
\noalign{\vskip 1.5mm} & & $\phantom{{-}}{1}$ & $\phantom{{-}}\epsilon^{2}$ & $\phantom{{-}}\epsilon^{4}$ & $\phantom{{-}}\epsilon^{6}$ & $\phantom{{-}}\epsilon^{8}$ & $\phantom{{-}}\epsilon^{10}$ \\ \hline\rule{0pt}{\dimexpr\ht\strutbox+0.8mm\relax}
 $\sigma^{10}$ & $|\sin(\psi)|^{1}$  & $\phantom{-}\times$  & $-5$  & $\phantom{-}15$  & $-\frac{75}{4}$  & $\phantom{-}\frac{175}{16}$  & $-\frac{315}{128}$  \\[0.3mm]
 & $|\sin(\psi)|^{3}$  & $\phantom{-}\times$  & $\phantom{-}20$  & $-180$  & $\phantom{-}255$  & $-\frac{625}{4}$  & $\phantom{-}\frac{1155}{32}$  \\[0.3mm]
 & $|\sin(\psi)|^{5}$  & $\phantom{-}\times$  & $-30$  & $\phantom{-}450$  & $-\frac{1809}{2}$  & $\phantom{-}\frac{4737}{8}$  & $-\frac{9009}{64}$  \\[0.3mm]
 & $|\sin(\psi)|^{7}$  & $\phantom{-}\times$  & $\phantom{-}20$  & $-420$  & $\phantom{-}1119$  & $-\frac{3277}{4}$  & $\phantom{-}\frac{6435}{32}$  \\[0.3mm]
 & $|\sin(\psi)|^{9}$  & $\phantom{-}\times$  & $-5$  & $\phantom{-}135$  & $-\frac{1803}{4}$  & $\phantom{-}\frac{5959}{16}$  & $-\frac{12155}{128}$  \\[0.4mm]
\noalign{\vskip 1.5mm} & & $\phantom{{-}}\epsilon^{1}$ & $\phantom{{-}}\epsilon^{3}$ & $\phantom{{-}}\epsilon^{5}$ & $\phantom{{-}}\epsilon^{7}$ & $\phantom{{-}}\epsilon^{9}$ & $\phantom{{-}}\epsilon^{11}$ \\ \hline\rule{0pt}{\dimexpr\ht\strutbox+0.8mm\relax}
 $\sigma^{11}$ & 1  & $\phantom{-}\frac{63}{256}$  & $-\frac{315}{256}$  & $\phantom{-}\frac{315}{128}$  & $-\frac{315}{128}$  & $\phantom{-}\frac{315}{256}$  & $-\frac{63}{256}$  \\[0.3mm]
 & $|\sin(\psi)|^{2}$  & $-\frac{315}{256}$  & $\phantom{-}\frac{13125}{256}$  & $-\frac{15435}{128}$  & $\phantom{-}\frac{16425}{128}$  & $-\frac{16975}{256}$  & $\phantom{-}\frac{3465}{256}$  \\[0.3mm]
 & $|\sin(\psi)|^{4}$  & $\phantom{-}\frac{315}{128}$  & $-\frac{24675}{128}$  & $\phantom{-}\frac{57015}{64}$  & $-\frac{66915}{64}$  & $\phantom{-}\frac{71975}{128}$  & $-\frac{15015}{128}$  \\[0.3mm]
 & $|\sin(\psi)|^{6}$  & $-\frac{315}{128}$  & $\phantom{-}\frac{36225}{128}$  & $-\frac{126315}{64}$  & $\phantom{-}\frac{186309}{64}$  & $-\frac{210927}{128}$  & $\phantom{-}\frac{45045}{128}$  \\[0.3mm]
 & $|\sin(\psi)|^{8}$  & $\phantom{-}\frac{315}{256}$  & $-\frac{47775}{256}$  & $\phantom{-}\frac{223335}{128}$  & $-\frac{407871}{128}$  & $\phantom{-}\frac{499303}{256}$  & $-\frac{109395}{256}$  \\[0.3mm]
 & $|\sin(\psi)|^{10}$  & $-\frac{63}{256}$  & $\phantom{-}\frac{11865}{256}$  & $-\frac{69615}{128}$  & $\phantom{-}\frac{152973}{128}$  & $-\frac{204739}{256}$  & $\phantom{-}\frac{46189}{256}$  \\[0.4mm]
\end{tabular}
\captionsetup{width=0.87\linewidth}
\caption{Tabulated coefficient values for series expansion of $\cos(\phi)/\cos(\psi)$ as in Theorem 1. Top rows and left-most columns show series and polynomial variables and the table body contains the corresponding values of $d^{\cos}_{k,l,n}$ for indicated indices. Values marked with $\times$ are not part of the expansion.}
\label{tab:d_cos_phi_evo}
\end{table}

\end{document}